\documentclass[11pt]{amsart}
\usepackage{xcolor}
\usepackage{amscd}
\usepackage{url}
\newtheorem{theorem}{Theorem}[section]
\newtheorem{lemma}[theorem]{Lemma}
\newtheorem{corollary}[theorem]{Corollary}
\newtheorem{proposition}[theorem]{Proposition}

\theoremstyle{definition}
\newtheorem{definition}[theorem]{Definition}
\newtheorem{notation}[theorem]{Notation}
\newtheorem{example}[theorem]{Example}

\theoremstyle{remark}
\newtheorem{remark}[theorem]{Remark}

\numberwithin{equation}{section}

\newcommand*{\into}{\hookrightarrow}
\newcommand*{\onto}{\twoheadrightarrow}
\newcommand{\env}{\operatorname{env}}

\newcommand{\eps}{\varepsilon}

\newcommand{\C}{\mathbb{C}}

\newcommand{\N}{\mathbb{N}}

\newcommand{\Aut}{\operatorname{Aut}}
\newcommand{\Ad}{\operatorname{Ad}}
\newcommand{\id}{\operatorname{id}}
\newcommand{\Ind}{\operatorname{Ind}}

\newcommand{\op}{\operatorname{op}}
\newcommand{\one}{\boldsymbol{1}}

\newcommand{\cspn}{\overline{\operatorname{span}}}

\newcommand{\Op}{{\operatorname{Op}}}
\newcommand{\diag}{{\operatorname{diag}}}

\newcommand{\B}{{\mathcal B}}
\newcommand{\K}{{\mathcal K}}
\usepackage{amsfonts}
\usepackage{amsfonts}
\usepackage[centertags]{amsmath}
\usepackage{amssymb}
\usepackage{amsthm}
\usepackage{newlfont}

\usepackage{color}

\begin{document}

\title{$C^*$-operator systems and crossed products}


\author{Massoud Amini}
\address{Department of Mathematics, Faculty of Mathematical Sciences, Tarbiat Modares University,
Tehran 14115-134, Iran}
\email{mamini@modares.ac.ir}

\author{Siegfried Echterhoff}
\address{Mathematisches Institut\\
 Westf\"alische Wilhelms-Universit\"at M\"un\-ster\\
 Einsteinstr.\ 62\\
 48149 M\"unster\\
 Germany}
\email{echters@uni-muenster.de}

\author{Hamed Nikpey}
\address{Department of Basic Sciences\\ Shahid Rajaee Teacher Training University\\
P. O. Box 16783-163, Tehran\\ Iran}
 \email{hamednikpey@gmail.com}
\thanks{Funded by the Deutsche Forschungsgemeinschaft (DFG, German Research Foundation) under Germany's Excellence Strategy EXC 2044 –390685587, Mathematics Münster: Dynamics–Geometry–Structure,  and SFB 878, Groups, Geometry \& Actions. Also funded by the German Academic Exchange Service DAAD}
\date{}

\dedicatory{}

\commby{}

\begin{abstract} 
The purpose of this paper is to introduce  a consistent notion of universal and reduced 
crossed products by actions and coactions of  groups on operator systems and operator spaces.
In particular we shall put emphasis to reveal the full power of the universal properties of  the 
the universal crossed products. It turns out that to make things  consistent, it seems 
useful to perform our constructions on some bigger categories which allow the right 
framework for studying the universal properties and which are stable under the construction 
of crossed products even for non-discrete groups. In the case of operator systems, this 
larger category is what we call a $C^*$-operator system, i.e., a selfadjoint subspace $X$
of some $\mathcal B(H)$ which contains a $C^*$-algebra $A$ such that 
$AX=X=XA$. In the case of operator spaces, the larger category is given by what 
we call $C^*$-operator bimodules. After we introduced the respective crossed products 
we show that the classical Imai-Takai and Katayama duality theorems for crossed products by group (co-)actions 
on $C^*$-algebras extend one-to-one to our notion of crossed products by
$C^*$-operator systems and $C^*$-operator bimodules.

\end{abstract}

\maketitle

\section{Introduction}
In the world of $C^*$-algebras, the construction of crossed-products $A\rtimes_\alpha G$  for an action of a  locally compact group $G$ 
an a $C^*$-algebra $A$ is one of the most fundamental tools in the theory--not only to construct interesting examples of $C^*$-algebras, but also in the application of $C^*$-algebra theory in  Harmonic Analysis, Non-commutative Geometry, Topology, and other areas of mathematics.
Having this in mind it is very surprising that a serious study of a similar construction did not appear in the world of non-selfadoint operator algebras, 
operator systems, or operator spaces until the recent works of Katsoulis and Ramsay \cite{KR} in the setting of operator algebras 
and the even more recent work \cite{HK} of Harris and Kim in which they give a construction of crossed products of operator systems by actions of {\em discrete} groups, following some of the ideas developed in \cite{KR} for the construction. But we should also mention the earlier preprint
\cite{Ng} by Chi-Keung Ng, where he introduces reduced  (i.e., spatially defined) crossed products for coactions of quantum groups on operator spaces.

Just a few days before  the first version of \cite{KR}
appeared on the arXiv, the authors of this paper posted a preprint describing a crossed product construction for group actions on operator spaces (see \cite{AEN}). Although this paper contained many ideas which were quite similar to ideas used in \cite{KR}, the proof of a central theorem 
(\cite[Theorem 4.3]{AEN}) turned out to be wrong, and since we didn't see a way for a quick repair, we decided to withdraw the paper from the arXiv. 
We are very grateful to Elias Katsoulis for  having pointed out this error to us! 
The problem was that in \cite{AEN} we defined the full and reduced crossed products $V\rtimes_\alpha^uG$ and $V\rtimes_\alpha^rG$ as 
the completions of $C_c(G,V)$ inside the full and reduced crossed products by a canonical action of $G$ on the enveloping $C^*$-algebra 
$C_e^*(X(V))$, where $X(V)$ denotes the Paulsen system of $V$. Indeed, there is no problem with this 
in case of the reduced crossed products, but the universal crossed product $V\rtimes_\alpha^uG$ should enjoy a universal property for 
suitable covariant representations of the system $(V,G,\alpha)$ (which was the content of the unfortunate \cite[Theorem 4.3]{AEN}).
However, \cite[Theorem 5.6]{HK} indicates that this cannot be true in general.  The way out is to define the universal crossed product 
$V\rtimes_\alpha^uG$ as the closure of $C_c(G,V)$ inside the universal crossed product $C_u^*(X(V))\rtimes_{\alpha,u}G$ where
$C_u^*(X(V))$ denotes the universal $C^*$-hull of $X(V)$ as introduced by Kirchberg and Wassermann in \cite{KW}. 
This was the approach of  \cite{KR} in case of  operator algebras and of \cite{HK} in case of operator systems. 

If we want to exploit the full power of universal properties for the universal crossed products, however, we would want to 
have a  one-to-one correspondence between completely bounded covariant maps $(\varphi, u)$ of the system $(V,G,\alpha)$ and 
the completely bounded maps $\Phi$ of $V\rtimes_\alpha^uG$ via a canonically defined  integrated form $\varphi\rtimes u$.
But it turns out  that  in order to obtain such a correspondence we need to remember more information of the 
ambient crossed product $C_u^*(X(V))\rtimes_\alpha^uG$. Indeed, taking the completion of 
$$C_c(G, X(V))=\left(\begin{matrix} C_c(G) & C_c(G, V)\\ C_c(G, V^*)& C_c(G)\end{matrix}\right)\subseteq C_u^*(X(V))\rtimes_{\alpha,u}G$$
and considering the convolution products between the upper diagonal entries gives $V\rtimes_\alpha^uG$ the structure of an
operator $C_u^*(G)$-bimodule, and it is this structure which one needs to take into account for a good  description of the universal properties.

Other problems appear if we consider crossed products by operator systems instead of operator spaces. Since in the above procedure we 
defined crossed products by $V$ via a crossed-product construction with the Paulsen system $X(V)$, it appears to be useful to 
consider  at first the case of crossed products by general operator systems $X$. As mentioned above, such crossed products have been introduced by Harris and Kim  in \cite{HK} for discrete groups. The reason for the restriction to the discrete case was the simple fact that the analogous construction 
for  non-discrete groups $G$ would result in a non-unital  (but selfadjoint) operator  space, hence it does not land in the right category. 
The other draw back is similar as the one described above for operator spaces crossed-products: we need to  keep more structure 
than simply the completion of $C_c(G,X)$ inside $C_u^*(X)\rtimes_uG$ in order to get the full power of the universal properties.

Our way out is to extend the category of operator systems to what we call $C^*$-operator systems: a concrete $C^*$-operator system
$(A,X)$ is a pair of subsets $A\subseteq X\subseteq \mathcal B(H)$ for some Hilbert space $H$, such that $X=X^*$, 
$A$ is a non-degenerate $C^*$-subalgebra of $\mathcal B(H)$, and $AX=X=XA$. A morphism  from $(A,X)$ to the $C^*$-operator system $(B,Y)$
is then a ccp map $\varphi_X: X\to Y$ such that the restriction $\varphi_A:=\varphi_X|_A$ is a $*$-homomorphisms from  $A$ to $B$
and such that $\varphi_X(ax)=\varphi_A(a)\varphi_X(x)$ and $\varphi_X(xa)=\varphi_X(x)\varphi_A(a)$ for all $a\in A$ and $x\in X$.
Of course, if $X\subseteq \mathcal B(H)$ is a classical operator system, 
then $(\C1, X)$ is $C^*$-operator system in this sense, and every ucp map between operator systems $X$ and $Y$ extends 
to a morphism in the above sense from $(\C1, X)$  to $(\C1, Y)$. Thus we get an inclusion of the category of operator systems into the
category of $C^*$-operator systems. After some preliminaries given in Section \ref{sec-prel} we introduce $C^*$-operator systems in Section 
\ref{sec-C*opsys}, where we also introduce a corresponding notion of multiplier $C^*$-operator systems which play an analogous role as the
multiplier algebra for a $C^*$-algebra. In particular, for a $C^*$-operator system $(A,X)$ the multiplier system $(M(A), M(X))$ 
can be considered as the largest unitization of $(A,X)$ and it always contains the unitization $(\tilde{A},\tilde{X})$ 
if $(A,X)$ has not unit (which means that $A$ has no unit). An important feature of the multiplier system is that every 
 non-degennerate morphism from $(A,X) $ to $(B,Y)$ extends uniquely to a morphism from $(MA), M(X))$ to $(M(B), M(Y))$. 
 
In Section \ref{sec-univ-env} we study $C^*$-hulls of $C^*$-operator systems, i.e., $C^*$-algebras $C$ together with
completely isometric representations $(j_A,j_X):(A,X)\to C$ such that $C$ is generated by the image $j_X(X)$ of $X$.
We show that there always exists a largest (the universal) $C^*$-hull $C_u^*(A,X)$ and a smallest (the enveloping)  $C^*$-hull 
$C_e^*(A,X)$, using  well-known ideas of Kirchberg, Wassermann and Hamana and Ruan. 
As a first hint that the category of $C^*$-operator systems is useful, we give in Section \ref{sec-tensor} a brief discusion of 
some tensor product constructions with $C^*$-operator systems. In particular, the spatial tensor products $X\otimes A$ 
of an operator $X$ with a $C^*$-algebra $A$  naturally carries the structure of a $C^*$-operator system $(1_X\otimes A, X\otimes  A)$
and if $A$ is not unital, this system is not unital as well!

In Section \ref{crossedproducts} we define  the universal crossed products 
$(A,X)\rtimes_\alpha^uG$ by a continuous action $\alpha$ of a locally compact group $G$ on a $C^*$-operator system $(A,X)$ 
as the pair of completion $(A\rtimes_\alpha^uG, X\rtimes_\alpha^uG)$ of $(C_c(G,A), C_c(G,X))$ inside 
the universal crossed product $C_u^*(A,X)\rtimes_{{\alpha}, u}G$ for the canonical action of $G$ on the universal $C^*$-hull $C_u^*(A,X)$.
Note that the $C^*$-part $A\rtimes_{\alpha}^uG$ is not always isomorphic to the universal $C^*$-algebra crossed product $A\rtimes_{\alpha,u}G$
(e.g., see part (b) of Remark \ref{rem-crossed}). We show that in this setting we get a very satisfying picture of the universal property: every covariant morphism $(\varphi_X, u)$ of the system $(A,X, G,\alpha)$ ``integrates'' to a morphism $\varphi_X\rtimes u$ of $(A,X)\rtimes_{\alpha}^uG$
and every (non-degenerate) morphism of  $(A,X)\rtimes_{\alpha}^uG$ appears as such integrated form.
In Section \ref{sec-reduced-crossed} we add a brief discussion of the spatially defined reduced crossed product $(A,X)\rtimes_{\alpha}^rG$.

In Section \ref{sec-coaction} we study coactions of groups on $C^*$-operator systems and their crossed products. 
Since locally compact groups are  always co-amenable, it is not surprising that the full and reduced crossed products coincide
in the sense that the spatially defined crossed product already enjoys the universal properties for covariant representations.
In Section \ref{sec-dual} we prove versions of the Imai-Takai and Katayama duality theorems for actions and coactions 
of groups on $C^*$-operator systems: starting with an action $\alpha:G\to \Aut(A,X)$ there are canonical dual coactions  
$\widehat\alpha_u$ and $\widehat\alpha_r$ on the universal and reduced crossed products, respectively, such that we get 
canonical $\widehat{\widehat\alpha}=\alpha\otimes\Ad\rho$ equivariant isomorphisms  (where $\rho$ denotes the right regular representation of $G$)
$$\big(A\rtimes_{\alpha}^uG\rtimes_{\widehat{\alpha}}\widehat{G}, X\rtimes_{\alpha}^uG\rtimes_{\widehat{\alpha^u}}\widehat{G}\big)
\cong \big(A\otimes \K(L^2(G)), X\otimes \K(L^2(G))\big)$$
and 
$$\big(A\rtimes_{\alpha}^rG\rtimes_{\widehat{\alpha^r}}\widehat{G}, X\rtimes_{\alpha}^rG\rtimes_{\widehat{\alpha^r}}\widehat{G}\big)
\cong \big(A\otimes \K(L^2(G)), X\otimes \K(L^2(G))\big).$$
The converse direction, when starting with a coaction $\delta$, known as Katayama's theorem in case of $C^*$-algebra crossed products,
is  a  bit more  involved, and the full analogue 
of the Imai-Takai theorem  only works under some additional assumptions, like when $G$ is amenable or if everything in sight 
was defined spatially (i.e., we would consider reduced group algebras and reduced crossed products  only).
Note that in  \cite{Ng}, Chi-Keung Ng  proves duality theorems for spatially defined 
crossed products in the more general case of (co-)actions by more general quantum groups.

In Section \ref{sec-bimodules} we come back to group actions on operator spaces $V$. As indicated above, also in this case it is  useful to 
find a suitable extension of  the category of operator spaces since the natural candidate for the crossed product has a canonical structure 
of a $C^*$-algebra bimodule via a left and right action of $C_u^*(G)$ on $V\rtimes_\alpha^uG$. So we found  that  the right category would be 
the category of (concrete) $C^*$-operator bimodules $(A,V,B)$ which consist of a concrete operator space $V\subseteq \mathcal B(K,H)$  for some Hilbert  spaces $H$ and $K$ together with $C^*$-subalgebras $A\subseteq \mathcal B(H)$ and $B\subseteq \mathcal B(K)$ such that
$$AV=V=VB. $$
Again, we can identify operator  spaces $V\subseteq \mathcal B(K,H)$ with the $C^*$-operator bimodule
$(\C1_H,V,  \C1_K)$. In this way the category of $C^*$-operator bimodules extends the category of operator spaces.
If a $C^*$-operator bimodule $(A,V,B)$ is given, we get a corresponding {\em Paulsen $C^*$-operator system} $\big(A\oplus B, X(A,V,B)\big)$ with 
$$X(A,V,B)=\left(\begin{matrix} A&V\\ V^*& B\end{matrix}\right)$$
and a one-to-one correspondence between morphism of $(A,V,B)$ and morphisms of $(A\oplus B, X(A,V,B)\big)$.
Thus it is fairly straightforward to apply the above described crossed-product constructions for $C^*$-operator systems to
the Paulsen systems $\big(A\oplus B, X(A,V,B)\big)$ to obtain complete analogues of the above  described results 
in this setting. In particular we get complete analogues of the Imai-Takai and Katayama duality theorems.

The authors are grateful to Elias Katsoulis and David Blecher for valuable discussions and comments concerning this project and in particular to the content of the preprint \cite{AEN}.

\section{Preliminaries}\label{sec-prel}
If $H$ is a Hilbert space we denote by $\mathcal B(H)$ the algebra of bounded operators on $H$
equipped with the operator norm and the canonical involution. 
A concrete operator space is a closed linear subspace $X\subseteq \mathcal B(H)$ for some Hilbert space $H$.
If $X\subseteq \mathcal B(H)$ and $Y\subseteq \mathcal B(K)$ are two operator spaces, then for each $n\in \N$ 
we have the matrix operator spaces 
$M_n(X)\subseteq \mathcal B(H^n)$ and $M_n(Y)\subseteq \mathcal B(K^n)$.
If  $\varphi:X\to Y$ is a linear map define
$\varphi_n:M_n(X)\to M_n(Y)$  by    
$\varphi_n\big((x_{ij})_{1\leq i,j\leq n}\big)=(\varphi(x_{ij}))_{1\leq i,j\leq n}$.
Then 
$\varphi:X\to Y$ is called {\em completely bounded} (or a {\em cb map}), if there exist a constant $C\geq 0$ such that
$$\|\varphi_n(x)\|_{\op}\leq C\|x\|_{\op}$$ for all $n\in \N$ and all
$x\in M_n(X)$.
If $C$ can be chosen to be less or equal to one, we say that $\varphi:X\to Y$ is {\em completely contractive}
and if  $\|\varphi_n(x)\|_{\op}= \|x\|_{\op}$ for all $n\in \N$ and $x\in M_n(X)$,  we say that $\varphi$ is
completely isometric.

Suppose now that $X=X^*$ and $Y=Y^*$ are symmetric closed subspaces of $\mathcal B(H)$ and $\mathcal B(K)$, respectively, not necessarily containing
the units. We then call a linear map $\varphi:X\to Y$ a {\em ccp map}   (completely contractive and  positive)
if it satisfies the following conditions:
\begin{enumerate}
\item $\varphi:X\to Y$ is completely contractive;
\item $\varphi(x^*)=\varphi(x)^*$ for all $x\in X$;
\item $\varphi_n(x)\geq 0$ for every positive $x\in M_n(X)$;
\end{enumerate}
where positivity of an element $x\in M_n(X)$ (resp. $y\in M_n(Y)$) means that $x$ (resp. $y$) is 
a positive element in $\mathcal B(H^n)$ (resp. $\mathcal B(K^n)$). If, in addition, $\varphi$ is completely isometric, we 
call it an {\em icp map}. 

Of course it is well known that if  $X$ and $Y$ are operator systems (i.e., they contain the units $1_H$ and $1_K$, respectively), then 
every unital linear map which satisfies (3) automatically satisfies (1) and (2), hence is a ccp map. As usual, we then say that 
$\varphi:X\to Y$ is a {\em ucp map}.

\section{$C^*$-operator systems and multiplier systems}\label{sec-C*opsys}
In this section we introduce a category of  possibly non-unital operator systems which 
include  $C^*$-algebras and classical operator systems as subcategories. This category will
play an important r\^ole in our construction of crossed products.

\begin{definition}\label{def-cstaros}
A  (concrete) {\em  $C^*$-operator system} $(A,X)$ on the Hilbert space $H$ is a 
pair of norm-closed self-adjoint subspaces $A\subseteq X\subseteq \mathcal B(H)$ such that
\begin{enumerate}
\item $A$ is a non-degenerate $C^*$-subalgebra of $\mathcal B(H)$, i.e., $A H=H$.
\item $\cspn\{a\cdot x: a\in A, x\in X\}=X$ (which by an application of Cohen's factorisation 
theorem is equivalent to $X=AX=\{ax: a\in A, x\in X\}$).
\end{enumerate}
A {\em morphism}  between two $C^*$-operator systems $(A,X)$ and $(B,Y)$ 
on Hilbert spaces $H$ and $K$, respectively, 
consists of a ccp map 
$\psi: X\to Y$ such that 
\begin{enumerate}
\item $\psi(A)\subseteq B$, and
\item for all $a\in A$ and $x\in X$ we have $\psi(a x) =\psi(a)\psi(x)$.
\end{enumerate}
A morphism $\psi:X\to Y$ is called {\em non-degenerate} if $\psi(A)Y=Y$.
We say that the $C^*$-operator system $(A,X)$ is unital, if $A$ is unital.
\end{definition}

\begin{example}\label{excstaros}
\begin{enumerate}
\item Clearly every classical operator system $X\subseteq \mathcal B(H)$ can be regarded 
as a unital $C^*$-operator system with respect to the $C^*$-subalgebra $A=\C 1\subseteq X$.
In that case a nonegenerate morphism from $(\C1, X)$ to $(\C1, Y)$ is just 
a ucp map.

\item If $(A,X)$ is a unital $C^*$-operator system on the Hilbert space $H$, then the unit of $A$
coincides with the identity operator on $H$ since $A$ acts non-degenerately on $H$.
Hence $X$ is also a classical operator system.

\item Every non-degenerate $C^*$-subalgebra $A\subseteq \mathcal B(H)$
gives rise to the $C^*$-operator system $(A,A)$ on $H$.
%
\item Suppose that $(A,X)$ and $(B,Y)$ are $C^*$-operator systems on the Hilbert spaces
$H$ and $K$, respectively. Then the norm-closed tensor product $X\otimes Y\subseteq \mathcal B(H\otimes K)$
contains the minimal tensor product $A\otimes B$ as a sub-$C^*$-algebra such that 
$(A\otimes B, X\otimes Y)$ becomes a $C^*$-operator system on $H\otimes K$.
In particular, if $X\subseteq \mathcal B(H)$ is a classical unital operator system and $B\subseteq \mathcal B(K)$ is a 
$C^*$-algebra, then the  minimal tensor product $X\otimes B$ has the structure of  a $C^*$-operator system 
with $C^*$-subalgebra $B\cong \C1\otimes B\subseteq X\otimes B$. This example shows that 
$C^*$-operator systems do appear quite naturally!
\end{enumerate}
\end{example}

\begin{definition}\label{def-rep}
Let $(A,X)$ be a $C^*$-operator system on the Hilbert space $H$.
A {\em representation} of $(A,X)$ on a Hilbert space $K$ is a ccp map 
$\pi: X\to \mathcal B(K)$ such that $\pi(ax)=\pi(a)\pi(x)$ for all $a\in A$ and $x\in X$ (in particular, $\pi|_A$ is 
a $*$-representation of $A$).
The representation $\pi$ is called non-degenerate, if $\pi|_A:A\to \mathcal B(K)$ is non-degenerate.
\end{definition}

%
%
\begin{definition}\label{def-nuos}
Suppose that $X\subseteq \mathcal B(H)$ is a self-adjoint norm-closed subset of $\mathcal B(H)$. A norm bounded 
net $(u_i)_{i\in I}$ of self-adjoint elements in $X$ 
is called an {\em approximate unit} for $X$ if 
 for all $x\in X$, $i\in I$ we have $u_ix, xu_i\in X$ and $u_ix, xu_i\to x$ in the norm of $\mathcal B(H)$.
\end{definition}

\begin{lemma}\label{lem-approx}
Suppose that $(A,X)$ is a $C^*$-operator system on $H$. Then every bounded 
self-adjoint approximate unit $(u_i)_{i\in I}$ of $A$ is an approximate unit for $X$ 
in the sense of the above definition. Moreover, $u_i\to 1_H$ $*$-strongly in $\mathcal B(H)$.
\end{lemma}
\begin{proof} The first assertion follows immediately from the requirement $AX=X$ for a 
$C^*$-operator system. The second assertion follows from $H=AH$.
\end{proof}

\begin{definition}\label{def-unitization}
Suppose that $(A,X)$ is a $C^*$-operator system on some Hilbert space $H$.
By a unitization of $(A,X)$ we understand a unital $C^*$-operator system $(\tilde{A},\tilde{X})$ on 
$H$ which contains $(A,X)$ such that the following are satisfied
\begin{enumerate}
\item $A$ is an ideal in $\tilde{A}$ and $A\tilde{X}\subseteq X$.
\item If $x\in \tilde{X}$ such that $ax=0$ for all $a\in A$, then $ x=0$.
\end{enumerate}
\end{definition}

It is well known that for a $C^*$-algebra $A$, the multiplier algebra $M(A)$ 
is the largest unitization of $A$. We shall now introduce an analogous construction
for $C^*$-operator systems:

\begin{lemma}\label{lem-multipliers}
Suppose that $(A,X)$ is a $C^*$-operator system on the Hilbert space $H$.
Let $M(A)=\{m\in \mathcal B(H): mA\cup Am\subseteq A\}$ be the realisation of the multiplier algebra of 
$A$ in $\mathcal B(H)$ and let 
$$M(X)=\{k\in \mathcal B(H): kA\cup Ak\subseteq X\}\subseteq \mathcal B(H).$$ 
Then $(M(A), M(X))$ is a unitization of $(A,X)$ in $\mathcal B(H)$.

Moreover,  if $\pi:X\to \mathcal B(K)$ is any non-degenerate ccp representation of $(A,X)$ on some Hilbert space $K$,
then there exists a unique unital extension $\bar\pi:M(X)\to \mathcal B(K)$ of $\pi$ as a ccp representation of 
 $(M(A), M(X))$ on $K$. Moreover,  $\bar\pi$ is completely isometric iff $\pi$ is completely isometric.
\end{lemma}

\begin{notation}\label{not-multiplier}
We call $(M(A), M(X))$ the {\em multiplier $C^*$-operator system} of $(A,X)$.

Notice that the space $M(X)$ very much depends on the $C^*$-subalgebra 
$A\subseteq X$, so that a  better notation would probably  be to write $M_A(X)$ instead 
of $M(X)$. However, it will always be clear from the context with respect to which $C^*$-subalgebra 
$A\subseteq X$  the set $M(X)$ is defined, so in order to keep notation simple 
we stick to $M(X)$.
\end{notation}

\begin{proof}[Proof of Lemma \ref{lem-multipliers}]
It is trivial to check that $(M(A), M(X))$ fulfils all properties of a unital
$C^*$-operator system. Note that $M(A)X\subseteq X$ since $X=AX$ and hence 
$M(A)X=M(A)(AX)=(M(A)A)X=AX=X$ and similarly $XM(A)=X$.
This easily implies that $M(A)M(X)\subseteq M(X)$.
Moreover, if $k\in M(X)$ such that $ak=0$ for all $a\in A$, then 
we also have $k^*a=0$ for all $a\in A$, hence 
$k^*(AH)=k^*H=\{0\}$ which then implies that $k^*=0$. But then $k=0$.
Thus it follows that $(M(A), M(X))$ is a unitization of $(A,X)$.

Suppose now that  $\pi:X\to \mathcal B(K)$ is a non-degenerate ccp representation of $(A,X)$.
Then  $\pi(A)K=K$ and we define the extension
$\bar\pi:M(X)\to \mathcal B(K)$ by
$$\bar\pi(k)(\pi(a)\xi):=\pi(ka)\xi.$$
To see that this is well defined, let 
let $(u_i)_{i\in I}$ be an approximate unit of $A$ consisting of 
positive elements of norm $\leq 1$.
Then, if $\pi(a_1)\xi_1=\pi(a_2)\xi_2$ for some elements $a_1,a_2\in A$ and $\xi_1, \xi_2\in K$,
we get
\begin{align*}
\pi(ka_1)\xi_1&=\lim_i\pi(ku_ia_1)\xi_1=\lim_i\pi(ku_i)\pi(a_1)\xi_1\\
&=\lim_i\pi(ku_i)\pi(a_2)\xi_2=\lim_i\pi(ku_ia_2)\xi_2=\pi(ka_2)\xi_2.
\end{align*}
This shows that $\bar\pi$ is well defined.

We now need to check that $\bar\pi(mk)=\bar\pi(m)\bar\pi(k)$ for all $m\in M(A)$ and $k\in M(X)$.
We first show that $\pi(ka)=\bar\pi(k)\pi(a)$ for all $k\in M(X), a\in A$.
To see this let $\eta\in H$. Then $\eta=\pi(b)\xi$ for some $b\in A, \xi\in K$. 
Then 
$$\pi(ka)\eta=\pi(kab)\xi=\bar\pi\pi(ab)\xi=\bar\pi\pi(a)\pi(b)\xi=\bar\pi(k)\pi(a)\eta.$$
Suppose now that $m\in M(A)$ and $k\in M(X)$. Then, for $\xi, \eta\in H$ and $a,b\in A$
we get
\begin{align*}
\langle \bar\pi(mk)\pi(a)\xi, \pi(b)\eta\rangle&= \langle \pi(b^*)\bar\pi(mk)\pi(a)\xi, \eta\rangle\\
&= \langle \pi(b^*)\pi(mka)\xi, \eta\rangle\\
&= \langle \pi(b^*mka)\xi, \eta\rangle\\
&= \langle \pi(b^*m)\pi(ka)\xi, \eta\rangle\\
&= \langle \pi(b^*)\bar\pi(m)\bar\pi(k)\pi(a)\xi, \eta\rangle\\
&= \langle \bar\pi(m)\bar\pi(k)\pi(a)\xi, \pi(b)\eta\rangle
\end{align*}
which then implies that $\bar\pi(mk)=\bar\pi(m)\bar\pi(k)$.

We need to show that $\bar\pi:M(X)\to \mathcal B(K)$ is completely positive. 
Since it is unital, this will also imply that it is completely contractive.
If $(A,X)$ is unital, then $M(X)=X$ and nothing has to be done.

If $(A,X)$ is not unital, let $(u_i)_{i\in I}$ be an approximate unit of $A$ consisting of 
positive elements of norm $\leq 1$. Since $\pi(A)K=K$ it follows that 
 $\pi(u_i)\to 1_K$ $*$-strongly in $\mathcal B(K)$. Then, if $m\in M(X)$, it follows that 
 $\pi(u_imu_i)=\pi(u_i)\bar\pi(m)\pi(u_i)$ weakly to $\bar\pi(m)$ in $\mathcal B(K)$.
Now let  $m\in M_n(M(X))\subseteq B(K^n)$ be any positive element. Let $v_i:=u_i \otimes I_n\in M_n(X)$.
Then $v_imv_i$ is a positive element of $M_n(X)$ such that $\pi_n(v_imv_i)$ converges weakly
to $\bar\pi_n(m)$ in $\mathcal B(K^n)$. Since weak limits of positive elements 
are positive, it follows that $\bar\pi_n(m)$ is positive.

Finally assume that $\pi:X\to \mathcal B(K)$ is completely isometric and let  $(\widetilde{A},\widetilde{X}):=(\pi(A),\pi(X))$ denote the image
of $(A,X)$ in $\mathcal B(K)$. Then by the first part of this proof applied to the system  $(\widetilde{A},\widetilde{X})$
 the inverse $\pi^{-1}:\widetilde{X}\to \mathcal B(H)$ extends 
uniquely to a ccp representation $\bar{\pi}^{-1}:M(\widetilde X)\to \B(H)$. Since $\bar\pi^{-1}\circ \bar\pi|_X:X\to \mathcal B(H)$ coincides with the 
identity on $X$, it follows from the uniqueness assertion for the extension to $M(X)$ that 
$\bar\pi^{-1}\circ \bar\pi:M(X)\to \mathcal B(H)$ is the identity on $M(X)$. Similarly, $\bar\pi\circ \bar{\pi}^{-1}$ is the identity on $M(\widetilde{X})$.
In particular, $\bar\pi:M(X)\to \mathcal B(K)$ is completely isometric.

\end{proof}

The following lemma shows that $(M(A), M(X))$ is the largest unitization of 
$(A,X)$. 

\begin{lemma}\label{lem-largest-unit}
Let $(A,X)$ be a $C^*$-operator system in $\mathcal B(H)$ and 
suppose that $(\tilde{A},\tilde{X})$ is a unitization of $(A,X)$ in $\mathcal B(H)$.
Then $\tilde{A}\subseteq M(A)$ and $\tilde{X}\subseteq M(X)$.

As a consequence, if $\pi:X\to \mathcal B(K)$ is a non-degenerate ccp representation of 
$(A,X)$ on a Hilbert space $K$, there exists a unique ccp representation 
$\tilde\pi:\tilde{X}\to \mathcal B(K)$ of $(\tilde{A},\tilde{X})$ which extends $\pi$.
\end{lemma}
\begin{proof}
Clearly, if $(\tilde{A},\tilde{X})$ is a unitization of $(A,X)$ in $\mathcal B(H)$,
then every $m\in \tilde{X}$ multiplies $A$ into $X$. Hence $\tilde{X}\subseteq M(X)$.
Since $A$ is an ideal in $\tilde{A}$ we also have $\tilde{A}\subseteq M(A)$. By Lemma \ref{lem-multipliers}
we know that $\pi$ extends uniquely to a representation $\bar\pi$ of $M(X)$.
We then put $\tilde\pi:=\bar\pi|_{\tilde{X}}$.
\end{proof}

\begin{definition}\label{def-nondeg-morphism}
Suppose that $(A,X)$ and $(B,Y)$ are $C^*$-operator systems. 
We say that $\varphi:X\to M(Y)$ is a (non-degenerate) {\em generalized morphism}
from $(A,X)$ to $(B,Y)$ if the following holds:
\begin{enumerate}
\item $\varphi:X\to M(Y)$ is a morphism from $(A,X)$ to $(M(B), M(Y))$, and
\item $\varphi(A)B=B$.
\end{enumerate}
\end{definition}

Note that since $BY=Y$, the  condition (2) also implies that $\varphi(A)Y=Y$.

\begin{lemma}\label{lem-non-deg-morphism}
Suppose that $(A,X)$ and $(B,Y)$ are $C^*$-operator systems. 
If $\varphi:X\to M(Y)$ is a non-degenerate generalized morphism from 
$(A,X)$ to $(B,Y)$, then there exists a unique extension 
$\bar\varphi:M(X)\to M(Y)$ of $\varphi$ as a morphism from
$(M(A), M(X))$ to $(M(B), M(Y))$.
In particular, if $\varphi:(A,X)\to (B,Y)$ is a completely positive
and completely isometric isomorphism of $C^*$-operator systems, the 
same holds for the extension $\bar\varphi: M(X)\to M(Y)$.
\end{lemma}
\begin{proof}
Assume that $(B,Y)$ and hence $(M(B), M(Y))$ are $C^*$-operator systems on 
the Hilbert space $K$. It follows then from the condition that $\varphi(A)B=B$
that $\varphi(A)K=\varphi(A)(BK)=BK=K$, so $\varphi:A\to \mathcal B(K)$ is a 
non-degenerate representation of $(A,X)$ on $\mathcal B(K)$. 
By Lemma \ref{lem-multipliers} we know that there is a unique ccp extension 
$\bar\varphi:M(X)\to \mathcal B(K)$. 
We then get 
$$\bar\varphi(M(X))B=\bar\varphi(M(X))\varphi(A)B=\varphi(M(X)A)B
\subseteq \varphi(X)B\subseteq M(Y)B\subseteq Y,$$
hence $\bar\varphi(M(X))\subseteq M(Y)$. A similar argument shows 
that $\bar\pi(M(A))\subseteq M(B)$. 

For the final statement assume that $\varphi:X\to Y$ is a completely 
isometric isomorphism of the $C^*$-operator systems $(A,X)$ and $(B,Y)$.
Let $\bar\varphi:M(X)\to M(Y)$ and ${\bar\varphi}^{-1}:M(Y)\to M(X)$ 
denote the unique extensions of $\varphi$ and $\varphi^{-1}$ to $M(X)$ and $M(Y)$, respectively.
Then ${\bar\varphi}^{-1}\circ \bar{\varphi}:M(X)\to M(X)$ extends the 
identity on $X$, and hence, by the uniqueness of the extension, must 
be equal to the identity mal on $M(X)$. Similarly, 
$\bar{\varphi}\circ {\bar\varphi}^{-1}$ is the identity on $M(Y)$.
\end{proof}

\begin{corollary}\label{cor-isometry}
Let $(A,X)$ be a $C^*$-operator system on $H$ and suppose
 that $\pi:X\to \mathcal B(K)$ is a completely isometric c.p representation of 
 $(A,X)$ on $K$. Then the unique extension $\bar\pi:M(X)\to \mathcal B(K)$
 is completely isomeric as well. The same holds for the extension
 $\tilde\pi:\tilde{X}\to \mathcal B(K)$ for any unitization $(\tilde{A},\tilde{X})$
 of $(A,X)$.
 \end{corollary}
 \begin{proof} If $\pi:X\to \mathcal B(K)$ is completely isometric, then 
 $(\pi(A),\pi(X))$ is a $C^*$-operator system in $\mathcal B(K)$ and 
 $\pi:X\to \pi(X)$ is a completely isometric isomorphism of 
 $C^*$-operator systems. Thus the unique extension 
 $\bar\pi:M(X)\to M(\pi(X))\subseteq \mathcal B(K)$ is completely 
 isometric by Lemma \ref{lem-non-deg-morphism}.
 \end{proof}

Of course there is also a smallest unitization of $(A,X)$:

\begin{definition}\label{def-adjoining-unit}
Suppose that $(A,X)$ is a $C^*$-operator system on a Hilbert space $H$. 
Let $X^1=X+\C1_H$ and $A^1=A+\C1_H\subseteq X^1$.
Then $(A^1, X^1)$ is the smallest unitization of $(A,X)$ in $\mathcal B(H)$.
We call it the {\em minimal} unitization of $(A,X)$.
\end{definition}

\begin{remark}\label{rem-extension}
Of course, if $\pi:X\to \mathcal B(K)$ is any ccp representation of 
$(A,X)$ on a Hilbert space $K$, then the unique extension 
$\pi^1:X^1\to \mathcal B(K)$ is given by $\pi^1(x+\lambda 1_H)=\pi(x)+\lambda 1_K$.
By Corollary \ref{cor-isometry},  if $\pi$ is 
completely isomertric, then $\pi^1$  is completely isometric as well.
\end{remark}

\section{$C^*$-hulls of
 $C^*$-operator systems}\label{sec-univ-env}

If $(A,X)$ is a $C^*$-operator system, $C$ is a $C^*$-algebra,
 and  $j:X\to C$ is a completely positive complete isometry
such that $j(ax)=j(a)j(x)$ for all $a\in A, x\in X$ and such 
$X$ generates $C$ as a $C^*$-algebra, then the pair $(C, j)$ 
is called a $C^*$-hull of $(A,X)$. Two $C^*$-hulls $(C,j)$ and $(C', j')$ of $(A,X)$ are 
called equivalent, if there exists a $*$-isomorphism 
$\varphi: C\to C'$ such that $\varphi\circ j=j'$. 
In what follows below we want to show that for any $C^*$-operator system $(A,X)$ there 
exist $C^*$-hulls $(C_u^*(A,X), j_u)$ and $(C_{\env}^*(A,X), j_{\env})$ such that for any
given $C^*$-hull $(C,j)$ of $(A,X)$ there exist unique surjective $*$-homomorphisms
$$C_u^*(A,X)\stackrel{\varphi_u}{\onto} C\stackrel{\varphi_{\env}}{\onto} C_{\env}^*(A,X)$$
such that $\varphi_u\circ j_u=j$ and $\varphi_{\env}\circ j=j_{\env}$.
It follows  directly from  these universal properties of $(C_u^*(A,X), j_u)$ and $(C_{\env}^*(A,X), j_{\env})$
that they are unique up to equivalence (if they exist). We call $(C_u^*(A,X), j_u)$
the universal $C^*$-hull of $(A,X)$ and we call $(C_{\env}^*(A,X), j_{\env})$ the 
enveloping $C^*$-algebra of $(A,X)$. 
%
%
%
%
Of course, the above notion of the universal $C^*$-hull of a $C^*$-operator 
system extends the notion of the universal $C^*$-hull of a classical 
operator system $X$ as introduced by Kirchberg and Wassermann in \cite{KW}
and the notion of the $C^*$-envelope extends the well-known notion of a $C^*$-envelope 
of an operator system due to Hamana \cite{Hamana1}.
Using the ideas of Kirchberg and Wassermann, we  now construct the universal $C^*$-hull  $(C_u^*(A,X), j_u)$.
We need:

\begin{definition}\label{def-Agenerated}
Suppose that $(A,X)$ is a $C^*$-operator system. A  representation 
$\pi:X\to \mathcal B(K)$ is called {\em finitely $A$-generated}, if there exists a 
finite subset $\{\xi_1,\ldots, \xi_l\}$ of $K$ such that $K=\cspn\{\pi(A)\xi_1,\ldots, \pi(A)\xi_l\}$.
\end{definition}

If $\kappa$ is the cardinality of a dense subset of $A$, then every finitely $A$-generated 
representation of $(A,X)$ can be regarded, up to unitary equivalence, as a representation
on a closed subspace of $\ell^2(I_\kappa)$, where $I_\kappa$ is a fixed set with cardinality $\kappa$.

\begin{theorem}\label{thm-universal}
For every $C^*$-operator system $(A,X)$ there exists a
universal hull $(C_u^*(A,X), j_u)$ for $(A,X)$.
\end{theorem}
\begin{proof}  Let $\kappa$ denote the cardinality of a dense subset of $A$ and let 
$S$ denote the set of all non-degenerate finitely $A$-generated ccp
representations $\pi: X\to \mathcal B(H_\pi)$ where $H_\pi$ is a closed subspace of  $\ell^2(I_\kappa)$.
Write $H_S=\bigoplus_{\pi\in S} H_\pi$ and $\pi_S=\bigoplus_{\pi\in S}\pi$. 

We claim that  $\pi_S: X\to \mathcal B(H_S)$ is a completely isometric representation of $(A,X)$.
For this  let 
us assume that $(A,X)$ is represented as a concrete $C^*$-operator system on the Hilbert space 
$H$.  Then  for each fixed $n\in \N$, $x\in M_n(X)$, and $\eps>0$ we  choose a finite rank projection 
$p\in \mathcal B(H)$ such that 
$$\|(p\otimes 1_n) x (p\otimes 1_n)\|\geq \|x\|-\eps.$$
Let $H_{x,\eps}:=\cspn\{apH:a\in A\}$ and let $q:H\to H_{x,\eps}$ denote the orthogonal projection.
Define
$$\pi_{x,\eps}: X\to \mathcal B(H_{x,\eps}); \pi_{x,\eps}(y):=qy q$$
for all $y\in X$. Since  $H_{x,\eps}$ is an $A$-invariant subspace of $H$, we 
see that $q$ commutes with the elements of $A$, hence 
$$\pi_{x,\eps}(ay)=qa y q=qaqyq=\pi_{x,\eps}(a)\pi_{x,\eps}(y)$$
for all $a\in A$, $y\in X$, so $\pi_{x,\eps}$ is a ccp 
representation of $(A,X)$ on $H_{x,\eps}$. By construction, $\pi_{x,\eps}$ is finitely $A$-generated and 
$\|\pi_{x,\eps,n}(x)\|\geq \|x\|-\eps$.  By choosing an isometric embedding of $H_{x,\eps}$ into $\ell^2(I_{\kappa})$ we 
may assume that $\pi_{x,\eps}\in S$. Since $\eps$ is arbitrary, it follows now that $\pi_S$ is completely
isometric.

We now define $C_u^*(A,X)$ as the $C^*$-subalgebra of $\mathcal B(H_S)$ generated by $\pi_S(X)$
and $j_u=\pi_S: X\to C_u^*(A,X)\subseteq \mathcal B(H_S)$. We then have
$$C_u^*(A,X)\subseteq \prod_{\pi\in S} \mathcal B(H_\pi)\subseteq \mathcal B(H_S).$$

Suppose now that $(C, j)$ is an arbitrary $C^*$-hull of $(A,X)$. We may assume that 
$C$ is realised as a non-degenerate subalgebra $C\subseteq \mathcal B(K)$
for some Hilbert space $K$. Let $C^1=C+\C1_K\subseteq \mathcal B(K)$ be the unitization of $C$.
Applying the above construction to each element $c\in M_n(C^1)$ yields a family of completely positive
maps $\rho_{c,\eps}:C^1\to \mathcal B(H_{\rho_{c,\eps}})$ such that $\rho_{c,\eps}\circ j:X\to \mathcal B(H_{\rho_{c,\eps}})$ is a finitely $A$-generated 
representation of $(A,X)$ and such that $\|\rho_{c,\eps,n}(c)\|\geq \|c\|-\eps$.
Let $S'$ denote the set of all such maps $\rho_{c,\eps}$. 
Then  $\rho_{S'}=\bigoplus_{\rho\in S'}\rho:C^1\to \mathcal B(\bigoplus_{\rho\in S'} H_\rho)$ is a
unital completely isometric map from $C^1$ into $\prod_{\rho\in S'} \mathcal B(H_\rho)$.
By \cite[Theorem 4.1]{CE2} there exists a unique unital $*$-homomorphism 
$\varphi: C^*(\rho_{S'}(C^1))\to C^1$ such that $\varphi\circ \rho_{S'}=\id_{C^1}$.
Since for each $\rho\in S'$, $\rho\circ j:X\to \mathcal B(H_\rho)$ is a finitely $A$-generated representation of $(A,X)$, 
we may identify $\rho\circ j$ 
with an element of $S$ via a suitable embedding of $H_\rho\hookrightarrow \ell^2(I_\kappa)$.
We then obtain a map $S'\to S; \rho\mapsto \rho\circ j$ and a 
$*$-homomorphism $\Phi: \prod_{\pi\in S} B(H_\pi)\to \prod_{\rho\in S'}\mathcal B(H_\rho)$ by
 sending a tupel $(T_\pi)_{\pi\in S}$ to $(T_{\rho\circ j})_{\rho\in S'}$.
 The restriction of $\Phi$ to $C_u^*(A,X)\subseteq \prod_{\pi\in S}\mathcal B(H_\pi)$ sends 
 $C_u^*(A,X)$ to $C^*(\rho_{S'}(j(X)))\subseteq C^*(\rho_{S'}(C^1))$. Thus we get a 
composition of  $*$-homomorphism 
 $$
 \begin{CD}
C_u^*(A,X) @>\Phi>> C^*(\rho_{S'}(j(X))) @>\varphi>> C
  \end{CD}
  $$ 
  such that for  $\varphi_u:=\varphi\circ \Phi$ we get $\varphi_u\circ j_u=j$. 
\end{proof}

\begin{lemma}\label{lem-universal}
Suppose that $\pi:X\to \mathcal B(K)$ is a ccp representation of the $C^*$-operator system $(A,X)$.
Then there exists a unique $*$-homomorphism $\tilde\pi:C_u^*(A,X)\onto C^*(\pi(X))\subseteq \mathcal B(K)$
such that $\tilde\pi\circ j_u=\pi$,
where $C^*(\pi(X))$ denotes the closed $C^*$-subalgebra of $\mathcal B(K)$ generated by $\pi(X)$.
\end{lemma}
\begin{proof}
Suppose that $(A,X)$ is a concrete $C^*$-operator system on the Hilbert space $H$ and 
assume that $\iota:X\into \mathcal B(H)$ is the inclusion map. Then $\iota\bigoplus \pi: X\to B(H\oplus K)$
is a completely isometric representation and therefore there exists a unique $*$-homomorphism
$\widetilde{\iota\oplus\pi}:C_u^*(A,X)\to C^*(\iota\oplus\pi(X))\subseteq B(H\otimes K)$. 
As $\iota\oplus\pi(X)\subseteq X\oplus \pi(X)\subseteq \mathcal B(H)\oplus \mathcal B(K)$, we obtain 
a well defined $*$-homomorphism $C^*(\iota\oplus \pi(X))\to C^*(\pi(X))$ given by
$T\mapsto P_KTP_K$, where $P_K:H\oplus K\to K$ denotes the orthogonal projection.
Thus $\tilde\pi=P_K \widetilde{\iota\oplus\pi}(\cdot)P_K$ will do the job. The uniqueness 
follows from the fact that $C_u^*(A,X)$ is generated by $j_u(X)$.
\end{proof}

At this point it is convenient to consider representations of $C^*$-operator systems on multiplier algebras:

\begin{definition}\label{def-rep-multiplier}
Suppose that $(A,X)$ is a $C^*$-operator system and let $D$ be a $C^*$-algebra. 
A  representation of $(A,X)$ into the multiplier algebra $M(D)$ 
is a ccp map $\Phi: X\to M(D)$ such that $\Phi(ax)=\Phi(a)\Phi(x)$ for all $a\in A, x\in X$.
We then say that $\Phi$ is {\em non-degenerate} if the restriction of $\Phi$ to $A$ is 
a non-degenerate $*$-homomorphism, i.e., if $\Phi(A)D=D$.
\footnote{Note that by Cohen's factorization theorem to have $\Phi(A)D=D$ it suffices to have
that $\cspn\{\Phi(A)D\}=D$.}
\end{definition}

\begin{remark}\label{rem-rep-multiplier}
Note that every (non-degenerate) representation of a $C^*$-operator system $(A,X)$ on a Hilbert space $H$
an be regarded as a (non-degenerate) representation into $M(\mathcal K(H))=\mathcal B(H)$. 
Conversely, if $\Phi:X\to M(D)$ is a representation of $(A,X)$ in $M(D)$ and 
if $D$ (and hence $M(D)$) is represented faithfully on the Hilbert space $H$, then 
$\Phi$ can also be regarded as a representation of $(A,X)$ on $H$ which is non-degenerate
if and only if $\Phi:X\to M(D)$ is non-degenerate.  But it is often more convenient 
to work with representations into $M(D)$.
\end{remark}

We now get

\begin{proposition}\label{prop-universal}
Let $(A,X)$ be a $C^*$-operator system and let $D$ be a $C^*$-algebra. Then there is a one-to-one correspondence 
between
\begin{enumerate}
\item The non-degenerate representations of $(A,X)$ into $M(D)$.
\item The non-degenerate $*$-homomorphisms of $C_u^*(A,X)$ into $M(D)$.
\end{enumerate}
If $\Phi: C_u^*(A,X)\to M(D)$ is as in (2), then the restriction $\Phi_X=\Phi\circ j_u:X\to M(D)$ 
gives the corresponding representation of $(A,X)$ as in (1).
\end{proposition}
\begin{proof} It clearly suffices to show that every non-degenerate representation of 
$(A,X)$ on $M(D)$ extends to a representation of $C_u^*(A,X)$. But representing 
$C_u^*(A,X)$ faithfully on a Hilbert space $H_u$, say, this follows easily from
Lemma \ref{lem-universal}.
\end{proof}

We now proceed with a discussion of the enveloping $C^*$-hull for $(A,X)$.
For this recall that an operator space  $V$  is {\em injective} if,  given  operator spaces $W_{1}\subseteq W_{2}$,
 any completely bounded linear map
$\varphi_{1}:W_{1}\rightarrow V$ can be extended to a completely bounded linear map
$\varphi_{2}:W_{2}\rightarrow V$ with
$\|\varphi_{2}\|_{cb}=\|\varphi_{1}\|_{cb}.$
The algebra $\mathcal B(H)$ is known to be an injective operator space \cite{Wittstock}.
  Hamana in \cite{Hamana1, Hamana2} and Ruan  in \cite{Ruan}
 independently  showed that
 for any operator space $V$ in $\mathcal B(H)$,  there is a unique minimal injective operator subspace $I(V)$ of $\mathcal B(H)$ containing $V$.
  It is called the injective envelope of $V$ and enjoys the following fundamental property, which we shall use
    heavily throughout this paper (e.g., see \cite[\S5]{Ruan}):

  \begin{proposition}\label{prop-unique}
  Let $V\subseteq \mathcal B(H)$ be an operator space. Then every completely contractive map $\psi:I(V)\to I(V)$
  which restricts to the identity on $V$ is the identity on $I(V)$.
  \end{proposition}
%

 We need the following result of Choi-Effros \cite{CE} (see \S6 in \cite{Effros} and particularly \cite[Theorem 6.1.3]{Effros}).

 \begin{theorem}\label{thm-product} If  $I \subseteq \mathcal B(H)$  is an injective operator system, then there is a unique
multiplication
$\circ :I\times I\to I$
making $I$ a unital $C^*$-algebra with its given $*$-operation and norm and identity $\one_H$. The 
multiplication is 
given by 
$$x \cdot_{\varphi} y = \varphi(xy),$$
where $\varphi:\mathcal B(H)\to I$ is a fixed ccp onto projection.
   \end{theorem}
   
Using these results we now show

\begin{proposition}\label{prop-envelope}
Suppose that $(A,X)$ is a $C^*$-operator system. Then there exists an enveloping $C^*$-hull
$(C^*_{\env}(A,X), j_{\env})$ of $(A,X)$.
\end{proposition}
\begin{proof}
Suppose that $(A,X)$ is a $C^*$-operator system on $H$. Let $(A^1, X^1)$ be the unitization of $(A,X)$ as in 
Definition \ref{def-adjoining-unit}. By Theorem \ref{thm-product} the injective envelope $I(X^1)$ 
of the unital operator system $X^1$ is a unital $C^*$-algebra with multiplication
$x\cdot_{\varphi} y=\varphi(xy)$ for some fixed ccp onto projection $\varphi:\mathcal B(H)\to I(X^1)$. 
Now, for each $a\in A$ and $x\in X$ we have $ax\in X$ and therefore $a\cdot_\varphi x=\varphi(ax)=ax$.
 Therefore the inclusion map $X\into \mathcal B(H)$ induces a completely isometric embedding $j:X\to I(X^1)$ 
such that $j(ax)=j(a)j(x)$ for all $a\in A, x\in X$. Define $C_{\env}^*(A,X)$ to be the  $C^*$-subalgebra 
of $I(X^1)$ generated by $j(X)$ and we let $j_{\env}=j:X\into C_{\env}^*(A,X)$ denote the inclusion map.
Note that by construction, the unitization  $C_{\env}^*(A,X)^1$ is just the enveloping $C^*$-algebra $C^*_{\env}(X^1)$ of the unital operator system
$X^1$ in the sense of Hamana \cite{Hamana1}.

To see that $(C_{\env}^*(A,X), j_{\env})$ satisfies the universal property let 
$(C,j)$ be any given $C^*$-hull of $(A,X)$. Choose a non-degenerate embedding $C\into \mathcal B(K)$ for some 
Hilbert space $K$ and let $C^1=C+\C1_K\subseteq \mathcal B(K)$. Then $j^1:X^1\to C^1$ is a completely isometric
embedding of the operator system $X^1$. It  follows therefore from the universal property of the 
enveloping $C^*$-algebra $C^*_{\env}(X^1)$ (see \cite{Hamana1}) 
that there exists a $*$-homomorphism $\varphi:C^1\to C^*_{\env}(X^1)=C_{\env}^*(A,X)^1$ 
which intertwines the inclusions of $X^1$ into these algebras. Restricting $\varphi$ to $C\subseteq C^1$ 
then gives the desired $*$-homomorphism $\varphi_{\env}:C\to C_{\env}^*(A,X)$.
\end{proof}

We close this section with the following useful result:

%

\begin{lemma}\label{lem mult univ}
Let $(C, j)$ be any $C^*$-hull of the $C^*$-operator system $(A,X)$. Then the inclusion map $j: (A,X)\to C$ extends to 
a completely isometric inclusion $$\bar{j}: \big( M(A), M(X)\big)\to M(C).$$
Moreover, we have $\bar{j}(M(A))\cap C=j(A)$.
\end{lemma}
\begin{proof} Since $A$ contains an approximate identity of $X$, and since $C$ is generated by 
$j(X)$ as a $C^*$-algebra, it follows that $j(A)$ contains an approximate identity of $C$.
It follows that $j:(A,X)\to C\subseteq M(C)$ is a completely isometric non-degenerate
(generalized) morphism, where we identify  $C$ with the $C^*$-operator system $(C,C)$.
The first assertion  then follows from Lemma \ref{lem-non-deg-morphism}.

To see the second assertion let $c\in C$ such that $ j(A) c\subseteq j(A)\subseteq M(C)$.
Let $(a_i)_{i\in I}$ be an approximate unit in $A$. Then $(a_i)_{i\in I}$ is also an approximate 
unit in $X$, and $(j(a_i))_{i\in I}$ is an approximate unit in $C$. 
But then it follows that $c=\lim_i j(a_i) c\in j(A)$.

\end{proof}

\section{Tensor products}\label{sec-tensor}
In this section we want to give a brief discussion on certain tensor product constructions of $C^*$-operator systems.
In particular we want to discuss analogues of the commutative maximal tensor product $\mathcal S\otimes_c\mathcal T$ of 
two operator systems $\mathcal S$ and $\mathcal T$  as introduced in \cite{KPTT} and of the minimal (or spacial) tensor product.

\begin{definition}\label{def-universal-tensor}
Suppose that $(A,X)$ and $(B,Y)$ are $C^*$-operator systems. 
Let $A\otimes_c B$ and $X\otimes_cY$ denote the closures of the algebraic tensor products
$A\odot B$ and $X\odot Y$ 
inside the maximal $C^*$-tensor product $C_u^*(A,X)\otimes_{\max}C_u^*(B,Y)$, respectively.
Then  $(A\otimes_c B,  X\otimes_c Y)$ is a $C^*$-operator system which we call
the {\em commuting universal tensor product} of $(A,X)$ with $(B,Y)$.
\end{definition}

\begin{lemma}\label{lem-factors}
There are unique completely isometric generalized morphisms  $i_X:X\to M(X\otimes_cY)$ 
and $i_Y:Y\to M(X\otimes_cY)$ such that $i_X(x)i_X(y)=x\otimes y\in X\otimes_cY$ for all $x\in X, y\in Y$.
\end{lemma}
\begin{proof} Write $D:=C_u^*(A,X)\otimes_{\max}C_u^*(B,Y)$ and assume that $D$ is represented 
faithfully and non-degenerately on a Hilbert space $K$, say.
By the properties of the maximal tensor product of $C^*$-algebras, there are 
isometric $*$-homomorphisms $i_{C_u^*(A,X)}:C_u^*(A,X)\to M(D)$
and $i_{C_u^*(B,Y)}:C_u^*(B,Y)\to M(D)$
such that $i_{C_u^*(A,X)}(c)i_{C_u^*(B,Y)}(d)=c\otimes d$ for all $c\in C_u^*(A,X), d\in C_u^*(B,Y)$,
respectively.  Let $i_X$ and $i_Y$ denote the restrictions  of 
$i_{C_u^*(A,X)}$ and $i_{C_u^*(B,Y)}$ to $X$ and $Y$, respectively.
Then $i_X$ and $i_Y$ are completely isometric representation of $(A,X)$ and $(B,Y)$ 
into $M(D)\subseteq \mathcal B(K)$ such that $i_X(x)i_Y(y)=x\otimes y$ for all $x\in X, y\in Y$,
if we regard the algebraic tensor product $X\odot Y$ as a subspace of $X\otimes_cY$.
So all we need to check is that $i_X$ and $i_Y$ have image in $M(X\otimes_cY)$, which
follows easily from $i_X(x)(a\otimes b)=i_X(x)i_X(a)i_Y(b)=i_X(xa)i_Y(b) =xa\otimes b\in X\otimes_cY$  
hence $i_X(X)(A\otimes_cB)\subseteq X\otimes_cB$
and, similarly, $i_Y(Y)(A\otimes_cB)\subseteq A\otimes_cY$, where $A\otimes_cB, A\otimes_cY, X\otimes_cB$ are defined 
as  the closures of the respective algebraic tensor products in $X\otimes_cY$.
\end{proof}

\begin{lemma}\label{lem-tensor}
The tensor product $(A\otimes_cB, X\otimes_c Y)$ has the following universal property:
whenever $(\varphi_X, \varphi_Y)$ is a pair of non-degenerate ccp representations $\varphi_X:X\to M(D), \varphi_Y:Y\to M(D)$ 
of $(A,X)$ and $(B,Y)$ into the multiplier algebra $M(D)$ for some $C^*$-algebra $D$
such that $\varphi_X(x)\varphi_Y(y)=\varphi_Y(y)\varphi_X(x)$ for all $x\in X$ and $y\in Y$, then 
there exists a unique ccp representation $\varphi=\varphi_X\rtimes\varphi_Y:X\otimes_cY\to M(D)$ 
of  $(A\otimes_cB, X\otimes_c Y)$such that 
$$\varphi(x\otimes y)=\varphi_X(x)\varphi_Y(y)$$
for all $x\in X, y\in Y$.
\end{lemma}

\begin{remark}\label{rem-rep}
If $H$ is a Hilbert space and $D=\mathcal K(H)$, we obtain a version 
of the above lemma for non-degenerate ccp representations on Hilbert space.
\end{remark}

\begin{proof}[Proof of Lemma \ref{lem-tensor}]
It follows from Proposition \ref{prop-universal} that there exist unique $*$-homomorphisms 
$\tilde\varphi_X:C_u^*(A,X)\to \mathcal B(H)$ and $\tilde\varphi_Y:C_u^*(B,Y)\to \mathcal B(H)$ 
such that $\tilde\varphi_X\circ j_u=\varphi_X$ and $\tilde\varphi_Y\circ j_u=\varphi_Y$.
Since $\varphi_X(x)$ commutes with $\varphi_Y(y)$ for all $x\in X, y\in Y$,
 and since $C_u^*(A,X)$ and $C_u^*(B,Y)$ are generated by $j_u(X)$ and $j_u(Y)$, respectively,
it follows that the ranges of $\tilde\varphi_X$ and $\tilde\varphi_Y$ commute as well.
Therefore, by the universal properties of the maximal tensor product, there 
exists a (unique) $*$-homomorphism $\tilde\varphi:C_u^*(A,X)\otimes_{\max}C_u^*(B,Y)\to \mathcal B(H)$ 
such that $\tilde\varphi(c\otimes d)=\tilde\varphi_X(c)\tilde\varphi_Y(d)$ for all 
$c\in C_u^*(A,X)$ and $d\in C_u^*(B,Y)$. It is then easily checked on elementary 
tensors that $\varphi(cz)=\varphi(c)\varphi(z)$ for all $c\in A\otimes_cB$ and $z\in X\otimes_cY$.
\end{proof}

\begin{remark}\label{rem-tensor}
It is an interesting question, whether  there exists a converse of the above lemma, 
i.e., whether every non-degenerate ccp representation $\pi:X\otimes_cY\to \mathcal B(K)$
can be realised as $\pi=\pi_X\rtimes\pi_Y$ for a pair of representations 
$(\pi_X,\pi_Y)$ as in the lemma. Indeed, this is only true if the representation 
$\pi$ preserves some more of the multiplicativity structure of $X\otimes_cY$, which 
is not directly part of the structure of $(A\otimes_cB, X\otimes_cY)$ as a $C^*$-operator system.
Realised as a subspace of   $D:=C_u^*(A,X)\otimes_{\max}C_u^*(B,Y)$, we see 
that an elementary tensor $x\otimes y$ can be written as a product 
$i_X(x)i_Y(y)=i_Y(y)i_X(x)$, where $i_X, i_Y$ are the canonical inclusions 
of $X$ and $Y$ into $M(X\otimes_cY)$ as in Lemma \ref{lem-factors}.
The representations constructed in Lemma \ref{lem-tensor} are precisely 
those whose extension $\bar\pi$ to $M(X\otimes_cY)$ preserves these relations:
If it does, then $\pi_X=\bar\pi\circ i_X$ and $\pi_Y=\bar\pi\circ i_Y$, satisfy the 
conditions of the lemma such that $\pi=\pi_X\rtimes\pi_Y$. 

But we believe that a general ccp representation $\pi:X\otimes_cY\to \mathcal B(K)$
does not need to satisfy these relations. But, as we see below, it does if 
$Y=B$ is a $C^*$-algebra.
\end{remark}

\begin{lemma}\label{lem-tensor-$C^*$-algebra}
Suppose that $(A,X)$ is a $C^*$-operator system and $B$ is a $C^*$-algebra
(viewed as the $C^*$-operator system $(B,B)$). 
Let $\varphi: X\otimes_cB\to M(D)$ be any non-degenerate ccp representation
of $(A\otimes_cB, X\otimes_cB)$. Then there is a unique  ccp representation
$\varphi_X:X\to M(B)$ and a $*$-representation $\varphi_B:B\to M(B)$ such that
$\varphi=\varphi_X\rtimes \varphi_B$.
A similar statement holds for the tensor product $(B\otimes_cA, B\otimes_cX)$.
\end{lemma}
\begin{proof}
Let $\varphi_X=\bar\varphi\circ i_X$ and $\varphi_B=\bar\varphi\circ i_B$ as in the above remark.
Note that $i_X$ maps $X$ into $M(X\otimes_cY)$ but  $i_B$ maps $B$ into 
$M(A\otimes_cB)\subseteq M(X\otimes_cB)$, since $(a\otimes b) i_B(c)=a\otimes bc\in A\otimes_cB$
for all $a\otimes b\in A\odot B$. 
Since the extension $\bar\varphi:M(X\otimes_cB)\to \mathcal B(K)$ is a unital c.p representation
of the $C^*$-operator system $(M(A\otimes_cB), M(X\otimes_cB))$, 
we get 
$$\varphi_X(x)\varphi_B(b)=\bar\varphi(i_X(x))\bar\varphi(i_B(b))=\bar\varphi(i_X(x)i_B(b))=\bar\varphi(x\otimes b)$$
and similarly $\varphi_B(b)\varphi_X(x)=\bar\varphi(x\otimes b)$. 
It is then clear that $\varphi=\varphi_X\rtimes\varphi_B$ as in Lemma \ref{lem-tensor}.
\end{proof}

\begin{definition}\label{def-spacial-tensor}
Suppose that $(A,X)$ and $(B, Y)$ are concrete $C^*$-operator systems on the Hilbert spaces 
$H$ and $K$, respectively. Then we define the spacial tensor product $(A\check\otimes B, X\check\otimes Y)$ 
via the closures of the algebraic tensor products $A\odot B$ and $X\odot Y$ in $B(H\otimes K)$.
\end{definition}

It is well known that the spatial tensor product does not depend, up to isomorphism,
on the particular embeddings of $X$ in $\mathcal B(H)$ and $Y$ in $\mathcal B(K)$. 
Let $i_X^s:X\to B(H\otimes K), i_X^s(x)=x\otimes 1_K$ and $i_Y^s:Y\to B(H\otimes K); i_Y^s(y)=1_H\otimes y$
denote the canonical embeddings of $X$ and $Y$ into $B(H\otimes K)$. It then follows from 
Lemma \ref{lem-tensor} that there exists a canonical surjective  morphism
$$\Phi:=i_X\times i_Y: X\otimes_cY\to X\check\otimes Y $$
from $(A\otimes_cB, X\otimes_cY)$ onto $(A\check\otimes B, X\check\otimes Y)$ 
.
The following proposition is now an easy consequence of our constructions:

\begin{proposition}\label{prop-nuclear}
Suppose that $(A,X)$ is a $C^*$-operator system. Then for any nuclear $C^*$-algebra $B$,
the canonical morphism $(A\otimes_cB, X\otimes_cB)$ onto $(A\check\otimes B, X\check\otimes B)$  
is an isomorphism (and similarly for $(B\otimes_cA, B\otimes_cX)$).
\end{proposition}
\begin{proof}
If $B$ is nuclear, then 
$C_u^*(A,X)\otimes_{\max}B=C_u^*(A,X)\check\otimes B$.
The result then follows from representing $C_u^*(A,X)$ (and hence $X$)
faithfully on a Hilbert space $H$.
\end{proof}

For later use we  also need to consider morphisms into the multiplier $C^*$-operator systems
of tensor products of $C^*$-operator systems with $C^*$-algebras. This 
is the special case of the above constructions if one of the factors is a pair $(C,C)$ 
for a  $C^*$-algebra $C$. Note that in this case we also have $C_u^*(C,C)=C$.

\begin{lemma}\label{lem-tensormor}
Suppose that $(A,X)$ and $(B,Y)$ are $C^*$-operator systems and let $C$ and $D$ be $C^*$-algebras.
Let $\varphi_X:X\to  M(Y)$ be a non-degenerate generalized homomorphism from $(A,X)$ into $(M(B), M(Y))$ 
and let $\varphi_C:C\to M(D)$ be a non-degenerate generalized homomorphism from $C$ to $D$.
Then there exists a unique non-degenerate generalized homomorphism
$$\varphi\otimes_c\psi: X\otimes_cC\to M(Y\otimes_cD)$$
(resp. $\varphi\check\otimes\psi: X\check\otimes C\to M(Y\check\otimes D)$) such that
$\varphi\otimes_c\psi(x\otimes c)=\varphi(x)\otimes\psi(c)$ (resp. $\varphi\check\otimes\psi(x\otimes c)=\varphi(x)\otimes\psi(c)$)
for all elementary tensors $x\otimes c\in X\odot C$.
\end{lemma}
\begin{proof}
Let $\pi:Y\otimes_cD\to \mathcal B(H)$ be a non-degenerate completely isometric representation of $(B\otimes_cD, Y\otimes_cD)$   on the Hilbert space $H$,
By Lemma \ref{lem-tensor-$C^*$-algebra} there are non-degenerate representations $\pi_Y:Y\to \mathcal B(H)$ and $\pi_D:D\to \mathcal B(H)$ such that 
$\pi=\pi_Y\otimes_c\pi_D$. Let $\psi_X:=\bar\pi_Y\circ \varphi_X$ and $\psi_C:=\bar\pi_D\circ \varphi_A$, where 
$\bar\pi_Y$ and $\bar\pi_D$ denote the unique extensions of $\pi_Y, \pi_D$ to $M(Y)$ and $M(D)$ as in 
Lemma \ref{lem-non-deg-morphism}. Note that for all $m\in M(Y), n\in M(D)$ we have
$$\bar\pi_Y(m)\bar\pi_D(n)= \bar\pi_D(n)\bar\pi_Y(m).$$
Indeed, this follows from the fact that $\pi(B\odot D)H$ is dense in $H$ (since $\pi$ is non-degenerate) and 
 for all $b\in B$ and $d\in D$, we have 
\begin{align*}
\bar\pi_Y(m)\bar\pi_D(n)\pi(b\otimes d)&= \bar\pi_Y(m)\bar\pi_D(n)\pi_Y(b)\pi_D(d)=\bar\pi_Y(m)\bar\pi_D(n)\pi_D(d)\pi_Y(b)\\
&=\bar\pi_Y(m)\pi_D(nd)\pi_Y(b)=\bar\pi_Y(m)\pi_Y(b)\pi_D(nd)\\
&=\pi_Y(mb)\pi_D(nd)
=\pi_D(nd)\pi_Y(mb)\\
&=\bar\pi_D(n)\pi_D(d)\pi_Y(mb)=\bar\pi_D(n)\pi_Y(mb)\pi_D(d)\\
&=\bar\pi_D(n)\bar\pi_Y(m)\pi_Y(b)\pi_D(d)=\bar\pi_D(n)\bar\pi_Y(m)\pi(b\otimes d).
\end{align*}
It follows from this that $\psi_X(x)\psi_C(c)=\psi_C(c)\psi_X(x)$ for all $x\in X, c\in C$. Thus, it follows from 
Lemma \ref{lem-tensor} that there exists a unique non-degenerate ccp representation 
$\psi:=\psi_X\rtimes\psi_C: X\otimes_cC\to \mathcal B(H)$ given on elementary tensors by 
$\psi(x\otimes c)=\psi_X(x)\psi_C(c)$. 

Now, by the construction of the multiplier system as in Lemma \ref{lem-multiplier} we may identify $M(Y\otimes_cD)$ 
with its image $\bar\pi(M(Y\otimes_cY)) \subseteq \mathcal B(H)$. Using this identification, we want to check that 
$\psi$ takes values in 
$$M(Y\otimes_cY)\cong\{m\in \mathcal B(H): m \pi(B\otimes_cD), \pi(B\otimes_cD) m\subseteq \pi(Y\otimes_{c}D).$$
For this  let $x\otimes c\in X\odot C$ be any elementary tensor and let $b\otimes d\in B\odot D$.
Then 
\begin{align*}
\psi(x\otimes c)\pi(b\otimes d)&=\bar\pi_Y(\varphi_X(x))\bar\pi_D(\varphi_C(c))\pi_Y(b)\pi_D(d)\\
&=\bar\pi_Y(\varphi_X(x))\pi_D(\varphi_C(c)d)\pi_Y(b)=\pi_Y(\varphi_X(x)b)\pi_D(\varphi_C(c)d)\\
&=\pi(\varphi_X(x)b\otimes \varphi_C(c)d)\in \pi(Y\otimes_cD).
\end{align*}
Hence $\psi(X\otimes_cC)\pi(B\otimes_cD)\subseteq \pi(Y\otimes_cD)$ and the inclusion 
 $\pi(B\otimes_cD)\psi(X\otimes_cC)\subseteq \pi(Y\otimes_cD)$ follows similarly.
 \end{proof}

\section{Universal crossed products by group actionns}\label{crossedproducts}
For a $C^*$-operator system $(A,X)$ let  $\Aut(A,X)$ denote the group of 
all invertible morphisms  $\alpha:(A,X)\to (A,X)$.
A strongly continuous action of the locally compact group $G$
on the $C^*$-operator system $(A,X)$ is a homomorphism $\alpha:G\to \Aut(A,X); g\mapsto\alpha_g$ such that
$g\mapsto \alpha_g(x)$ is continuous for all $x\in X$. 

It follows directly from the universal property of $(C^*_u(A,X), j_u)$  
that every automorphism of $(A,X)$ extends to a unique
automorphisms $\alpha^u$ 
of $C^*_u(A,X)$.
Since $C^*_u(A,X)$   is generated by a copy of $X$, any strongly 
continuous action $\alpha:G\to \Aut(A,X)$ extends to unique strongly continuous action
$\alpha^u:G\to \Aut(C_u^*(A,X))$.
This leads  to the 
following definition of the universal  crossed product
by an action $\alpha$ of $G$ on $(A,X)$.

\begin{definition}\label{def-crossed-product}
Let $\alpha:G\to \Aut(A,X)$ be an action as above. We define the {\em universal (or full)
crossed product} $(A,X)\rtimes_{\alpha}^uG$ for the action $\alpha$ as the 
pair $(A\rtimes_{\alpha}^uG, X\rtimes_{\alpha}^uG)$, where 
$A\rtimes_{\alpha}^uG$ and $X\rtimes_{\alpha}^uG$ are the closures of
$C_c(G,A)$ and $C_c(G,X)$ inside the universal $C^*$-algebra crossed product 
$C_u^*(A,X)\rtimes_{\alpha,u}G$.
%
%
%
\end{definition}

\begin{remark}\label{rem-crossed}
{\bf (a)} If $\alpha:G\to\Aut(A)$ is an action of $G$ on the $C^*$-algebra $A$, and if we consider the
corresponding $C^*$-operator system $(A,A)$, then $C_u^*(A,A)=A$, and therefore
the crossed product $(A,A)\rtimes_\alpha^uG$ is 
given by the pair $(A\rtimes_{\alpha,u}G, A\rtimes_{\alpha,u}G)$. 
 Thus, the universal crossed product construction for $C^*$-operator systems extends the 
well-known  universal crossed product constructions for $C^*$-algebras.

{\bf (b)} In general it is not true that that in the crossed product  
$(A,X)\rtimes_{\alpha}^uG=(A\rtimes_{\alpha}^uG, X\rtimes_{\alpha}^uG)$
the $C^*$-algebra $A\rtimes_{\alpha}^uG$ coincides with the universal $C^*$-algebra crossed 
product $A\rtimes_{\alpha, u} G$.
To see 
this let $G$ be any (second countable) non-amenable exact group. 
Then it follows from \cite{BCL} that there exists an amenable compact $G$-space $\Omega$, 
which implies that the full and reduced crossed 
products of $G$ by $C(\Omega)$ coincide. 
Now choose a faithful and non-degenerate representation of $C(\Omega)$ into $ \mathcal B(H)$ 
for some Hilbert space $H$ and consider $X:=C(\Omega)\subseteq  \mathcal B(H)$ as an operator system
(forgetting the multiplicative structure). As in Example \ref{excstaros} we regard this as the 
$C^*$-operator system $(\C, X)$.
Let $C_u^*(X)$ denote the enveloping $C^*$-algebra of $X$. We then get 
completely isometric embeddings $\C\into X\into C_u^*(X)$ which give rise 
to ccp maps between the full crossed products (in the $C^*$-algebra sense)
\begin{equation}\label{eq-comp}
C^*(G)=\C\rtimes_u G\to C(\Omega)\rtimes_u G\to C_u^*(X)\rtimes_u G.
\end{equation}
By definition of the crossed product $(\C, C(X))\rtimes_{\alpha^u}G=(\C\rtimes_\alpha^uG, X\rtimes_\alpha^uG)$,
$X\rtimes_{\alpha}^uG$ is identical to the image of $C(\Omega)\rtimes_u G$ under the second map 
and $\C\rtimes_\alpha^uG$ coincides with the image of $C^*(G)$
under the composition in (\ref{eq-comp}). But since
$C(\Omega)\rtimes_u G=C(\Omega)\rtimes_rG$ by amenability of the action of $G$ on $\Omega$
 the first map in (\ref{eq-comp}) factors through
the reduced group algebra $C_r^*(G)$. We therefore also have
 $\C\rtimes_\alpha^uG\cong C_r^*(G)\neq C^*(G)=\C\rtimes_u G$.
%
%
%
%
%
%
%
%

\end{remark}

In what follows we want to show that non-degenerate representations of  the full crossed product are
in one-to-one relation to the non-degenerate covariant representations of the system $(A,X, G,\alpha)$ as in 

\begin{definition}\label{def-covariant}
Suppose that $\alpha:G\to \Aut(A,X)$ is an action of $G$ on the $C^*$-operator system $(A,X)$.
A covariant representation of $(A,X,G,\alpha)$ is a pair $(\pi, u)$, where 
$\pi:X\to \mathcal B(H_\pi)$ is a ccp representation of $(A,X)$ on $\mathcal B(H_\pi)$ and 
$u:G\to U(H_\pi)$ is a strongly continuous unitary representation of $G$ such that
$$\pi(\alpha_g(x))=U_g\pi(x)U_g^*\quad \forall x\in X, g\in G.$$
\end{definition}

\begin{remark}\label{rem-regular-rep}
Suppose that $\rho:X\to \mathcal B(K)$ is any ccp representation of $(A,X)$ on $\mathcal B(K)$.
Then we can construct a covariant 
representation $\Ind\rho=(\tilde{\rho}, 1_K\otimes \lambda)$ on $\mathcal B(K\otimes L^2(G))$
by the usual formula
$$\big(\tilde{\rho}(x)\xi\big)(g)=\rho(\alpha_{g^{-1}}(x))\xi(g)\quad\text{and}\quad (1\otimes\lambda)_h\xi(g)=\xi(h^{-1}g)$$
for $\xi\in L^2(G, K)\cong K\otimes L^2(G)$, $x\in X$, and $g,h\in G$. 
Observe that if $\rho$ is completely isometric, then so is $\tilde\rho$. Hence there exist covariant representations 
$(\pi,u)$ of $(A,X, G,\alpha)$ 
in which the representation $\pi$ is completely isometric.
\end{remark}

It is actually useful to extend the notion of a covariant representation to allow representations into multiplier systems as in

\begin{definition}\label{def-covariant-mult}
Suppose that $\alpha:G\to \Aut(A,X)$ is an action of $G$ on the $C^*$-operator system $(A,X)$
and suppose that $(B,Y)$ is a $C^*$-operator system. By a non-degenerate covariant homomorphism 
of $(A,X,G,\alpha)$ into the multiplier system $(M(B), M(Y))$ of $(B,Y)$ we understand a pair of maps $(\varphi, u)$,
where $\varphi:X\to M(Y)$ is a non-degenerate generalized morphism from $(A,X)$ to $(B,Y)$
and $u:G\to UM(B)$ is a strictly continuous homomorphism such that
$$\varphi(\alpha_g(x))=u_g \varphi(x) u_g^*$$
for all $x\in X$ and $g\in G$.
\end{definition}

\begin{remark}
Note that if  $(B,Y)$ is represented completely isometrically and non-degenerately on a Hilbert  space 
$K$, then a non-degenerate covariant homomorphism of $(A,X,G,\alpha)$ into  $(M(B), M(Y))$
 turns into a non-degenerate
covariant representation 
of  $(A,X,G,\alpha)$  on $K$. But being a representation into $(M(B), M(Y))$ requires 
some additional structure of how the image interacts with $(B,Y)$.
\end{remark}

\begin{example}\label{ex-regular}
If $\alpha:G\to \Aut(A,X)$ is an action of $G$ on the $C^*$-operator system $(A,X)$ we can define a canonical
non-degenerate covariant homomorphism $(\Lambda_X, \Lambda_G)$ of $(A,X,G,\alpha)$ into 
 $M(X\check\otimes \mathcal K(L^2(G)))$ as follows:
We first define a non-degenerate representation $\Lambda_X:X\to M(X\otimes \mathcal K(L^2(G)))$ of $(A,X)$
by the composition of maps
$$X\stackrel{\tilde\alpha}{\longrightarrow} M(X\check\otimes C_0(G)) \stackrel{\id_X\otimes M}{\longrightarrow} M(X\check\otimes \mathcal K(L^2(G))),$$
where $\tilde\alpha: X\to C_b(G,X)\subseteq M(X\check\otimes C_0(G))$ sends the element $x\in X$ to the 
bounded continuous function $g\mapsto\alpha_{g^{-1}}(x)$ and where $M:C_0(G)\to \mathcal B(L^2(G))$ is the representation 
of $C_0(G)$ by multiplication operators. 
Moreover, we define $\Lambda_G:G\to M(X\check\otimes \mathcal K(L^2(G)))$ 
by $\Lambda_G(g)=1\otimes \lambda_g$, where $\lambda:G\to U(L^2(G))$ is the regular representation of $G$.
We leave it to the reader to check that $(\Lambda_X, \Lambda_G)$ satisfies the covariance condition.
We call $(\Lambda_X, \Lambda_G)$ the {\em regular representation} of $(A,X,G,\alpha)$.

Note that this construction directly extends the construction of the regular representation $(\Lambda_B, \Lambda_G)$
of a $C^*$-dynamical  system $(B,G,\beta)$ into $M(B\otimes\K(L^2(G)))$. In particular,  the restriction $(\Lambda_X|_A. \Lambda_G)$
of $(\Lambda_X, \Lambda_G)$ to $(A,G,\alpha)$ coincides with the regular representation of $(A,G,\alpha)$.

Note also that if $\rho:X\to \mathcal B(K)$ is any ccp representation of $(A,X)$, we recover the representation 
$\Ind\rho=(\tilde\rho, 1\otimes \lambda)$ of Remark \ref{rem-regular-rep} as the composition 
$(\rho\otimes \id_{\mathcal K(L^2(G)})\circ (\Lambda_X, \Lambda_G)$.
\end{example}

\begin{proposition}\label{prop-covariant}
For each non-degenerate covariant representation $(\varphi, u)$ of $(A,X,G,\alpha)$ into $(M(B), M(Y))$
there exists a unique generalized homomorphism 
$\varphi\rtimes u:X\rtimes_\alpha^uG\to M(Y)$ from  $(A\rtimes_{\alpha}^uG, X\rtimes_\alpha^uG)$ to $(M(B),M(Y))$  
 given on $f\in C_c(G,X)$ by
$$\varphi\rtimes u(f)=\int_G \varphi(f(g))u_g\, dg.$$
\end{proposition}
\begin{proof} 
Let $(\varphi,u)$ be given and let $(B,Y)$ be represented completely isometrically on a Hilbert space $K$.
By Lemma \ref{lem-universal} there exists a unique $*$-representation
$\tilde\varphi:C_u^*(A,X)\to C^*(\varphi(X))\subseteq \mathcal B(K)$ which extends $\varphi$. Applying this fact to the 
representations $\varphi\circ \alpha_g=\Ad u_g\circ \varphi$ shows that 
$(\tilde\varphi, u)$ is a covariant representation of the $C^*$-dynamical system
$(C_u^*(A,X), G, \alpha^u)$ into $\mathcal B(K)$. It therefore integrates to a $*$-representation
$\tilde\varphi\rtimes u: C_u^*(A,X)\rtimes_{\alpha^u,u}G\to \mathcal B(K)$  given on $f\in C_c(G, C_u^*(A,X))$ by 
the integral formula in the lemma.
The restriction of $\tilde\varphi\rtimes u$ to $X\rtimes_{\alpha}^uG$ is then 
the desired representation $\varphi\rtimes u$. To see that it maps into $M(Y)$, we only need to check 
that $\big(\varphi\rtimes u(f)\big)b, b\big(\varphi\rtimes u(f)\big)\in Y$ for all $b\in B$ and $f\in C_c(G,X)$.
But the integration formula gives
\begin{equation}\label{eq-integral}
\big(\varphi\rtimes u(f)\big)b=\int_G f(g)u_g b\,dg
\end{equation}
Since $u_g\in M(B)$, we have $u_gb\in B$ and hence $f(g)u_gb\in Y$ for all $g\in G$.
Thus the integral  (\ref{eq-integral})  gives an element in $Y$.
A similar argument shows that to $b\big(\varphi\rtimes u(f)\big)\in Y$.
\end{proof}

%

\begin{lemma}\label{lem-multiplier}
Suppose that $\alpha:G\to \Aut(A,X)$ is an action. 
Then there is a canonical covariant morphism 
$$(i_X, i_G): (A,X, G, \alpha)\to \big(M(A\rtimes_\alpha^u G), M(X\rtimes_\alpha^u G)\big)$$
 such that for each 
$x\in X$,  $f_1\in C_c(G,A)$, $f_2\in C_c(G,X)$ and $g,h\in G$, we have
$$(i_X(x)f_1)(g)=xf_1(g)\quad\text{}\quad (f_1 i_X(x))(g)= f_1(g)\alpha_g(x)$$
and 
 $$(i_G(h)f_2)(g)=\alpha_h(f_2(h^{-1}g))\quad\text{}\quad (f_2i_G(h))(g)=f_2(gh^{-1})\Delta(h^{-1}).$$
Moreover, the integrated form $i_X\rtimes i_G: X\rtimes_\alpha^u G\to M(X\rtimes_\alpha^uG)$
is the identity map on $X\rtimes_\alpha^uG$.
\end{lemma}
\begin{proof} Suppose that $C_u^*(A,X)\rtimes_{\alpha^u,u}G$ is represented faithfully and nonde\-gene\-rately 
on a Hilbert-space $H_u$. Then the restriction to $X\rtimes_\alpha^uG$ gives a completely isometric 
representation of $(A,X)\rtimes_{\alpha}^uG$ on $H_u$ as well. 
By the universal properties of the maximal crossed product  $C_u^*(A,X)\rtimes_{\alpha^u,u}G$ there
exists a unique covariant homomorphism $(i_{C_u^*(A,X)}, i_G)$ of $(C_u^*(A,X), G, \alpha^u)$ 
into $M(C_u^*(A,X)\rtimes_{\alpha^u,u}G)\subseteq B(H_u)$ which is given for $b\in C_u^*(A,X)$, $g,h\in G$ and 
$f\in C_c(G, C_u^*(A,X))$ by the formulas
$$(i_{C_u^*(A,X)}(b)f)(g)=bf (g)\quad\quad (f i_{C_u^*(A,X)}(b))(g)= f(g)\alpha_g(b)$$
and 
$$(i_G(h)f)(g)=\alpha_h(f(h^{-1}g))\quad\quad  (fi_g(h))(g)=f(gh^{-1})\Delta(h^{-1}),$$
and such that the integrated form $i_{C_u^*(A,X)}\rtimes i_G$ coincides with the original 
representation of the crossed product on $H_u$. Let $i_X=i_{C_u^*(A,X)}\circ j_u$, where 
$j_u:X\to C_u^*(A,X)$ denotes the embedding.  Then for each $f\in C_c(G,X)$ the 
integral $i_X\rtimes i_G(f)=\int_G i_X(f(g))i_G(g)\,dg$ coincides with the inclusion of 
$f\in C_c(G,X)\into  X\rtimes_\alpha^uG\subseteq C_u^*(A,X)\rtimes_{\alpha^u,u}G$, and therefore 
extends to the identity on $X\rtimes_\alpha^uG$. 

Thus we only need to check that $(i_X,i_G)$ is a non-degenerate covariant morphism into 
$\big(M(A\rtimes_\alpha^uG), M(X\rtimes_\alpha^uG)\big)$.
First, if $f\in C_c(G,A)$ and $h\in G$, then $g\mapsto (i_G(h)f)(g)=\alpha_h(f(h^{-1}g))$
lies in $C_c(G,A)$ as well, hence 
$i_G(h)\big(i_X\rtimes i_G(f)\big)=i_X\rtimes i_G([g\mapsto \alpha_h(f(h^{-1}g))])\in A\rtimes_\alpha^uG$,
which shows that $i_G$ takes its values in $M(A\rtimes_\alpha^uG)$. 
Moreover, if $f\in C_c(G,A)$ and $x\in X$, then $[g\mapsto (i_X(x)f)(g)=xf(g)]\in C_c(G,X)$, 
and hence 
$i_X(x)(i_X\rtimes i_G(f))=i_X\rtimes i_G([g\mapsto xf(g)])\in X\rtimes_\alpha^uG$, 
which implies that $i_X(x)\in M(X\rtimes_\alpha^uG)$. 
This completes the proof.
\end{proof}

We are now ready for the converse of Proposition \ref{prop-covariant}.

\begin{proposition}\label{prop-covariant1}
Suppose that $\alpha:G\to \Aut(A,X)$ is an action. Then for each 
non-degenerate generalized morphism  $\Phi:X\rtimes_\alpha^uG\to M(Y)$
from $(A\rtimes_\alpha^uG, X\rtimes_\alpha^uG)$ to $(B,Y)$ there exists a unique 
non-degenerate covariant morphism $(\varphi, u)$ of $(A,X,G,\alpha)$
to $(M(B), M(Y))$ such that $\Phi=\varphi\rtimes u$. 

More precisely, if $(i_X, i_G)$ is the covariant morphism of $(A,X,G,\alpha)$
into $(M(A\rtimes_\alpha^uG), M(X\rtimes_\alpha^uG))$ as in Lemma \ref{lem-multiplier} and if
$\bar\Phi:M(X\rtimes_\alpha^uG)\to M(Y)$ denotes the unique extension of $\Phi$
as in Lemma \ref{lem-non-deg-morphism}, then 
$$\varphi=\bar\Phi\circ i_X\quad\text{and}\quad u=\bar\Phi\circ i_G.$$
\end{proposition}

\begin{proof} It is straightforward to check on functions in $C_c(G,X)$ 
that $(\varphi, u)$ is a covariant morphism of $(A,X,G,\alpha)$
into $(M(B), M(Y))$ such that 
$$\varphi\rtimes u=\bar\Phi\circ (i_X\rtimes i_G)=\bar\Phi\circ \id_{X\rtimes^uG}=\Phi.$$
\end{proof}

Note that we may regard covariant morphism of $(A,X, G,\alpha)$ 
into $M(B)$ for any $C^*$-algebra $B$ as the special case of 
covariant morphisms into the multiplier system of the 
$C^*$-operator system $(B,B)$. Similarly, we may regard
covariant representations on a Hilbert space $H$ as the special 
case in which $B=\mathcal K(H)$. Thus Proposition \ref{prop-covariant}
will give us

\begin{corollary}\label{cor-covariant}
Suppose that $(\varphi, u)$ is a covariant morphism of $(A,X,G,\alpha)$
into $M(D)$ (resp. covariant represention into $\mathcal B(K)$ for a Hilbert space $K$). 
Then there exist an integrated form 
$$\varphi\rtimes u: X\rtimes_\alpha^uG\to M(D)\quad \text{(resp. $\mathcal B(K)$)}$$
given for $f\in C_c(G,X)$ by $\varphi\rtimes u(f)=\int_G \varphi(f(g))u_g\,dg$.

Conversely, if $\Phi: X\rtimes_\alpha^uG\to M(D)$ (resp. $\mathcal B(K)$) is any non-degenerate
homomorphism (resp. representation) of $(A\rtimes_\alpha^uG, X\rtimes_\alpha^uG)$, then 
there exists a unique non-degenerate covariant homomorphism (resp. representation)
of $(A,X,G,\alpha)$ into $M(D)$ (resp. $\mathcal B(K)$) such that $\Phi=\varphi\rtimes u$.
Indeed, if $(i_X, i_G)$ are as in Lemma \ref{lem-multiplier} and $\bar\Phi$ 
denotes the unique extension of $\Phi$ to $M(X\rtimes^uG)$, then 
$$\varphi= \bar\Phi\circ i_X\quad\text{and}\quad u=\bar\Phi\circ i_G.$$
\end{corollary}

\begin{corollary}\label{cor-universal}
Let $(A,X, G,\alpha)$ be a $G$-$C^*$-operator system. Then 
$$C_u^*(A,X)\rtimes_{\alpha,u} G= C_u^*(A\rtimes_\alpha^uG, X\rtimes_\alpha^uG).$$
\end{corollary}
\begin{proof} Since $C_u^*(A,X)$ is generated by $X$, it is fairly straightforward to check that 
$C_u^*(A,X)\rtimes_{\alpha_u,u}G$ 
is generated by $X\rtimes_\alpha^uG\subseteq C_u^*(A,X)\rtimes_{\alpha_u,u}G$. 
So all we need to show is that  any isometric representation
$\psi:X\rtimes_{\alpha}^uG\into \mathcal B(H)$ of $(A\rtimes_\alpha^uG, X\rtimes_\alpha^uG)$ 
extends to a $*$-homomorphism $\Psi: C_u^*(A,X)\rtimes_{\alpha^u,u}G\to \mathcal B(H)$.
Passing to the closed subspace $\varphi(X)H=\varphi(A)H$ if necessary, we may assume without loss of generality 
that $\varphi$ is non-degenerate. It follows then from Corollary  \ref{cor-covariant} that there exists 
a covariant representation $(\varphi, u)$ of $(A,X, G,\alpha)$ such that $\psi=\varphi\rtimes u$. 
Then, as in the proof of Proposition \ref{prop-covariant}, we see that $(\varphi, u)$ extends uniquely  to a covariant
representation $(\bar\varphi, u)$ of $(C_u^*(A,X), G, \alpha^u)$ such that $\psi=\varphi\rtimes u: X\rtimes_\alpha^uG\to \mathcal B(H)$
coincides with the restriction to $X\rtimes_\alpha^uG$  of the integrated form $\bar\varphi\rtimes u: C_u^*(A,X)\rtimes_{\alpha^u,u}G\to \mathcal B(H)$.
This finishes the proof.
\end{proof}

\section{Reduced crossed products}\label{sec-reduced-crossed}
We now turn our attention to the construction of the 
reduced crossed product $(A\rtimes_\alpha^rG, X\rtimes_\alpha^rG)$ 
by an action $\alpha:G\to \Aut(A,X)$ of a group on a $C^*$-operator system $(A,X)$.
Indeed, we define the reduced crossed product  via the regular representation 
$(\Lambda_A, \Lambda_X)$ of $(A,X,G,\alpha)$ into $M(A\otimes\K(L^2(G), X\otimes \K(L^2(G))$ as 
constructed  in Example \ref{ex-regular}:

\begin{definition}\label{defn educed crossed product}
Let $\alpha:G\to \Aut(A,X)$ be an action of $G$ on the $C^*$-operator system $(A,X)$.
Then we define the reduced crossed product as the image 
$$(A,X)\rtimes_\alpha^rG=
(A\rtimes_\alpha^rG, X\rtimes_\alpha^rG):=\big(\Lambda_A(A\rtimes_\alpha^uG), \Lambda_X(X\rtimes_\alpha^uG)\big)$$
 of the universal crossed product by the regular representation 
$(\Lambda_A, \Lambda_X)$ inside $\big((M(A\otimes \K(L^2(G)), M(X\otimes \K(L^2(G))\big)$.
\end{definition}

For the following proposition recall from Remark \ref{rem-regular-rep} the construction 
of the 
representation $\Ind\rho:=(\tilde{\rho}, 1_K\otimes \lambda)$ on $\mathcal B(K\otimes L^2(G))$
given by the  formula
$$\big(\tilde{\rho}(x)\xi\big)(g)=\rho(\alpha_{g^{-1}}(x))\xi(g)\quad\text{and}\quad (1\otimes\lambda)_h\xi(g)=\xi(h^{-1}g),$$ 
where $\rho:X\to \mathcal B(K)$ is any given representation of $(A,X)$ on some Hilbert space $K$.
We call $\Ind\rho$ the covariant representation of $(A,X,G,\alpha)$ induced by $\rho$.
As an easy consequence of our definition of reduced crossed product 
and the discussion at the end of Example \ref{ex-regular} we obtain

\begin{proposition}\label{prop-reduced}
Let $\alpha:G\to \Aut(A,X)$ be an action of $G$ on the $C^*$-operator system $(A,X)$.
Then for every non-degenerate representation $\rho:X\to \mathcal B(K)$ of $(A,X)$
 the  integrated form $\tilde\rho\rtimes (1\otimes \lambda): X\rtimes_\alpha^u G\to B(K\otimes L^2(G))$
of the induced representation 
$\Ind\rho=(\tilde\rho, 1\otimes \lambda)$ 
(which we shall also simply denote by $\Ind\rho$) factors through the reduced crossed product 
$X\rtimes_\alpha^rG$ to give a representation of $(A,X)\rtimes_\alpha^rG$ into $B(K\otimes L^2(G))$. Moreover, if 
$\rho$ is completely isometric, then so is $\Ind\rho$.
\end{proposition}
\begin{proof}
By the discussion at the end of  Example \ref{ex-regular} we have the identity.
$\Ind\rho=(\rho\otimes\id_{\K(L^2(G))})\circ (\Lambda_A, \Lambda_X)$
as a representation from $(A,X)\rtimes_\alpha^uG$ to $\mathcal B(K\otimes L^2(G))$.
Therefore $\Ind\rho$ clearly factors through $(A,X)\rtimes_\alpha^rG=
\big(\Lambda_A(A\rtimes_\alpha^uG), \Lambda_X(X\rtimes_\alpha^uG)\big)$.
If $\rho:X\to \mathcal B(K)$ is completely isometric, the same holds for
 $\rho\otimes\id_{\K(L^2(G))} :(A\otimes \K(L^2(G)), X\otimes L^2(G))\to \mathcal B(K\otimes L^2(G))$ 
 and its extension to the multiplier system $\big(M(A\otimes \K(L^2(G)), M(X\otimes L^2(G))\big)$  (see Lemma \ref{lem-multipliers}).
Therefore $\Ind\rho$ factors through a completely isometric representation of $(A,X)\rtimes_\alpha^rG$ into $\B(K\otimes L^2(G))$
as claimed.
\end{proof}

\begin{remark}\label{rem-reduced} 
{\bf (a)}  It follows in particular from the above proposition that, up to
completely isomeric isomorphism, the reduced crossed product $(A,X)\rtimes_\alpha^rG$ does not 
depend on the particular representation of $(A,X)$ on a Hilbert space $H$.

{\bf (b)} Suppose  that $(C, j)$ is any $C^*$-hull of $(A,X)$ such that a given 
action $\alpha:G\to\Aut(A,X)$ extends to an action $\alpha:G\to \Aut(C)$. 
Let $\rho:C\to \B(K)$ be any faithful and non-degenerate representation on a Hilbert space $K$.
Then $\rho_X:=\rho\circ j:X\to \mathcal B(K)$ is a non-degenerate  completely isometric representation 
of $(A,X)$ into $\mathcal B(K)$. Then the regular representation 
$\Ind\rho: C\rtimes_{\alpha,u}G\to \mathcal B(K\otimes L^2(G))$
factors through a faithful representation of 
the $C^*$-reduced crossed product $C\rtimes_{\alpha,r}G$ whose composition 
with the canonical inclusion of $(C_c(G,A), C_c(G,X))$ into $C\rtimes_{\alpha,r}G$ 
coincides the completely isometric representation $\Ind\rho_X: (A,X)\rtimes_\alpha^rG\to \mathcal B(K\otimes L^2(G))$
on the dense subsystem $(C_c(G,A), C_c(G,X))$.
It follows that, up to a completely isometric isomorphism, the reduced crossed product 
$(A,X)\rtimes_{\alpha}^rG$ can be identified with the closure of the pair 
$(C_c(G,A), C_c(G,X))$  inside $C\rtimes_{\alpha,r}G$. 

This observation applies in  particular to the 
 universal $C^*$-hull $(C_u^*(A,X), j_u)$ and the enveloping $C^*$-algebra $(C_{\env}^*(A,X), j_{\env})$,
 where it easily follows from the respective universal properties that 
 every action on $(A,X)$ extends to actions on $C_u^*(A,X)$ and $C_{\env}^*(A,X)$, respectively.

\end{remark}

\section{Crossed products by coactions}\label{sec-coaction}

Recall (e.g. from \cite[Appendix A]{EKQR}) 
that if $G$ is a locally compact group, there is a canonical comultiplication 
$\delta_G:C^*(G)\to M(C^*(G) \check\otimes C^*(G))$ on $C^*(G)$  
which is given as the integrated form of the unitary representation $g\mapsto u_g\otimes u_g\in
UM(C^*(G)\check\otimes C^*(G))$, where $u:G\to UM(C^*(G))$ is the canonical representation of $G$ 
into $UM(C^*(G))$. A {\em coaction} of $G$ on a $C^*$-algebra $A$ is an injective 
non-degenerate $*$-homomorphism $\delta:A\to M(A\check\otimes C^*(G))$  such that the 
following conditions hold:
\begin{enumerate}
\item  $\delta(A)(1\check\otimes C^*(G))\subseteq  A\check\otimes C^*(G)$. 
\item The following diagram of maps commutes
$$\begin{CD}
A @>\delta>> M(A\check\otimes C^*(G))\\
@V\delta VV   @VV \id_A\otimes \delta_GV\\
M(A\check\otimes C^*(G))  @>> \delta\otimes \id_{G}> M(A\check\otimes C^*(G)\check\otimes C^*(G))
\end{CD}
$$
\end{enumerate}
where $\id_G$ denotes the identity on $C^*(G)$. If, in addition, we have the identity
\begin{itemize}
\item[(1$'$)] $\delta(A)(1\check\otimes C^*(G))= A\check\otimes C^*(G)$
\end{itemize}
then the coaction $\delta$ is called {\em non-degenerate}. Note that condition 
(1$'$) is automatic if $G$ is amenable or discrete (see \cite[Proposition 6]{Kat}, \cite[Lemma 3.8]{Land} and \cite{BS}).

%

We are now going to extend the definition of a coaction of $C^*(G)$  to the category of $C^*$-operator systems.

\begin{definition}\label{def coaction}
Let $G$ be a locally compact group. A coaction of $G$ on the $C^*$-operator system $(A,X)$ 
is an injective  non-degenerate generalized morphism  
$$\delta_X: (A,X) \to \big(M(A\check\otimes C^*(G)), M(X\check\otimes C^*(G))\big)$$
such that the following holds:
\begin{enumerate}
\item The map $\delta_A:=\delta_X|_A: A\to M(A\check\otimes C^*(G))$ is a coaction of $C^*(G)$ on $A$.
\item The following diagram of maps commutes
$$\begin{CD}
X@>\delta_X>> M(X\check\otimes C^*(G))\\
@V\delta_X VV   @VV \id_X\otimes \delta_GV\\
M(X\check\otimes C^*(G))  @>> \delta_X\otimes \id_{G}> M(X\check\otimes C^*(G)\check\otimes C^*(G))
\end{CD}
$$
\end{enumerate}
\end{definition}

\begin{remark}\label{rem nondeg}
{\bf (a)} Notice that condition (1) always implies that 
 \begin{align*}
 \delta_X(X)(1\check\otimes C^*(G))&=\delta_X(XA)(1\check\otimes C^*(G))= \delta_X(X)\big(\delta_X(A)(1\check\otimes C^*(G))\big)\\
 &\subseteq \delta_X(X)(A\check\otimes C^*(G))= X\check\otimes C^*(G),
 \end{align*}
 where the last equation follows from the nondegeneracy of $\delta_X\to M(X\check\otimes C^*(G))$.
    
  {\bf (b)} Let $1_G:C^*(G)\to \C$ denote the integrated form of the trivial representation of $G$. Then 
  it follows from the definition of a coaction $\delta_X:(A,X)\to M(X\check\otimes C^*(G))$ that 
  $(\id_X\otimes 1_G)\circ \delta_X$ is the identity on $X$. To see this observe that 
  it follows from (a), that for all $z\in C^*(G)$ and $x\in X$ we have $(\id_X\otimes 1_G)(\delta_X(x)(1\otimes z))\in X$.
Choosing $z$ such that $1_G(z)= 1$ then implies that $(\id_X\otimes 1_G)(\delta_X(x))=(\id_X\otimes 1_G)(\delta_X(x)(1\otimes z))\in X$
as well. Now, using condition (2) and the relation $(\id_G\otimes 1_G)\circ \delta_G=\id_G$ we get
\begin{align*}
\delta_X(x)&=(\id_X\otimes \id_G\otimes 1_G)\circ (\id_X\otimes \delta_G)\circ \delta_X(x)\\
&=
(\id_X\otimes \id_G\otimes 1_G)\circ (\delta_X\otimes \id_G)\circ \delta_X(x)\\
&=(\delta_X\otimes 1_G)\circ \delta_X(x)=\delta_X\big( (\id_X\otimes 1_G)(\delta_X(x))\big),
\end{align*}
 which implies $\delta_X\circ (\id_A\otimes 1_G)\circ \delta_X=\delta_X$. Since $\delta_X$ is injective, this implies 
 that  $(\id_X\otimes 1_G)\circ \delta_X=\id_X$. In particular, it follows that $\delta_X$ is completely isometric.

  \end{remark}

  \begin{example}\label{ex dual}
  Suppose that $\alpha:G\to \Aut(A,X)$ is an action of the locally compact group $G$ 
  on the $C^*$-operator system $(A,X)$. Then there is a {\em dual coaction}
  $$\widehat{\alpha}: (A\rtimes_\alpha^uG, X\rtimes_\alpha^uG)\to \big(M(A\rtimes_\alpha^uG\check\otimes C^*(G)), M(X\rtimes_\alpha^uG\check\otimes C^*(G))\big)$$
  given by the integrated form of the generalized covariant homomorphism
  $(i_X\otimes 1, i_G\otimes u)$ of $(A,X, G,\alpha)$ into $ M(X\rtimes_\alpha^uG\check\otimes C^*(G))$, where 
  $(i_X, i_G):(X, G)\to M(X\rtimes_\alpha^uG)$ denotes the canonical covariant homomorphism and 
  $u:G\to UM(C^*(G))$ the universal representation.
  
To see that this satisfies the conditions of Definition \ref{def coaction} we choose a faithful and non-degenerate representation 
 of $C_u^*(A,X)\rtimes_{\alpha,u}G$ into some $ \mathcal B(H)$. This restricts to a faithful representation of $(A\rtimes_\alpha^uG, X\rtimes_{\alpha}^uG)$,  into $\mathcal  B(H)$.
 Moreover, by choosing a faithful representation of $C^*(G)$ onto some Hilbert space $K$, say, we obtain a faithful and non-degenerate representation of $C_u^*(A,X)\rtimes_{\alpha^u,u}G\check\otimes C^*(G)$ on $H \otimes  K$, which restricts to 
 a completely isometric representation of $X \check\otimes C^*(G))$,  and hence of 
 $M(X\rtimes_\alpha^uG\check\otimes C^*(G))$, respectively (use Corollary \ref{cor-isometry}).
 Now, there is a dual coaction 
 $$\widehat{\alpha}_u:  C_u^*(A,X)\rtimes_{\alpha,u}G\to M(C_u^*(A,X)\rtimes_{\alpha,u}G\check\otimes C^*(G))$$
 given as the integrated form of the covariant homomorphism  $(i_{C_u^*(A,X)}\otimes 1, i_G\otimes u)$ of 
 $(C_u^*(A,X), G, \alpha)$ into $M(C_u^*(A,X)\rtimes_{\alpha,u}G\check\otimes C^*(G))$. 
 This representation clearly restricts to the representation  $(i_X\otimes 1, i_G\otimes u):(X, G)\to M(X\rtimes_\alpha^uG\check\otimes C^*(G))$ 
 and the conditions in Definition \ref{def coaction} can then easily be deduced from the properties of the 
 coaction $\widehat{\alpha}_u$ of $C^*(G)$ on $C_u^*(A,X)\rtimes_{\alpha,u}G$. 

 Similarly, the dual  coaction 
 $$\widehat{\alpha_r}: C_u^*(A,X)\rtimes_{\alpha, r}G\to M(C_u^*(A,X)\rtimes_{\alpha, r}G\check\otimes C^*(G))$$
  restricts 
 to a dual coaction 
 $$\widehat{\alpha}_r:=(i_X^r\otimes 1)\rtimes (i_G^r\otimes u): X\rtimes_\alpha^r G\to M(X\rtimes_\alpha^rG \check\otimes C^*(G))$$ of $C^*(G)$
  on the reduced 
 crossed product $(A,X)\rtimes_{\alpha}^rG=(A\rtimes_{\alpha}^rG, X\rtimes_{\alpha}^rG)$, where 
 $(i_X^r, i_G^r)$ denotes the canonical covariant homomorphism of $(A,X,G,\alpha)$ into $M(X\rtimes_{\alpha}^rG)$
 (i.e., the regular representation).
  \end{example}

We now want to relate coactions of $C^*(G)$ on $(A,X)$ with coactions of $C^*(G)$ on $C_u^*(A,X)$.
In order to formulate the result, observe that 
$\big(C_u^*(A,X)\check\otimes C^*(G), j_u\otimes \id_{C^*(G)}\big)$ is a $C^*$-hull of $(A\check\otimes C^*(G), X\check\otimes C^*(G))$ and therefore 
we get a canonical completely isometric embedding 
$$\overline{j_u\otimes \id_G}: M(A\check\otimes C^*(G), X\check\otimes C^*(G))\into M(C_u^*(A,X)\check\otimes C^*(G)).$$

\begin{proposition}\label{prop coaction}
Let $(A,X)$ be a $C^*$-operator system. Then there is a one-to-one correspondence 
between coactions $\delta_X$ of $G$ on $(A,X)$ and coactions 
$$\delta_u:C_u^*(A,X)\to M(C_u^*(A,X)\check\otimes C^*(G))$$ 
of $G$ on $C_u^*(A,X)$ which satisfy the conditions
\begin{equation}\label{eq-rest}
\delta_u(A)\subseteq M(A\check\otimes C^*(G))\quad{and}\quad 
\delta_u(X)\subseteq M(X\check\otimes C^*(G)),
\end{equation}
where we regard $A$ and $X$ as subspaces of $C_u^*(A,X)$ via the inclusion map $j_u$.
Given such coaction $\delta_u$ of $G$ on $C_u^*(A,X)$, the corresponding coaction of $G$ on 
$(A,X)$ is given by the restriction  $\delta_X:=\delta_u|_X$.
\end{proposition}
\begin{proof} 
Suppose first that $\delta_X:X\to M(X\check\otimes C^*(G))$ is a coaction of $C^*(G)$ on $(A,X)$.
It follows then from part (b) of Remark \ref{rem nondeg} and Lemma \ref{lem mult univ} that 
$$(j_u\otimes \id_G)\circ \delta_X: X\to M(C_u^*(A,X)\check\otimes C^*(G))$$
is a completely isometric representation of $(A,X)$ into $M(C_u^*(A,X)\check\otimes C^*(G))$. 
By Proposition \ref{prop-universal}, it extends to a non-degenerate $*$-homomorphism 
$$\delta_u:C_u^*(A,X)\to M(C_u^*(A,X)\check\otimes C^*(G)).$$
Since $C_u^*(A, X)$ is generated by $X$ and since
$$\delta_u(x)(1\otimes z)=\delta_X(x)(1\otimes z)\in X\check\otimes C^*(G)\subseteq C_u^*(A,X)\check\otimes C^*(G)$$
for all $x\in X$, it follows that 
$$\delta_u(C_u^*(A,X))(1\check\otimes C^*(G))\subseteq C_u^*(A,X)\check\otimes C^*(G).$$
Using this, it  follows that $(\id_{C_u^*(A,X)}\otimes 1_G)\circ \delta_u$ maps $C_u^*(A,X)$ into $C_u^*(A,X)$. 
Moreover, since $(\id_{C_u^*(A,X)}\otimes 1_G)\circ \delta_u$ restricts to $(\id_X\otimes 1_G)\circ \delta_X=\id_X$ on $X$,
it follows that $(\id_{C_u^*(A,X)}\otimes 1_G)\circ \delta_u$ extends the identity on $X$ and therefore must be equal to the 
identity on $C_u^*(A,X)$. Thus it follows that $\delta_u$ is injective. 

In order to check that
\begin{equation}\label{eq-coact}
(\delta_u\otimes \id_G)\circ \delta_u=(\id_{C_u^*(A,X)}\otimes \delta_G)\circ \delta_u
\end{equation}
as maps from $C_u^*(A,X)$ into $M(C_u^*(A,X)\check\otimes C^*(G)\check\otimes C^*(G))$ we simply observe that,
via the canonical embedding of $M(X\check\otimes C^*(G)\check\otimes C^*(G))$ into \linebreak
$M(C_u^*(A,X)\check\otimes C^*(G)\check\otimes C^*(G))$, 
the left hand  side restricts to $(\delta_X\otimes \id_G)\circ \delta_X$ and the right hand side restricts to $(\id_X\otimes \delta_G)\circ \delta_X$.
But by condition (ii) of Definition \ref{def coaction}, these restrictions to $X$ coincide and then (\ref{eq-coact}) 
follows  from the uniqueness assertion of Proposition \ref{prop-universal}.

Conversely, suppose  that $\delta:C_u^*(A,X)\to M(C_u^*(A,X)\check\otimes C^*(G))$ is a coaction of $G$ 
on $C_u^*(A,X)$ such that the equations (\ref{eq-rest}) hold. We need to check that 
$\delta_X:=\delta|_X$ is a coaction of $G$ on $(A,X)$, where we realize $M(X\check\otimes C^*(G))$ as a subspace of 
$M(C_u^*(A,X)\check\otimes C^*(G))$ as explained above. 
Since $\delta$ is injective, the same holds for $\delta_X$ and condition (ii) of Definition \ref{def coaction}
clearly follows from the similar condition for $\delta$. Thus, all we need to show is that 
$\delta(A)(1\check\otimes C^*(G))\subseteq A\check\otimes C^*(G)$. But since $\delta(A)(1\check\otimes C^*(G))\subseteq C_u^*(A,X)\check\otimes C^*(G)$
we observe that
$$\delta(A)(1\check\otimes C^*(G))\subseteq M(A\check\otimes C^*(G))\cap (C_u^*(A,X)\check\otimes C^*(G)),$$
where the intersection is taken inside $M(C_u^*(A,X)\check\otimes C^*(G))$. But  it follows from Lemma \ref{lem mult univ}
that this intersection equals $A\check\otimes C^*(G)$ and the result follows.
\end{proof}

\begin{definition}\label{def-nondeg}
A coaction $\delta_X: X\to M(X\check\otimes C^*(G))$ of $G$ on $(A,X)$ is called {\em non-degenerate} 
if the corresponding coaction $\delta_u$ of $G$ on $C_u^*(A,X)$ is non-degenerate, 
i.e., 
$$\delta_u(C_u^*(A,X))(1\check\otimes C^*(G))=C_u^*(A,X)\check\otimes C^*(G).$$
\end{definition}

\begin{remark}\label{rem nondeg-coaction}
Of course it would be more satisfactory to define nondegeneracy of a coaction of $G$ on $(A,X)$ 
via a condition like 
$$\delta_X(X)(1\check\otimes C^*(G))=X\check\otimes C^*(G).$$
However, we were not able to prove that this condition is equivalent to nondegeneracy of $\delta_u$,
and it is the latter  condition we shall need later when dealing with Imai-Takai duality. 
Note that nondegeneracy of a coaction on $(A,X)$ is automatic for amenable or discrete $G$, since,
as we remarked before,  this holds true for coactions on $C^*$-algebras. The same holds for all dual coactions:
\end{remark}

\begin{lemma}\label{lem-non-degenerate}
Suppose that $\alpha:G\to \Aut(A,X)$ is an action of $G$ on the $C^*$-operator system $(A,X)$.
Then the dual coactions $\widehat{\alpha}:X\rtimes_{\alpha}^uG\to M(X\rtimes_{\alpha}^uG\check\otimes C^*(G))$ 
and $\widehat{\alpha_r}:X\rtimes_{\alpha}^rG\to M(X\rtimes_{\alpha}^rG\check\otimes C^*(G))$
are non-degenerate.
\end{lemma} 
\begin{proof}
The first assertion follows from Corollary \ref{cor-universal} together with the fact that
dual coactions on $C^*$-algebra crossed products are non-degenerate (e.g., see the discussion at the end of 
\cite[Example A.26]{EKQR}). For the dual coaction on the reduced crossed $(A,X)\rtimes_{\alpha}^rG$ 
observe that the identity map on $C_c(G,X)$ induces surjective morphisms
$$C_u^*(A,X)\rtimes_{\alpha, u}G\onto C_u^*(A\rtimes_\alpha^rG, X\rtimes_{\alpha}^rG)\onto C_u^*(A,X)\rtimes_{\alpha,r}G,$$
where the first map exists by the universal property of $C_u^*(A\rtimes_\alpha^uG, X\rtimes_\alpha^uG)\cong 
C_u^*(A,X)\rtimes_{\alpha, u}G$ together with the obvious morphism of $(A\rtimes_\alpha^uG, X\rtimes_\alpha^uG)$ into 
$C_u^*(A\rtimes_\alpha^rG, X\rtimes_{\alpha}^rG)$. 
These maps are equivariant for the respective dual coactions, where the one on $C_u^*(A\rtimes_\alpha^rG, X\rtimes_{\alpha}^rG)$
is induced from the dual coaction of $C^*(G)$ on $(A\rtimes_\alpha^rG, X\rtimes_\alpha^rG)$ via 
Proposition \ref{prop coaction}. But then it is easy to see that nondegeneracy of the dual coaction $\widehat{\alpha}$ 
on $C_u^*(A,X)\rtimes_{\alpha, u}G$ implies nondegeneracy of the (dual) coaction on $C_u^*(A\rtimes_\alpha^rG, X\rtimes_{\alpha}^rG)$.
\end{proof}
We are now going to study covariant representations for coactions on $C^*$-operator systems $(A, X)$, extending the 
well known theory for coactions on $C^*$-algebras. In what follows let $w_G\in UM(C_0(G)\check\otimes C^*(G))$ denote 
the unitary given by the map $[g\mapsto u_g]\in C_b(G, C^*(G))\subseteq M(C_0(G)\check\otimes C^*(G))$.
Recall from  \cite[Definition A.32]{EKQR} the following definition

\begin{definition}\label{def-covariant-coact}
Let $\delta:D\to M(D\check\otimes C^*(G))$ be a coaction of $G$ on the $C^*$-algebra $D$ and let $B$ be a $C^*$-algebra.
Then a {\em covariant representation} of $(D,G,\delta)$ into $M(B)$ is a pair $(\pi,\mu)$, where
$\pi:D\to M(B), \mu:C_0(G)\to M(D)$ are  non-degenerate
$*$-homomorphism satisfying the covariance condition
$$(\pi\otimes\id_G)\circ \delta(d)=(\mu\otimes \id_G)(w_G)(\pi(d)\otimes 1)(\mu\otimes \id_G)(w_G)^*.$$
If $B=\mathcal K(H)$ for some Hilbert space $H$, then we say that $(\pi, \mu)$ is a covariant representation on $H$.
\end{definition}

If $(\pi, \mu)$ is a covariant representation of $(D,G,\delta)$ as above, then 
 $$\pi(D)\mu(C_0(G)):=\cspn\{\pi(d)\mu(f) : d\in D, f\in C_0(G)\}$$
  is a $C^*$-subalgebra of $M(B)$
(see \cite[Proposition A.36]{EKQR}). Moreover, it is shown in \cite[Proposition A.37]{EKQR} that 
the pair $(\Lambda_D, \Lambda_{\widehat{G}}):=\big((\id_D\otimes \lambda)\circ \delta, 1\otimes M)$
where $M:C_0(G)\to \mathcal B(L^2(G))=M(\mathcal K(L^2(G)))$ 
denotes the representation by multiplication operators, defines a covariant representation, called {\em regular representation}, of 
$(D, G, \delta)$ into $M(D\otimes \mathcal K(L^2(G)))$. 
The {\em crossed product} $D\rtimes_\delta\widehat{G}$ of the co-system $(D,G,\delta)$ is then defined 
as the $C^*$-algebra
$$D\rtimes_\delta\widehat{G}:=\Lambda_D(D)\Lambda_{\widehat{G}}(C_0(G))\subseteq M(D\otimes\mathcal K(L^2(G))).$$
We can then view $(\Lambda_D, \Lambda_{\widehat{G}})$ as a covariant representation into $M(D\rtimes_\delta\widehat{G})$ in a canonical way.
It is then shown in \cite[Theorem A.41]{EKQR} that the triple $(D\rtimes_\delta\widehat{G}, \Lambda_D, \Lambda_{\widehat{G}})$
satisfies the following universal property:
If $(\pi,\mu)$ is any covariant representation of $(D,G,\delta)$ into $M(B)$, then there exists a unique 
$*$-homomorphism
$\pi\rtimes\mu: D\rtimes_\delta \widehat{G}\to M(B)$ such that 
\begin{equation}\label{eq-covuniv}
(\pi\rtimes \mu)\circ \Lambda_D=\pi\quad\text{and}\quad (\pi\rtimes\mu)\circ \Lambda_{\widehat{G}}= \mu.
\end{equation}
Moreover, we get
$$\pi\rtimes\mu(D\rtimes_\delta \widehat{G})=\pi(D)\mu(C_0(G)).$$
\medskip

We are now going derive analogues of the above constructions and facts for coactions 
of $G$ on $C^*$-operator systems $(A,X)$. We start with

\begin{definition}\label{def-covariant$C^*$}
Suppose that $\delta_X:X\to M(X\check\otimes C^*(G))$ is a coaction of $G$ on $(A,X)$ and let $(B,Y)$ 
be a $C^*$-operator system. Then a covariant morphism of $(A,X, G,\delta_X)$ into $M(B,Y)=(M(B), M(Y))$
consists of a non-degenerate generalized morphism $\pi:X\to M(Y)$ of $(A,X)$ together with a non-degenerate 
$*$-homomorphism $\mu:C_0(G)\to M(B)$ such that the pair $(\pi, \mu)$ satisfies the 
covariance condition
$$(\pi\otimes\id_G)\circ \delta_X(x)=(\mu\otimes \id_G)(w_G)(\pi(x)\otimes 1)(\mu\otimes \id_G)(w_G)^*$$
for all $x\in X$. If $(B,Y)=(B,B)$ is a $C^*$-algebra, we say that $(\pi,\mu)$ is a covariant representation 
of $(A,X,G, \delta_X)$ into $M(B)$. If, in addition, $B=\mathcal K(H)$ for some Hilbert space $H$,
we say that $(\pi,\mu)$ is a covariant representation of $(A,X, G,\delta_X)$ on $H$.
\end{definition}

\begin{remark} Observe that the restricted pair $(\pi_A, \mu)$ with $\pi_A:=\pi|_A$ of a covariant representation 
$(\pi,\mu)$ of $(A,X,G,\delta_X)$  into $M(B,Y)$ is a non-degenerate covariant homomorphism of $(A, G, \delta_A)$ 
into $M(B)$.
\end{remark}


\begin{proposition}\label{prop-image}
Suppose that $(\pi, \mu)$ is a covariant morphism of \linebreak $(A,X, G, \delta_X)$ into $M(B, Y)$ for some $C^*$-operator system $(B, Y)$.
Then \linebreak $\big(\pi(A)\mu(C_0(G)), \pi(X)\mu(C_0(G))\big)$ (closed spans!) is a $C^*$-operator subsystem of $M(B,Y)$.
\end{proposition}
\begin{proof} We first  observe that it follows directly from the above discussion that $\pi(A)\mu(C_0(G))$ is 
a non-degenerate $C^*$-subalgebra of $M(B)$. Note that this implies in particular $\mu(C_0(G))$ acts as 
multipliers on this $C^*$-algebra. On the other hand, precisely the same arguments as used in the proof 
of \cite[Proposition A.36]{EKQR} show that 
$\pi(X)\mu(C_0(G))=\mu(C_0(G))\pi(X)$, from which it follows that 
$\pi(X)\mu(C_0(G))$ is a selfadjoint  subspace of $M(Y)$. 
So in order to complete the proof, we only need to show that 
\begin{align*}
\big(\pi(A)\mu(C_0(G))\big)\big(\pi(X)\mu(C_0(G))\big)&=\pi(X)\mu(C_0(G))\\
&=\big(\pi(X)\mu(C_0(G))\big)\big(\pi(A)\mu(C_0(G))\big),
\end{align*}
but this follows $AX=X=XA$ and 
$\pi(X)\mu(C_0(G))=\mu(C_0(G))\pi(X)$.
\end{proof}

\begin{proposition}\label{prop-extend universal}
Let $\delta_X:X\to M(X\check\otimes C^*(G))$ be a coaction of $G$ on $(A,X)$ and let $B$ be a $C^*$-algebra. 
Let $\delta_u:C_u^*(A,X)\to M(C_u^*(A,X)\check\otimes C^*(G))$ denote the corresponding coaction of $G$ on the 
universal $C^*$-hull $C_u^*(A,X)$ of $(A,X)$ as in Proposition \ref{prop coaction}. 

Then there is 
a one-to-one correspondence between the non-degenerate covariant representations 
$(\pi, \mu)$ of $(A,X, G, \delta_X)$ into $M(B)$ 
and the non-degenerate covariant representations of $(C_u^*(A,X), G, \delta_u)$ into $M(B)$, given by  sending a covariant pair
$(\pi, \mu)$ of  
$(A,X, G, \delta_X)$  to the covariant pair $(\bar\pi, \mu)$, where  $\bar\pi:C_u^*(A,X)\to M(B)$
denotes the unique $*$-homomorphism which extends $\pi$.
\end{proposition}
\begin{proof}
Since $(\pi\otimes\id_G)\circ \delta_X=(\mu\otimes \id_G)(w_G)(\pi(\cdot)\otimes 1)(\mu\otimes \id_G)(w_G)^*$
as maps from $X$ into $M(D\check\otimes C^*(G))$, it follows that both of the 
 $*$-homomorphisms  $(\bar\pi\otimes\id_G)\circ \delta_u$ and $(\mu\otimes \id_G)(w_G)(\pi(\cdot )\otimes 1)(\mu\otimes \id_G)(w_G)^*$
 from $C_u^*(A,X)$ to $M(B\check\otimes C^*(G))$ extend the same non-degenerate generalized morphism of $(A,X)$,
 and therefore they coincide by Proposition \ref{prop-universal}. This shows that any covariant representation of 
 $(A,X,G,\delta_X)$ has a unique extension to $(C_u^*(A,X), G, \delta_u)$. The converse direction follows by restriction.
\end{proof} 

\begin{lemma}\label{lem-regularrep}
Suppose that $(A,X, G,\delta)$ is a coaction of $G$ on $(A,X)$. Then the pair 
$(\Lambda_X, \Lambda_{\widehat{G}}):=((\id_X\otimes \lambda)\circ \delta_X, 1\otimes M)$ 
defines a covariant representation of $(A,X, G,\delta)$ into $M\big(A\otimes \mathcal K(L^2(G)), X\otimes \mathcal K(L^2(G))\big)$
which we call the {\em regular representation} 
of $(A,X, G,\delta)$ into $M(X\check\otimes C^*(G))$.

Moreover, via the completely isometric embedding of $(A,X)$ into $C_u^*(A, X)$ and the corresponding completely isometric embedding 
of $(A\otimes \K(L^2(G)), X\otimes \K(L^2(G)))$ into $C_u^*(A,X)\otimes\K(L^2(G))$, we may view $(\Lambda_X, \Lambda_{\widehat{G}})$
as a covariant representation into $M(C_u^*(A,X)\otimes\K(L^2(G)))$, which uniquely extends to the regular representation 
$(\Lambda_{C_u^*(A,X)}, \Lambda_{\widehat{G}})$ of $(C_u^*(A,X), G, \delta_u)$ in the sense of Proposition \ref{prop-extend universal}.
\end{lemma}
\begin{proof} For the first assertion we follow the proof of \cite[Proposition A.37]{EKQR}. Using the identity
$$(\lambda\otimes \id_G)\circ \delta_G=\Ad(M\otimes \id_G)(w_G)\circ (\lambda\otimes 1),$$
which has been established in the proof of \cite[Proposition A.37]{EKQR}, we compute
\begin{align*}
&\big((\id_X\otimes\lambda)\circ \delta_X\otimes\id_G\big)\circ \delta_X(x)=(\id_X\otimes\lambda \otimes \id_G)\circ (\delta_X\otimes \id_G)\circ \delta_X(x)\\
&=(\id_X\otimes \lambda\otimes\id_G)\circ (\id_X\otimes \delta_G)\circ \delta_X(x)\\
&=(\id_X\otimes (\lambda\otimes \id_G)\circ \delta_G)\circ \delta_X(x)\\
&=(1\otimes M\otimes \id_G)(w_G)\big((\id_X\otimes\lambda)(\delta_X(x)\otimes 1)(1\otimes M\otimes \id_G)(w_G)^*.
\end{align*}
This proves the covariance condition for $((\id_X\otimes \lambda)\circ \delta_X, 1\otimes M)$. The second assertion is now obvious.
\end{proof}

\begin{definition}\label{def-dual-cross}
Suppose that $\delta_X:X\to M(X\check\otimes C^*(G))$ is a coaction of $G$ on the $C^*$-operator system $(A,X)$.
Then we define the crossed product $(A,X)\rtimes_{\delta_X}G$ as the $C^*$-operator system 
$$(A\rtimes_{\delta_A}\widehat{G}, X\rtimes_{\delta_X}\widehat{G}):=\left(\Lambda_X(A)\Lambda_{\widehat{G}}(C_0(G)), \Lambda_X(X)\Lambda_{\widehat{G}}(C_0(G))\right)$$
generated by the regular representation $(\Lambda_X, \Lambda_{\widehat{G}})$ of $(A,X,G, \delta_X)$ as in Proposition \ref{prop-image}.
\end{definition}

\begin{remark}\label{rem-dual-cross} Note that it follows directly from Lemma \ref{lem-regularrep} and the above definition 
that $C_u^*(A,X)\rtimes_{\delta_u}\widehat{G}=\Lambda_{C_u^*(A,X)}(C_u^*(A,X))\Lambda_{\widehat{G}}(C_0(G))$
is a $C^*$-hull of  $(A\rtimes_{\delta_A}\widehat{G}, X\rtimes_{\delta_X}\widehat{G})$. Indeed, we shall see below, that it the universal $C^*$-hull
of $(A\rtimes_{\delta_A}\widehat{G}, X\rtimes_{\delta_X}\widehat{G})$.
\end{remark}

We now show that the above defined crossed product does enjoy a universal property for covariant representations:

\begin{proposition}\label{prop-dualuniv}
Suppose that $(\pi, \mu)$ is a covariant morphism of $(A,X,G,\delta_X)$ into $M(B,Y)$ for some
$C^*$-operator system $(B,Y)$. Then there is a unique generalized morphism 
$$\pi\rtimes\mu: (A\rtimes_{\delta_A}\widehat{G}, X\rtimes_{\delta_X}\widehat{G})\to M(B,Y)$$
such that
\begin{equation}\label{eq-univ}
(\pi\rtimes\mu)\circ \Lambda_X=\pi\quad\text{and}\quad (\pi\rtimes\mu)\circ \Lambda_{\widehat{G}}=\mu.
\end{equation}
Conversely, if $\Phi:  (A\rtimes_{\delta_A}\widehat{G}, X\rtimes_{\delta_X}\widehat{G})\to M(B,Y)$ is any non-degenerate 
generalized morphism, then there is a unique covariant morphism $(\pi, \mu)$ of $(A,X,G,\delta_X)$ 
such that $\Phi=\pi\rtimes\mu$.
\end{proposition}
\begin{proof}
As for the case of actions, we are going to use the correspondence between covariant representations of $(A,X, G, \delta_X)$ and covariant representations 
of $(C_u^*(A,X), G, \delta_u)$ as established in Proposition \ref{prop-extend universal}. 
For this we choose a non-degenerate completely isometric embedding of $(B,Y)$ into $\mathcal B(H)$ for some Hilbert space $H$.
Then, if we compose a covariant representation of $(\pi,\mu)$ of $(A,X, G, \delta_X)$ into $M(B,Y)$ with this inclusion, we may view
$(\pi, \mu)$ as a representation into $\mathcal B(H)=M(\K(H))$. By Proposition \ref{prop-extend universal}, this extends to a covariant representation,
say $(\bar\pi, \mu)$
of $(C_u^*(A,X), G, \delta_u)$ into $\mathcal B(H)$. By the universal property of the $C^*$-algebra crossed product 
$(C_u^*(A,X)\rtimes_{\delta_u} \widehat{G}, \Lambda_{C_u^*(A,X)}, \Lambda_{\widehat{G}})$ there exists a unique $*$-homomorphism 
$\bar\pi\rtimes\mu: C_u^*(A,X)\rtimes_{\delta_u} \widehat{G}\to \mathcal B(H)$ such that 
$$(\bar\pi\rtimes\mu)\circ \Lambda_{C_u^*(G,A)}=\bar\pi\quad\text{and}\quad (\bar\pi\rtimes\mu)\circ \Lambda_{\widehat{G}}=\mu.$$
Define $\pi\rtimes\mu$ as the restriction of $\bar\pi\rtimes \mu$ to $X\rtimes_{\delta_X}\widehat{G}\subseteq 
C_u^*(A,X)\rtimes_{\delta_u} \widehat{G}$.  Then the restriction
 restriction of $(\bar\pi\rtimes\mu)\circ \Lambda_{C_u^*(G,A)}$ to $X$ equals  $(\pi\rtimes\mu)\circ \Lambda_X$.
Since $\bar\pi$ extends $\pi$, we see that $\pi\rtimes\mu: (A\rtimes_{\delta_A}\widehat{G}, X\rtimes_{\delta_X}\widehat{G})\to \mathcal B(H)$ 
is a representation which satisfies (\ref{eq-univ}). 

We still need to check that $\pi\rtimes\mu$ can be viewed as a generalized morphism into $M(B,Y)$. For this it suffices  
to check that $\pi\rtimes \mu(X\rtimes_{\delta_X}\widehat{G})B\subseteq Y$ and $\pi\rtimes \mu(A\rtimes_{\delta_A}\widehat{G})B=B$.
But if we apply $\pi\rtimes\mu$ on a typical element of the form $\Lambda_X(x)\Lambda_{\widehat{G}}(f)$ of $X\rtimes_{\delta_X}\widehat{G}$
with $x\in X$, $f\in C_0(G)$, it follows from equation (\ref{eq-univ}) that  
$$\pi\rtimes \mu\big((\Lambda_X(x)\Lambda_{\widehat{G}}(f)\big)b=\pi(x)\mu(f)b\in Y$$
since, by definition of a covariant representation into $M(B,Y)$, we have $\mu(f)\in M(B)$ and $\pi(x)\in M(Y)$. 
Moreover, since $\pi|_A:A\to M(B)$ and $\mu:A\to M(B)$ are supposed to be non-degenerate, we get 
$$\pi\rtimes \mu(A\rtimes_{\delta_A}\widehat{G})B=(\pi(A)\mu(C_0(G)))B=\pi(A)(\mu(C_0(G))B)=\pi(A)B=B.$$

If, conversely, $\Phi:  (A\rtimes_{\delta_A}\widehat{G}, X\rtimes_{\delta_X}\widehat{G})\to M(B,Y)$ is any non-degenerate 
generalized morphism, then we leave it as a straightforward exercise to check that the pair $(\pi,\mu)$ 
with $\pi:=\Phi\circ \Lambda_X, \mu:=\Phi\circ \Lambda_{\widehat{G}}$ is a non-degenerate covariant morphism 
such that $\Phi=\pi\rtimes\mu$.
\end{proof}

\begin{corollary}\label{cor-univhull}
Suppose that $(A,X, G, \delta_X)$ is a coaction of $G$ on $(A,X)$. Then 
$$C_u^*(A\rtimes_{\delta_A}\widehat{G}, X\rtimes_{\delta_X}\widehat{G})=C_u^*(A,X)\rtimes_{\delta_u}\widehat{G}.$$
\end{corollary}
\begin{proof} We already observed in Remark \ref{rem-dual-cross} that $C_u^*(A,X)\rtimes_{\delta_u}\widehat{G}$
is a $C^*$-hull of $(A\rtimes_{\delta_A}\widehat{G}, X\rtimes_{\delta_X}\widehat{G})$. So we only need to show that every  
representation $\Phi: (A\rtimes_{\delta_A}\widehat{G}, X\rtimes_{\delta_X}\widehat{G})\to B$ extends to a homomorphism 
$\bar\Phi: C_u^*(A,X)\rtimes_{\delta_u}\widehat{G}\to B$.
For this let us assume without loss of generality that $B$ is generated by the image $\Phi(X\rtimes_{\delta_X}\widehat{G})$. Then $\Phi$
is non-degenerate and there exists a unique non-degenerate covariant representation $(\pi,\mu)$ of $(A,X,G, \delta_X)$ such that 
$\Phi=\pi\rtimes\mu$. By Proposition \ref{prop-extend universal} $(\pi,\mu)$ extends uniquely to a covariant homomorphism of 
$(\bar\pi,\mu)$ of $(C_u^*(A,X), G,\delta_u)$ into $B$. The arguments given in the proof of Proposition \ref{prop-dualuniv} then show that
$\bar\pi\rtimes \mu$ is a $*$-homomorphism from $C_u^*(A,X)\rtimes_{\delta_u}\widehat{G}$ into $B$ which restricts to
$\pi\rtimes\mu$ on $X$.
\end{proof}

\begin{remark}\label{rem-coact}
We should note that, different from the situation for crossed products by actions, the definition of the 
crossed product $(A\rtimes_{\delta_A}\widehat{G}, X\rtimes_{\delta_X}\widehat{G})$ does not depend 
on the crossed product by the universal $C^*$-hull $C_u^*(A,X)$. This algebra is only used to reduce the 
proof of the universal properties to the well known case of coaction crossed products by $C^*$-algebras.
One should observe that the definition of the crossed product for coactions is more like the 
definition of the reduced crossed product in case of group actions. The fact that this constructions already enjoys
the universal property for covariant morphisms comes from the fact that the locally compact quantum group $C^*(G)$ 
is amenable for all $G$ (or, in other words, every group $G$ is coamenable (e.g., see \cite{BS1, KV} for a discussion 
of these notions).  Hence we only have one reasonable candidate for a coaction crossed product!
For this reason, it also follows that the algebra part $A\rtimes_{\delta_A}G$ of 
$(A\rtimes_{\delta_A}\widehat{G}, X\rtimes_{\delta_X}\widehat{G})$ is the (universal and reduced) crossed product of 
$A$ with respect to the coaction $\delta_A$.
\end{remark}

\section{Duality}\label{sec-dual}
We now want to deduce versions of the Imai-Takai and Katayama duality for crossed products by actions and coactions.
By Example \ref{ex dual} we know that the universal  and reduced crossed products $(A\rtimes_\alpha^uG, X\rtimes_\alpha^uG)$ 
and $(A\rtimes_\alpha^rG, X\rtimes_\alpha^rG)$ for an action $\alpha:G\to \Aut(A,X)$ carry canonical dual 
coactions $\widehat{\alpha}$ and $\widehat{\alpha_r}$ which are given as the integrated forms of the 
covariant morphisms $\widehat{\alpha}=(i_X\otimes 1)\rtimes (i_G\otimes u)$ and 
$\widehat{\alpha_r}=(i_X^r\otimes 1)\rtimes (i_G^r\otimes u)$ into $M(X\rtimes_{\alpha}^uG\check\otimes C^*(G))$ (resp. $M(X\otimes_\alpha^rG\check\otimes C^*(G))$),
where $(i_X, i_G)$ and $(i_X^r, i_G^r)$ are the canonical covariant morphisms from $(A,X, G,\alpha)$ into 
$M(X\rtimes_\alpha^uG)$ and $M(X\rtimes_\alpha^rG)$, respectively.
Note that it follows directly from the constructions in Example \ref{ex dual}, that these coactions extend to the dual coaction on 
$C_u^*(A,X)\rtimes_{\alpha, u}G=C_u^*(A\rtimes_\alpha^uG, X\rtimes_\alpha^uG)$ and  $C_u^*(A,X)\rtimes_{\alpha_u, r}G$,
respectively.

In a similar way, we have

\begin{proposition}\label{prop-dualaction}
Suppose that $(A,X,G, \delta_X)$ is a coaction of $G$ on $(A,X)$. Then there is a canonical dual action 
$$\widehat\delta:G\to \Aut(A\rtimes_{\delta_A}\widehat{G}, X\rtimes_{\delta_X}\widehat{G})$$
wich is given on a typical element $\Lambda_X(x)\Lambda_{\widehat{G}}(f)$ by
$$\widehat\delta_{g}\big(\Lambda_X(x)\Lambda_{\widehat{G}}(f)\big)=\Lambda_X(x)\Lambda_{\widehat{G}}(\sigma_g(f)),$$
where $\sigma: G\to \Aut(C_0(G))$ denotes the right translation action, i.e., $\sigma_g(f)(s)=f(sg)$ for all $g,s\in G, f\in C_0(G)$.
\end{proposition}
\begin{proof}
For any covariant representation $(\pi, \mu)$ of $(A,X,G,\delta_X)$ the pair \linebreak $(\pi, \mu\circ \sigma_g)$ 
is a covariant representation as well. Indeed, for all $x\in X$ we have
\begin{align*}
(\mu\circ &\sigma_g\otimes\id_G)(w_G)(x\otimes 1)(\mu\circ \sigma_g\otimes\id_G)(w_G)^*\\
&=(\mu \otimes\id_G)(w_G)(1\otimes u_g)(x\otimes 1)(1\otimes u_g^*)(\mu \otimes\id_G)(w_G)^*\\
&=(\mu \otimes\id_G)(w_G)(x\otimes 1)(\mu \otimes\id_G)(w_G)^*\\
&=\pi(x).
\end{align*}
Applying this to the regular representation $(\Lambda_X, \Lambda_{\widehat{G}})$ we get a covariant representation
$(\Lambda_X, \Lambda_{\widehat{G}}\circ \sigma_g)$ of $(A,X,G, \delta_X)$ into $M(X\rtimes_{\delta_X} \widehat{G})$
whose integrated form $\widehat{\delta}_g$ maps $\Lambda_X(x)\Lambda_{\widehat{G}}(f)$ to $\Lambda_X(x)\Lambda_{\widehat{G}}(\sigma_g(f))$.
It is then clear that $\widehat{\delta}_{g^{-1}}$ inverts $\widehat{\delta}_g$ and that $g\mapsto \widehat\delta_g$ is a homomorphism
into $\Aut(X\rtimes_{\delta_X}\widehat{G})$. Since the action $\sigma:G\to \Aut(C_0(G))$ is strongly 
continuous, the same holds for $\widehat\delta$.
\end{proof}

We now formulate the analogue of the Imai-Takai duality theorem for crossed products of $C^*$-operator systems by actions.

\begin{theorem}\label{thm-ImaiTakai}
Suppose that $\alpha: G\to \Aut(A,X)$ is an action and let $\K:=\K(L^2(G))$. Then there are canonical isomorphisms 
$$\big(A\rtimes_{\alpha}^uG\rtimes_{\widehat{\alpha}}\widehat{G}, X\rtimes_{\alpha}^uG\rtimes_{\widehat{\alpha^u}}\widehat{G}\big)
\cong \big(A\otimes \K, X\otimes \K\big)$$
and 
$$\big(A\rtimes_{\alpha}^rG\rtimes_{\widehat{\alpha^r}}\widehat{G}, X\rtimes_{\alpha}^rG\rtimes_{\widehat{\alpha^r}}\widehat{G}\big)
\cong \big(A\otimes \K, X\otimes \K\big)$$
which transfer the double dual actions $\widehat{\widehat{\alpha}}$ and (resp.  $\widehat{\widehat{\alpha^r}}$) to the diagonal 
action $\alpha\otimes \Ad\rho$, respectively, where $\rho:G\to U(L^2(G))$ denotes the right regular representation of $G$.
\end{theorem}
\begin{proof}
Both assertions can be deduced easily from the well known Takai-Takesaki duality theorem for the corresponding action 
on the universal $C^*$-hull $C_u^*(A,X)$. Indeed, it is shown in \cite[Theorem 5.1]{Rae} (in the even more general situation of a dual 
coaction of a twisted action) that for any action $\beta:G\to \Aut(B)$ the Imai-Takai isomorphism
$$\Psi: B\rtimes_\beta G\rtimes_{\widehat\beta} \widehat{G}\stackrel{\cong}{\to} B\otimes \K$$
is given by the integrated form $(\Lambda_B\rtimes \Lambda_G)\rtimes \Lambda_{\widehat{G}}$
where $\Lambda_B\rtimes\Lambda_G: B\rtimes_\beta G\to M(B\otimes\K(L^2(G)))$ is the regular representation 
of $B\rtimes_\beta G$ and $\Lambda_{\widehat{G}}=1\otimes M$. It is then clear that this factors through a homomorphism 
of $B\rtimes_{\beta,r} G\rtimes_{\widehat\beta} \widehat{G}$, which explains that both crossed products are the same.

Now, if we apply this to the system $(C_u^*(A,X), G,\alpha)$ we obtain the isomorphism
$$(\Lambda_{C_u^*(A,X)}\rtimes\Lambda_G)\rtimes\Lambda_{\widehat{G}}:
C_u^*(A,X)\rtimes_{\alpha_u} G\rtimes_{\widehat{\alpha_u}} \widehat{G}\stackrel{\cong}{\to} C_u^*(A,X)\otimes \K$$
which then clearly restricts to an isomorphism
$$(\Lambda_X\rtimes\Lambda_G)\rtimes\Lambda_{\widehat{G}}:
X\rtimes_{\alpha}^uG\rtimes_{\widehat{\alpha}} \widehat{G}\stackrel{\cong}{\to} X\otimes \K$$
and similarly for $X\rtimes_{\alpha}^rG\rtimes_{\widehat{\alpha_r}}\widehat{G}$.
The statement on the double dual action $\widehat{\widehat{\alpha}}$ follows from the analoguous statement for 
the double dual action $\widehat{\widehat{\alpha}}$ on  the double crossed product of $C_u^*(A,X)$.
\end{proof}

We now proceed to a discussion of Katayama duality, where we want to study double crossed products 
$$(A,X)\rtimes_{\delta_X}\widehat{G}\rtimes_{\widehat{\delta}_X}G$$
by dual actions of coactions. Note that here it will usually matter whether we take the universal or the 
reduced crossed product (or any exotic crossed product in between) on the outside, so we need to clarify this point.
So let us first recall the situation if we start with a coaction $\delta: B\to M(B\check\otimes C^*(G))$ of $G$ on a $C^*$-algebra $B$.

It is shown by Nilsen in \cite{Nilsen} that there exists a surjective $*$-homomorphisms
\begin{equation}\label{eq-max}
\Phi_B:
B\rtimes_{\delta}\widehat{G}\rtimes_{\widehat{\delta}, u}G\onto B\otimes \K(L^2(G))
\end{equation}
given by the integrated form
$$\Phi_B=\big (\Lambda_B\rtimes \Lambda_{\widehat{G}}\big)\rtimes(1\otimes \rho)$$
of the covariant homomorphism $(\Lambda_B\rtimes\Lambda_{\widehat{G}}, 1\otimes \rho)$
of $(B\rtimes_{\delta}\widehat{G}, G, \widehat{\delta})$ where $(\Lambda_B,\Lambda_{\widehat{G}})=\big((\id_B\otimes \lambda)\circ \delta,  1\otimes M\big)$ 
is the regular representation of $(B,G,\delta)$ into $M(B\otimes \K(L^2(G)))$.
A coaction $\delta$ is called {\em maximal}, if $\Phi$ is an isomorphism, and it is called {\em normal}, if 
$\Phi$ factors through an isomorphism $B\rtimes_{\delta}\widehat{G}\rtimes_{\widehat{\delta},r}G\cong B\otimes \K(L^2(G))$.
In general, the isomorphism $\Phi$ will factor through an isomorphism 
$$B\rtimes_{\delta}\widehat{G}\rtimes_{\widehat{\delta}, \mu}G \cong B\otimes \K(L^2(G))$$
of some {\em exotic crossed product} $B\rtimes_{\delta}\widehat{G}\rtimes_{\widehat{\delta}, \mu}G$ which lies between 
the maximal and the reduced crossed product in the sense that it is a $C^*$-completion of $C_c(G, B\rtimes_\delta\widehat{G})$ such that 
the identity map on $C_c(G, B\rtimes_\delta\widehat{G})$ induces surjective $*$-homomorphisms
\begin{equation}\label{eq-surjection}
B\rtimes_{\delta}\widehat{G}\rtimes_{\widehat{\delta}, u}G\onto B\rtimes_{\delta}\widehat{G}\rtimes_{\widehat{\delta}, \mu}G\onto B\rtimes_{\delta}\widehat{G}\rtimes_{\widehat{\delta}, r}G.
\end{equation}
It has been shown by Quigg  \cite{Q} that  $(B, G, \delta)$ is normal if and only if $\Lambda_B=(\id_B\otimes\lambda)\circ \delta: B\to M(B\otimes \K(L^2(G)))$ is faithful. In general the coaction $\delta$ determines a normal coaction $\delta_n$ (called the {\em normalization} of $\delta$)
on  the quotient $B_n:=B/(\ker\Lambda_B)$ such that the $\delta-\delta_n$ equivariant quotient map
$\Psi_n:B\onto B_n$ descents to an isomorphism of the dual systems
$$(B\rtimes_{\delta}\widehat{G}, G, \widehat{\delta})\cong (B_n\rtimes_{\delta_n}\widehat{G}, G, \widehat{\delta_n}).$$
If $(B,\delta)=(A\rtimes_{\alpha} G, \widehat{\alpha})$ is the dual coaction on the full crossed product by an action $\alpha$ of $G$
on a $C^*$-algebra $A$, then  $(B,\delta)$ is maximal (see \cite{EKQ}) and the normalization of $(B,\delta)$ is  
given by the pair $(B_n,\delta_n)=(A\rtimes_rG, \widehat{\alpha}_r)$, the dual coaction of the reduced crossed product.

We now want to extend this picture to crossed products by $C^*$-operator systems. We start with the following obvious 
consequence to the above:

\begin{theorem}\label{thm-Katayama}
Suppose that $(A,X,G, \delta_X)$ is a coaction of the locally compact group $G$ on the $C^*$-operator system $(A,X)$.
Then there is a canonical surjective morphism 
$$\Phi_X: \big(A\rtimes_{\delta_A}\widehat{G}\rtimes_{\widehat{\delta_A}}^uG, X\rtimes_{\delta_X}\widehat{G}\rtimes_{\widehat{\delta_X}}^uG\big)\onto 
\big(A\otimes\K(L^2(G)), X\otimes \K(L^2(G)\big)$$
given by the integrated form 
$$\Phi=\big (\Lambda_X\rtimes \Lambda_{\widehat{G}}\big)\rtimes(1\otimes \rho)$$
of the covariant homomorphism $(\Lambda_X\rtimes\Lambda_{\widehat{G}}, 1\otimes \rho)$
of $(X\rtimes_{\delta_X}\widehat{G}, G, \widehat{\delta_X})$ where $(\Lambda_X,\Lambda_{\widehat{G}})=\big((\id_X\otimes \lambda)\circ \delta_X,  1\otimes M\big)$ 
is the regular representation of $(A,X,G,\delta_X)$ into $M(X\otimes \K(L^2(G)))$.
Moreover, there is an exotic completion 
$(A\rtimes_{\delta_A}\widehat{G}\rtimes_{\widehat{\delta_A}}^{\mu}G, X\rtimes_{\delta_X}\widehat{G}\rtimes_{\widehat{\delta_X}}^{\mu}G)$ of the pair $(C_c(G, A\rtimes_{\delta_A} \widehat{G}), (C_c(G, X\rtimes_{\delta_X} \widehat{G})$, lying between the maximal and the reduced crossed products,
such that $\Phi_X$ factors through a completely isometric 
 isomorphism
 $$\big(A\rtimes_{\delta_A}\widehat{G}\rtimes_{\widehat{\delta_A}}^\mu G, X\rtimes_{\delta_X}\widehat{G}\rtimes_{\widehat{\delta_X}}^\mu G\big)\cong 
 \big(A\otimes\K(L^2(G)), X\otimes \K(L^2(G)\big).$$
\end{theorem}
\begin{proof} The result follows from the above cited results for coactions on $C^*$-algebras applied to the 
coaction $(B,G, \delta)=(C_u^*(A,X), G, \delta_u)$ and the restriction of the corresponding $*$-homomorphism $\Phi_B$ to $A\subseteq X\subseteq X$.
\end{proof}

\begin{corollary}\label{cor-katayama}
Suppose that $G$ is an amenable locally compact group. Then the morphism
$$\Phi_X: \big(A\rtimes_{\delta_A}\widehat{G}\rtimes_{\widehat{\delta_A}}^uG, X\rtimes_{\delta_X}\widehat{G}\rtimes_{\widehat{\delta_X}}^uG\big)\onto 
\big(A\otimes\K(L^2(G)), X\otimes \K(L^2(G)\big)$$
of Theorem \ref{thm-Katayama} is a completely isometric isomorphism.
\end{corollary}

\section{$C^*$-operator bimodules}\label{sec-bimodules}
In this section we want to study $C^*$-operator systems which are related to $C^*$-operator bimodules.
This will later lead to an easy way to define crossed products by group actions on 
$C^*$-operator bimodules. Since every operator space can be regarded as a $C^*$-operator bimodule 
in a canonical way, this will also give a construction of crossed products by group actions on operator spaces.

\begin{definition}\label{def-C*-operatorbimodule}
Let $H$ and $K$ be Hilbert spaces. 
A concrete {\em $C^*$-operator bimodule} $(A,V, B)$  inside $\mathcal B(K,H)$  consists of a norm closed 
subset $V\subseteq \mathcal B(K,H)$ together with a $C^*$-subalgebra $A\subseteq \mathcal B(H)$ 
and a $C^*$-subalgebra $B\subseteq \mathcal B(K)$
satisfying   $AV=V=VB$,  $AH=H$,  and $BK=K$.

A {\em representation} of the $C^*$-operator bimodule $(A,V,B)$ on a pair of Hilbert spaces 
 is $(K', H')$ is a triple of maps $\rho=(\rho_A, \rho_V, \rho_B)$
such that $\rho_A:A\to \mathcal B(H')$ and $\rho_B:B\to \mathcal B(K')$ are $*$-homomorphisms and 
$\rho_V: V\to \mathcal B(K',H')$   is a completely bounded map such that
$$\rho_V(avb)=\rho_A(a)\rho_V(v)\rho_B(b)\quad\forall a\in A, v\in V, b\in B.$$
We say that $\rho$ is {\em non-degenerate} if $\rho_A$ and $\rho_B$ are non-degenerate. 
We  say that $\rho$ is {\em completely contractive} if $\rho_V$ is completely contractive
and $\rho$ is called {\em completely isometric} if $\rho_A$ and 
$\rho_B$ are faithful and $\rho_V$ is completely isometric.

A {\em morphism} from the $C^*$-operator bimodule $(A,V,B)$ to the $C^*$-operator bimodule 
$(C, W,D)$ inside $\mathcal B(K', H')$ is a representation $\varphi=(\varphi_A,\varphi_V,\varphi_B)$
of $(A,V,B)$ to $\mathcal B(K',H')$ such that $\varphi_A(A)\subseteq C, \varphi_V(V)\subseteq W$, and $\varphi_B(B)\subseteq D$.
The invertible morphisms (or  isomorphisms) are then precisely the surjective completely isometric morphisms.
We shall often identify isomorphic $C^*$-operator systems.
\end{definition}

\begin{remark}\label{rem-ospace}
{(1)} Every concrete operator space$V\subseteq \mathcal B(K,H)$ determines the concrete $C^*$-operator system 
$(\C1_H, V, \C 1_K)$. If $V'\subseteq \mathcal B(K',H')$ is another operator space and 
$\varphi_V: V\to V'$ is a completely bounded map, then $\varphi=(\varphi_{\C1_H}, \varphi_V, \varphi_{\C 1_K})$ with 
$\varphi_{1_H}(\lambda 1_H)=\lambda 1_{H'}$ and $\varphi_{1_K}(\lambda 1_K)=\lambda 1_{K'}$  
is an morphism from $(\C 1_H, V, \C1_K)\to (\C 1_{H'}, V', \C1_{K'})$. Hence $V$ and $V'$ are  isomorphic as operator spaces
if and only if $(\C 1_H, V, \C1_K)$ and $(\C 1_{H'}, V', \C1_{K'})$ are isomorphic as $C^*$-operator bimodules. 
In this way we may regard the category of (concrete) $C^*$-operator bimodules as an extension of the category of 
(concrete) operator spaces.
\\
{(2)} If $(A,V, B)$ is a triple of subsets of $\big(\mathcal B(H), \mathcal B(K,H) , \mathcal B(K)\big)$ 
which satisfies all requirements of a $C^*$-operator bimodule 
as in Definition \ref{def-C*-operatorbimodule} except the non-degeneracy requirements $AH=H$ and $BK=K$,  
 let $H':=AH\subseteq H$ and $K'=BK\subseteq K$. Then, via restriction, we obtain a completely isomeric and non-degenerate representation of 
 $(A,V,B)$ on $\mathcal B(K',H')$.
\\
{\bf (3)} If $\rho=(\rho_A, \rho_V, \rho_B)$ is a completely bounded representation (or morphism) of $(A,V,B)$ with $0<C:=\|\rho_V\|_{cb}$,
then $\frac{1}{C}\rho:=(\rho_A, \frac{1}{C}\rho_V, \rho_B)$ is a completely contractive representation (resp. morphism).
This easy observation shows that in most situations one may assume without loss of generality that $\rho$ is completely contractive.
\end{remark}

The following proposition extends the well-known construction which assigns to each operator space $V\subseteq \mathcal B(K,H)$ 
the Paulsen-operator system $ X(V):=\left(\begin{matrix} \C 1_H &V\\ V^* & \C 1_K\end{matrix}\right)\subseteq \mathcal B(H\oplus K)$.
For this 
let $(A,V, B)$ be a $C^*$-operator bimodule in $\mathcal B(H,K)$. Let
$$X(A,V,B):=\left\{\left(\begin{matrix} a&v\\ w^*&b\end{matrix}\right): a\in A, v,w\in V, b\in B\right\}\subseteq  \mathcal B(H\oplus K),$$
and let $A\oplus B$ be viewed as the set of diagonal operators 
$\left(\begin{matrix}a&0\\0&b\end{matrix}\right)\in \mathcal B(H\oplus K)$ with $a\in A, b\in B$.
Then it is easily checked that $(A\oplus B, X(A,V,B))$ is a $C^*$-operator system in $\mathcal B(H\oplus K)$ as defined in Definition \ref{def-cstaros}.

On the other hand, one easily checks that the set of operators
$$\Op(A,V,B):=\left\{\left(\begin{matrix} a&v\\ 0&b\end{matrix}\right): a\in A, v\in V, b\in B\right\}\subseteq  \mathcal B(H\oplus K),$$
is a concrete operator algebra in $\mathcal B(H\oplus K)$ such that each approximate unit of $A\oplus B$ serves as an approximate unit of 
$\Op(A,V,B)$.

\begin{definition}\label{def-Paulsen}
We call $(A\oplus B, X(A,V,B))$ the {\em Paulsen $C^*$-operator system} of $(A,V, B)$ and we call $\Op(A,V,B)$ the {\em Paulsen operator algebra} of $(A,V,B)$. 
\end{definition}

\begin{proposition}\label{prop-reps}
Let $(A,V,B)$ be a $C^*$-operator bimodule. Then there is a one-to-one correspondence between
\begin{enumerate}
\item non-degenerate completely contractive representations of $(A,V,B)$;
\item non-degenerate completely positive contractive representations of the $C^*$-operator system  $\big(A\oplus B, X(A,V,B)\big)$; and
\item non-degenerate completely contractive operator algebra representations of $\Op(A,V,B)$.
\end{enumerate}
Given a representation $\rho=(\rho_A, \rho_V,\rho_B):(A,V,B)\to \mathcal B(K, H)$ the 
corresponding representation $\pi$ of $X(A,V, B)$ into $\mathcal B(H\oplus K)$ 
 is given by
$$\pi\Big(\left(\begin{matrix} a&v\\ w^*& b\end{matrix}\right)\Big)= \left(\begin{matrix} \rho_A(a)& \rho_V(v)\\ \rho_V(w)^*& \rho_B(b)\end{matrix}\right)\quad  a\in A, v,w\in V, b\in B.$$
and given a representation $\pi: X(A,V,B)\to \mathcal B(L)$, the corresponding representation of $\Op(A,V,B)$ is given by the restriction 
of $\pi$ to $\Op(A,V,B)\subseteq X(A,V,B)$. 
%
%
%
\end{proposition}

\begin{proof} Let $\rho=(\rho_A,\rho_V,\rho_B)$ be a completely contractive representation of $(A,V,B)$ into $\mathcal B(K',H')$.
We may further assume without loss of generality that $A$ and $B$ are unital -- otherwise we 
replace $A$ and $B$ by their unitisations $\tilde{A}=A+\C 1_H$ and $\tilde{B}=B+ \C1_K$ and $\rho_A$ and $\rho_B$ by their canonical 
unital extensions to $\tilde{A}$ and $\tilde{B}$, respectively. 

Suppose now that $T:=\left(\begin{matrix} a&v\\ w^*& b\end{matrix}\right)\in X(A,V,B)$ is positive. 
Then $w=v$ and $a$ and $b$ are 
positive elements in $A$ and $B$, respectively. Let $\pi$ be as in the proposition. In order to see that $\pi(T)$ is positive, it suffices to show that 
$\pi(T+\eps 1)=\pi(T)+\eps 1$ is positive for all $\eps>0$. Writing $a_\eps:=a+\eps1$ and $b_\eps:=b+\eps1$, we get
\begin{align*}
0&\leq \left(\begin{matrix} a_\eps^{-\frac{1}{2}}&0\\ 0& b_\eps^{-\frac{1}{2}}\end{matrix}\right)\left(\begin{matrix} a_\eps&v\\ v^*& b_\eps\end{matrix}\right)\left(\begin{matrix} a_\eps^{-\frac{1}{2}}&0\\ 0& b_\eps^{-\frac{1}{2}}\end{matrix}\right)\\
&=\left(\begin{matrix} 1&a_\eps^{-\frac{1}{2}}vb_\eps^{-\frac{1}{2}}\\ b_\eps^{-\frac{1}{2}}v^*a_\eps^{-\frac{1}{2}}& 1\end{matrix}\right),
\end{align*}
from which it follows that $\|a_\eps^{-\frac{1}{2}}vb_\eps^{-\frac{1}{2}}\|\leq 1$. Since $\rho_V$ is completely contractive, we 
also have $\|\rho_V(a_\eps^{-\frac{1}{2}}vb_\eps^{-\frac{1}{2}})\|\leq 1$. It then follows that 
\begin{align*}
&\pi(T+\eps 1)\\
 &=\left(\begin{matrix} \rho_A(a_\eps^{\frac{1}{2}})&0\\ 0& \rho_B(b_\eps^{\frac{1}{2}})\end{matrix}\right)
\left(\begin{matrix} 1&\rho_V(a_\eps^{-\frac{1}{2}}vb_\eps^{-\frac{1}{2}})\\ \rho_V(b_\eps^{-\frac{1}{2}}v^*a_\eps^{-\frac{1}{2}})& 1\end{matrix}\right)
\left(\begin{matrix} \rho_A(a_\eps^{\frac{1}{2}})&0\\ 0& \rho_B(b_\eps^{\frac{1}{2}})\end{matrix}\right)
\end{align*}
is positive as well. A similar computation performed an matrix algebras over $X(A,B,V)$ then shows that 
$\pi:X(A,V,B)\to \mathcal B(H\oplus K)$ is completely positive. Since $\pi$ is unital, it is also completely contractive.

It is clear that every non-degenerate completely contractive representation of $X(A,B,V)$ restricts to a nondenerate  completely contractive operator algebra 
 representation of $\Op(A,V,B)$. 
 
 So let us finally assume that we have a non-degenerate completely contractive operator algebra representation 
 $\pi: \Op(A,V,B)\to \mathcal B(L)$ for some Hilbert space $L$. Let us regard $A, V$ and $B$ as subspaces of $\Op(A,V,B)$ in the 
 canonical way. Then the restrictions $\pi_A, \pi_V, \pi_B$ of $\pi$ to $A,V$ and $B$ are completely contractive as well.
 Since $\pi_A:A\to \mathcal B(L)$ and $\pi_B:B\to \mathcal B(L)$ are contractive algebra homomorphisms, it follows from 
 \cite[Proposition A.5.8]{Blecher1} that they are $*$-homomorphisms. Writing $H:=\pi_A(A)L$ and $K=\pi_B(B)L$ we get $L=H\oplus K$ 
 and $(\pi_A,\pi_V, \pi_B)$ is a non-degenerate representation of $(A,V,B)$ into $\mathcal B(H,K)$ as in Definition \ref{def-C*-operatorbimodule}.
\end{proof}

\begin{remark}\label{rem-nondeg}
If we allow possibly degenerate representations of $(A,V,B)$, $X(A,V,B)$ or $\Op(A,V,B)$ in the statement of Proposition \ref{prop-reps} then 
we can always pass to appropriate subspaces of the representation  spaces to make these representations non-degenerate.
Then the one-to-one correspondence will still hold modulo the possible addition of direct sums on which all operators act trivially.
\end{remark}

As a direct consequence of Proposition \ref{prop-reps} we now get

\begin{corollary}\label{cor-morphisms}
Suppose that $(A,V,B)$ and $(C,W,D)$ are $C^*$-operator bimodules. Then there is a one-to-one correspondence between
\begin{enumerate}
\item completely contractive morphism  $\varphi:(A,V,B)\to (C,W,D)$,
\item completely positiv contractive morphism $\phi:X(A,V,B)\to X(C,W,D)$ preserving the corners, and
\item complete contractive homomorphisms $\psi:\Op(A,V,B)\to\Op(C,W,D)$ preserving the corners.
\end{enumerate}
If $\varphi=(\varphi_A,\varphi_V,\varphi_B)$ is as morphism from $(A,V,B)$ to $(C,W,D)$
then the corresponding morphism $\phi:X(A,V,B)\to X(C,W,D)$ is given by 
$$\phi\Big(\left(\begin{matrix} a&v\\ w^*& b\end{matrix}\right)\Big)= \left(\begin{matrix} \varphi_A(a)& \varphi_V(v)\\ \varphi_V(w)^*& \varphi_B(b)\end{matrix}\right)\quad  a\in A, v,w\in V, b\in B,$$
and if $\phi:X(A,V,B)\to X(C,W,D)$ is a morphism as in (2), then its restriction $\psi$ to $\Op(A,V,B)$ is the corresponding morphism 
from $\Op(A,V,B)$ to $\Op(C,W,D)$.

The correspondences are compatible with taking compositions of morphisms.
\end{corollary}
\begin{proof} It is clear that the constructions given above preserve all required algebraic properties. The combination of 
Proposition \ref{prop-reps} with Remark \ref{rem-nondeg} shows that they also preserve the property of being completely contractive.
\end{proof} 

\begin{remark}\label{rem-C*-hull-bimodule}
In what follows it is useful to consider representations of $C^*$-operator bimodules into general $C^*$-algebras.
By such a representation we understand a triple $\rho=(\rho_A,\rho_V,\rho_B)$ of $(A,V,B)$ into a 
$C^*$-algebra $C$ satisfying
\begin{enumerate}
\item $\rho_A:A\to C$ and $\rho_B:B\to C$ are $*$-homomorphisms such that $\rho_A(A) \rho_B(B)=\{0\}$,
\item $\rho_V:V\to C$ is completely contractive and $\rho_V(bva)=\rho_B(b)\rho_V(v)\rho_A(a)$ for 
all $a\in A, v\in V$ and $b\in B$.
\end{enumerate}
Then we have a well-defined $*$-homomorphism $\rho_A\oplus \rho_B:A\oplus B\to C$ mapping $a\oplus b$ to $\rho_A(a)+\rho_B(b)$
and we say that $\rho$ is non-degenerate if $\rho_A\oplus \rho_B$ maps approximate units of $A\oplus B$ to approximate units of $C$
(this is equivalent to the fact that $\rho_A\oplus\rho_B(A\oplus B)C=C$). In general we may always pass to the subalgebra $C'$ of $C$ generated 
 by $\rho_A(A)\cup\rho_V(V)\cup\rho_B(B)$ to obtain a non-degenerate representation into this subalgebra. 
Then, representing $C$ (resp. $C'$) faithfully and non-degenerately on a Hilbert space $L$, we may regard $\rho$ as a non-degenerate
representation of $(A,V,B)$ in $\mathcal B(K,H)$ with $H=\rho_A(A)L, K=\rho_B(B)L$. 
This allows us  to use the results of Proposition \ref{prop-reps} also for representations into $C^*$-algebras.

Note that conversely any triple of subsets $(A,V,B)$ of a $C^*$-algebra $C$ such that $A$ and $B$ are $C^*$-subalgebras of $C$, 
$V$ is a closed subspace of $C$, $AB=\{0\}$ and  $AV=V=VB$ determines the  structure of a $C^*$-operator bimodule  on $(A,V,B)$ 
via a faithful representation of $C$ on Hilbert space. 
\end{remark}

\begin{definition}\label{def-C*-hull}
Let $(A,V,B)$ be a $C^*$-operator bimodule.
 If $j=(j_A,j_V, j_B)$ is a completely isometric representation of $(A,V,B)$ into a $C^*$-algebra $C$ such that 
$C$ is generated  by $j_A(A)\cup j_V(V)\cup j_B(B)$ as a $C^*$-algebra, then $\big(C, (j_A,j_V,j_B)\big)$
is called a {\em $C^*$-hull} of $(A,V,B)$. 

A $C^*$-hull  $\big(C_u^*(A,V,B), (i_A,i_V,i_B)\big)$ is called {\em universal}
if for any completely contractive representation $\rho=(\rho_A, \rho_V,\rho_B)$ of $(A,V,B)$ into some $C^*$-algebra $D$
there  exists a $*$-homomorphism 
$$\rho_C:C_u^*(A,V,B)\to D\; \text{such that} \; \rho_A=\rho_C\circ i_A, \;\rho_V=\rho_C\circ i_V,\;\text{and}\; \rho_B=\rho_C\circ i_B.$$
On the other hand, we say that a $C^*$-hull $\big(C_e^*(A,V,B), (k_A,k_V,k_B)\big)$ is  {\em enveloping}  if 
for any other $C^*$-hull $\big(C, (j_A, j_V, j_B)\big)$ there exists a $*$-homomor\-phism
$$k_C:C\to C_e^*(A,V,B)\;\text{such that}\; k_A=k_C\circ j_A, k_V=k_C\circ j_V,\;\text{and}\; k_B=k_C\circ j_B.$$
(The $*$-homomorphisms $\rho_C$ and $k_C$ are then uniquely determined by these properties.) 
\end{definition}

As a consequence, if the universal and the enveloping $C^*$-hulls exist, then for any $C^*$-hull $\big(C, (j_A, j_V, j_B)\big)$ of $(A,V,B)$
we obtain unique surjective $*$-homomorphisms
$$C_u^*(A,V,B)\onto C\onto C_e^*(A,V,B)$$
which commute with the embeddings of $(A, V,B)$ into these $C^*$-algebras. Moreover, it follows easily from the universal properties that
the universal and enveloping $C^*$-hulls are unique up to isomorphism which are compatible with the embeddings of $(A,V,B)$.

\begin{proposition}\label{prop-C*-hull-bimodule}
For each $C^*$-operator bimodule $(A,V,B)$ the universal and enveloping $C^*$-hulls exist.  To be more
precise: let 
$$\big(C_u^*(X(A,V,B)), i_{X(A,V,B)}\big)\quad\text{and}\quad \big(C_e^*(X(A,V,B)), k_{X(A,V,B)}\big)$$
 denote 
the universal and enveloping $C^*$-hulls of the $C^*$-operator system $X(A,V,B)$ and let
$(i_A,i_V,i_B)$ and $(k_A, k_V, k_B)$ be the compositions of $i_{X(A,V.B)}$ and $k_{X(A,V,B)}$ with the 
canonical inclusions of $(A,V,B)$ into $X(A,V,B)$. Then 
$$\big(C_u^*(X(A,V,B)),(i_A,i_V,i_B)\big)\quad\text{and}\quad \big(C_e^*(X(A,V,B)), (k_A, k_V, k_B)\big)$$
are the universal and enveloping $C^*$-hulls of $(A,V,B)$.

Alternatively, let 
$$\big(C_u^*(\Op(A,V,B)), i_{\Op(A,V,B)}\big)\quad\text{and}\quad\big(C_e^*(\Op(A,V,B)), k_{\Op(A,V,B)}\big)$$ denote 
the universal and enveloping $C^*$-hulls of the operator algebra \linebreak
$\Op(A,V,B)$ as in \cite[Propositions 2.4.2 and 4.3.5]{Blecher1} and let
$(i_A,i_V,i_B)$ and $(k_A, k_V, k_B)$ be the compositions of $i_{\Op(A,V.B)}$ and $k_{\Op(A,V,B)}$ with the 
canonical inclusions of $(A,V,B)$ into $\Op(A,V,B)$. Then 
$$\big(C_u^*(\Op(A,V,B)),(i_A,i_V,i_B)\big)\quad\text{and}\quad\big(C_e^*(\Op(A,V,B)), (k_A, k_V, k_B)\big)$$
are the universal and enveloping $C^*$-hulls of $(A,V,B)$.
\end{proposition}
\begin{proof}
The proof is an easy consequence of the definitions of the universal and enveloping $C^*$-hulls together 
Proposition \ref{prop-reps} and Remark \ref{rem-C*-hull-bimodule}. So we omit further details.
\end{proof}

We are now turning our attention to multipliers:

\begin{definition}\label{def-multopbimodule}
Let $(A,V,B)$ be a $C^*$-operator bimodule which is non-degenerately and completely isometrically 
represented on $\mathcal B(K,H)$.
Then the {\em multiplier bimodule} of $(A,V,B)$ is   the triple $\big(M(A), M(V), M(B)\big)$ in which
$M(A)$ and $M(B)$ are the multiplier algebras of the $C^*$-algebras $A$ and $B$, respectively, and where
$$M(V)=\{T\in \mathcal B(K,H): AT\cup TB\subseteq V\}.$$
\end{definition}
\begin{remark} One easily checks that $\big(M(A), M(V), M(B)\big)$ is again a $C^*$-operator bimodule represented 
on $\mathcal B(K,H)$. If one of $A$ or $B$ is unital, we clearly have $M(V)=V$.

Notice that, similarly to the construction of the multiplier $C^*$-operator system $(M(A), M(X))$ for a given 
$C^*$-operator system $(A,X)$, the space $M(V)$ heavily depends on the algebras $A$ and $B$. 
We retain from using a notation like $_AM_B(V)$ to keep things simple. 
\end{remark}

Recall (e.g., from \cite{Blecher1}) that for any completely isometrically and faithfully represented operator algebra $\mathcal A\subseteq \mathcal B(L)$ 
the multiplier algebra $M(\mathcal A)$ can be defined (up to completely isometric isomorphism) as 
$$M(\mathcal A)=\{T\in \mathcal B(L): T\mathcal A\cup \mathcal AT\subseteq \mathcal A\}.$$
Recall also the definition of the multiplier system of a $C^*$-operator system as given in Lemma \ref{lem-multipliers}.

\begin{proposition}\label{prop-multiplier}
Let $\big(M(A), M(V), M(B)\big)$ be the muliplier bimodule of the $C^*$-operator bimodule $(A,V,B)$ in $\mathcal B(K,H)$.
Then
$$ \big(M(A\oplus B), M(X(A,V,B))\big) =\big(M(A)\oplus M(B), X\big(M(A), M(V), M(B)\big)\big)$$
is the multiplier system 
of the $C^*$-operator system
$\big(A\oplus B, X(A,V,B)\big)$ and $\Op\big(M(A), M(V), M(B)\big)=M\big(\Op(A,V,B)\big)$.
\end{proposition}
\begin{proof} Recall from Lemma \ref{lem-multipliers} that $M(X(A,V,B))$ is defined as the set of all  elements 
$T\in \mathcal B(H\oplus K)$ such that $(A\oplus B)T\cup T(A\oplus B)\subseteq X(A,V,B)$.
Writing 
 $$T=\left(\begin{matrix}T_{11} & T_{12}\\ T_{21}& T_{22}\end{matrix}\right)\in \left(\begin{matrix}\mathcal B(H)& \mathcal B(K,H)\\
 \mathcal B(H,K)&  \mathcal B(K)\end{matrix}\right)$$
 and  computing  $\diag(a,0)T, T\diag(a,0), \diag(0,b)T, T\diag(0,b)\in X(A,V, B)$ easily shows that
$T\in \left(\begin{matrix} M(A)&M(V)\\ M(V)^*&M(B)\end{matrix}\right)=X\big(M(A), M(V), M(B)\big)$. 
Conversely one easily checks that $X\big(M(A), M(V), M(B)\big)\subseteq M(X(A,V,B))$.
A similar argument also shows $\Op\big(M(A), M(V), M(B)\big)=M\big(\Op(A,V,B)\big)$.
\end{proof}

\begin{definition}\label{def-genmor}
Let $(A,V,B)$ and $(C,W,D)$ be two $C^*$-operator bimodules. A {\em generalized morphism} from 
$(A,V,B)$ to $(C,W,D)$ is a completely contractive morphism 
$\varphi:(A,V,B)\to (M(C), M(W), M(D))$ such that 
$\varphi_A(A)C=C$ and $\varphi_B(B)D=D$.
\end{definition}

\begin{example}
Every non-degenerate representation $\rho$ of $(A,V,B)$ to some $\mathcal B(K,H)$ can be regarded as a 
generalized morphism from $(A,V,B)$ to the $C^*$-operator bimodule 
$\Big(M(\mathcal K(H)), M(\mathcal K(K,H)), M(\mathcal K(K))\Big)$.
\end{example}

The following proposition is now a direct combination of Proposition \ref{prop-reps}, Proposition \ref{prop-multiplier}
and Lemma \ref{lem-non-deg-morphism}, so we leave the details to the reader:

\begin{proposition}\label{prop-genmorextend}
Every generalized morphism 
$$\varphi:(A,V,B)\to (M(C), M(W), M(D))$$
 from $(A,V,B)$ to $(C,W,D)$ 
extends uniquely to a morphism 
$$\tilde\varphi:(M(A), M(V), M(B))\to (M(C), M(W), M(D)).$$
If $\varphi$ is completely isometric, then so is $\tilde\varphi$. 
In particular, every non-degenerate (completely isometric) representation $\rho$ of $(A,V,B)$ into $\mathcal B(K,H)$ uniquely extends 
to a (completely isometric) representation of $\big(M(A), M(V), M(B)\big)$ to $\mathcal B(K,H)$. 
\end{proposition}

We close this section with the following analogue of Lemma \ref{lem mult univ}:

\begin{proposition}\label{prop-multuniv}
Let $\big(C, (j_A, j_V, j_B)\big)$ be a $C^*$-hull of $(A,V,B)$. Then 
the inclusions $j=(j_A, j_V, j_B): (A,V,B)\to C$ extend to a completely isometric morphism
$$\bar{j}:=(\bar{j}_{M(A)}, \bar{j}_{M(V)}, \bar{j}_{M(B)}): (M(A), M(V), M(B))\to M(C)$$
such that $\bar{j}_{M(A)}(M(A))\cap C=j_A(A)$ and $\bar{j}_{M(B)}(M(B))\cap C=j_B(B)$. 
\end{proposition}
\begin{proof}
Let $j_{X(A,V,B)}: X(A,V,B)\to C$ denote the corresponding completely isometric representation of 
$X(A,V,B)$ into $C$. Then $\big(C, j_{X(A,V,B)}\big)$ is a $C^*$-hull of $X(A,V,B)$ and by Lemma 
 \ref{lem mult univ} there exists a unique extension 
 $$\big(\bar{j}_{M(A\oplus B)}, \bar{j}_{M(X(A,V,B))}\big): \big(M(A\oplus B), M(X(A,V,B)\big)\to M(C)$$
 such that $(\bar{j}_{M(A\oplus B)}(M(A)\oplus B))\cap C=A\oplus B$. 
 The result now easily follows from an application of Proposition \ref{prop-multiplier}.
 \end{proof}

\section{Crossed products by  $C^*$-operator bimodules}

Let $(A,V,B)$ be a $C^*$-operator bimodule and let $\Aut(A,V,B)$ denote the group of completely isometric isomorphisms of $(A,V,B)$ to $(A,V,B)$.
A {\em continuous action}  of $G$ on $(A,V,B)$ is a homomorphism $\alpha:G\to\Aut(A,V,B)$ 
with $\alpha_g=(\alpha_g^A, \alpha_g^V,\alpha_g^B)$ such that all components 
 $g\mapsto \alpha_g^A(a), \alpha_g^V(v), \alpha_g^B(b)$ are continuos for all $a\in A, v\in V, b\in B$.
 We then call $\big((A,V,B),G,\alpha\big)$  a {\em $C^*$-operator bimodule dynamical system}.
 
 A {\em covariant morphism} of $\big((A,V,B),G,\alpha\big)$  to the $C^*$-operator bimodule $(C,W,D)$ is 
 a quintuple $(\rho_A, \rho_V, \rho_B, u, v)$ such that $(\rho_A, \rho_V,\rho_B)$ is 
 a morphism from $(A,V,B)$ into $(C,W,D)$, $u:G\to UM(C)$ and $v:G\to UM(D)$ are 
 strictly continuous unitary representations of $G$ such that 
  $(\rho_A, u)$ and $(\rho_B,v)$ satisfy the usual cavariance conditions for the actions $\alpha^A$ and $\alpha^B$,
 respectively, and  such that for all $v\in V$ and $g$ in $G$ we have
  $$\rho_V(\alpha_g(v))=u_g\rho_V(v)v_{g^{-1}}.$$
A {\em generalized covariant morphism} of $\big((A,V,B),G,\alpha\big)$  to $(C,W,B)$ is a 
covariant morphism $(\rho_A, \rho_V, \rho_B, u, v)$ into $\big(M(C), M(W), M(D)\big)$ such that 
$(\rho_A, \rho_V, \rho_B):(A,V,B)\to \big(M(C), M(W), M(D)\big)$ is a generalized morphism in the sense
of Definition \ref{def-genmor}. 

A {\em covariant representation} of $(A,V,B)$ is a covariant morphism into\linebreak
$\big(\mathcal B(H), \mathcal B(K,H), \mathcal B(K)\big)$ for some pair of Hilbert spaces $(H,K)$.

If $(\rho_A, \rho_V, \rho_B, u, v)$ is covariant morphism of $(A,V,B)$ into $(C,W,D)$, we 
have {\em integrated forms} 
$$(\rho_A\rtimes u,   u\ltimes\rho_V\rtimes v , \rho_B\rtimes v): \big(C_c(G,A), C_c(G,V), C_c(G,B)\big)\to (C,W,D)$$
in which $\rho_A\rtimes u$ and $\rho_B\rtimes v$ are the usual integrated forms of the covariant homomorphisms 
 $(\rho_A, u)$ and $(\rho_B,v)$  of the systems $(A,G,\alpha^A)$ and $(B,G,\alpha^B)$, respectively, and 
where   $u\ltimes \rho_V\rtimes v :C_c(G,V)\to W$ is given by 
$$u\ltimes \rho_V\rtimes v(f)=\int_G \rho_V(f(s))v_s\, ds =\int_G u_s\rho_V\big(\alpha_{s^{-1}}^V(f(s))\big)\, ds,$$
where the right equation follows from the covariance condition (this should explain the notation $v\ltimes \rho_V\rtimes u$).
We have the usual convolution products on $C_c(G,A)$ and $C_c(G,B)$ and obvious convolution formulas for pairings 
$$C_c(G,A)\times  C_c(G,V)\to C_c(G,V) \quad \text{and}\quad C_c(G,V)\times C_c(G,B)\to C_c(G,V)$$
which are preserved by the integrated form $(\rho_A\rtimes u, v\ltimes \rho_V\rtimes u, \rho_B\rtimes v)$ 
of $(\rho_A, \rho_V, \rho_B, u, v)$. 

 The following is a direct consequence of Corollary \ref{cor-morphisms} and the universal properties 
 of the universal and the enveloping $C^*$-hulls of $(A,V,B)$. 
 
 \begin{proposition}\label{prop-actions}
 Let $(A,V,B)$ be a $C^*$-operator bimodule and let \linebreak
 $\big(C_u^*(A,V,B), (i_A, i_V, i_B)\big)$ and $\big(C_e^*(A,V,B), (k_A, k_V, k_B)\big)$  denote 
 the universal and enveloping $C^*$-hulls of $(A,V,B)$, respectively. Then there is a canonical 
 one-to-one correspondence between
 \begin{enumerate}
 \item continuous actions $\alpha:G\to \Aut(A,V,B)$,
 \item continuous actions $\alpha^X: G\to \Aut\Big(\big(A\oplus B, X(A,V,B)\big)\Big)$ which preserve the corners, 
 \item continuous operator algebra actions $\alpha^{\Op}:G\to \Aut\big(\Op(A,V,B)\big)$ which preserves the corners,
 \item continuous actions $\alpha^u:G\to \Aut(C_u^*(A,V,B))$ by automorphisms which preserve the subspaces $i_A(A), i_V(V)$ and $i_B(B)$. 
 \item continuous actions $\alpha^e:G\to \Aut(C_e^*(A,V,B))$ by automorphisms which preserve the subspaces $k_A(A), k_V(V)$ and $k_B(B)$.  \end{enumerate}
 Moreover, there are one-to-one correspondences between the covariant representations (morphisms) 
 $(\rho_A, \rho_V, \rho_B, u, v)$ of  $\big((A,V,B),G,\alpha\big)$ and covariant representations (morphisms) 
 of the actions in (2), (3), and (4) above  via the known correspondence for representations (morphisms)
 of $(A,V,B)$ and $X(A,V,B), \Op(A,V,B)$ and $C_u^*(A,V,B)$ with unitary parts given by the direct sum $u\oplus v$.
 \end{proposition}
 
 We now give the definition of the universal crossed product by an action of a locally compact group $G$ on a 
 $C^*$-operator bimodule:
 
 \begin{definition}\label{def-full-crossed-C*-bimodule}
 Let $\alpha:G\to \Aut(A,V,B)$ be a strongly continuous action of the locally compact group $G$.
 We define the universal crossed 
 product 
 $$(A,V,B)\rtimes_{\alpha}^uG:= \big(A\rtimes_\alpha^uG, V\rtimes_\alpha^uG, B\rtimes_\alpha^uG\big)$$ 
 of $(A,V,B)$ by $G$ as the  respective closures 
 of $\big(C_c(G, A), C_c(G,V), C_c(G,B)\big)$ inside $C_u^*(A,V,B)\rtimes_{\alpha, u} G$ (here we identify $A,V$ and $B$  with $i_A(A), i_V(V)$ and $i_B(B)$   inside $C_u^*(A,V,B)$, respectively).
 \end{definition}
 
 To see that $(A,V,B)\rtimes_{\alpha}^uG$ has a canonical structure of a $C^*$-operator bimodule, let 
 $(i_{C_u^*(A,V,B)}, i_G)$ denote the universal representation of  
 the crossed product $C_u^*(A,V,B)\rtimes_{\alpha,u}G$ an the Hilbert space $L_u$, i.e., 
 the direct some of all GNS representations associated to the states of $C_u^*(A,V,B)\rtimes_{\alpha,u}G$.
 Then there  exists a decomposition $L_u=H_u\oplus K_u$ such that $i_{C_u^*(A,V,B)}$ restricts to  
a representation of $(A,V,B)$ on $\mathcal B(K_u, H_u)$ . It is then easy to check that the covariant 
representation $(i_{C_u^*(A,V,B)},  i_G)$ restricts to the covariant representation 
$(i_A, i_V, i_B, i_G^A, i_G^B)$ of $\big((A,V,B),G,\alpha\big)$ where $i_A, i_V, i_B$  denote the
 restrictions  of
$i_{C_u^*(A,V,B)}$ to $A,V$ and $B$, respectively (viewed as operators in the respective corners of $\mathcal B(H_u\oplus K_u)$),
and $i_G^A:=i_G|_{H_u}$,  $i_G^B:=i_G|_{K_u}$. 

The integrated form $i_{C_u^*(A,V,B)}\rtimes  i_G$ restricts to the integrated forms
$i_A\rtimes i_G^A: C_c(G,A)\to \mathcal B(H_u)$, 
$i_G^B\ltimes i_V\rtimes i_G^A: C_c(G,V) \to\mathcal B(K_u, H_u)$ and 
$i_B\rtimes i_G^B: C_c(G,B)\to \mathcal B(K_u)$ and similarly for the respective 
completions.  They therefore extend to a completely isometric representation of 
$\big(A\rtimes_\alpha^uG, V\rtimes_\alpha^uG, B\rtimes_\alpha^uG\big)$ as a concrete $C^*$-operator bimodule 
in $\mathcal B(K_u, H_u)$.
Moreover, it is easy to check that $(i_A, i_V, i_B, i_G^A, i_G^B)$ take  their values in 
$\big(M(A\rtimes_\alpha^uG), M(V\rtimes_\alpha^uG), M(B\rtimes_\alpha^uG)\big)$.
Hence $(i_A, i_V, i_G, i_B, i_G^A, i_G^B)$ is a generalized covariant morphism 
of $\big((A,V,B),G,\alpha\big)$ into $\big(M(A\rtimes_\alpha^uG), M(V\rtimes_\alpha^uG), M(B\rtimes_\alpha^uG)\big)$ 
which integrates to the identity on $\big(A\rtimes_\alpha^uG, V\rtimes_\alpha^uG, B\rtimes_\alpha^uG\big)$.
We call $(i_A, i_V, i_G, i_B, i_G^A, i_G^B)$ the universal morphism of $\big((A,V,B),G,\alpha\big)$. 

The following proposition shows that $(A,V,B)\rtimes_\alpha^uG$ has the right universal properties for
covariant morphisms (representations) of $\big((A,V,B),G,\alpha\big)$. 
\begin{proposition}\label{prop-universalAVB}
Let $\alpha:G\to \Aut(A,V,B)$ be a continuous action.
For every generalized covariant morphism $(\rho_A,\rho_V,\rho_B, u,v)$ of $\big((A,V,B),G,\alpha\big)$ into 
$\big(M(C), M(W), M(D)\big)$ the integrated form 
$(\rho_A\rtimes u, u\ltimes\rho_V\rtimes v, \rho_B\rtimes v)$  from  $(C_c(G,A), C_c(G,V), C_c(G,D))$ into $(M(C),M(W),M(D))$
extends uniquely to a morphism
$$(\rho_A\rtimes u, u\ltimes\rho_V\rtimes v, \rho_B\rtimes v): (A\rtimes_\alpha^uG, V\rtimes_\alpha^uG, B\rtimes_\alpha^uG)\to (M(C),M(W),M(D)).$$
(which takes values in $(C,W,D)$  if $(\rho_A, \rho_V,\rho_B)$ does). If  $(\rho_A,\rho_V,\rho_B, u,v)$ is non-degenerate, then 
so is $(\rho_A\rtimes u, u\ltimes\rho_V\rtimes v, \rho_B\rtimes v)$. 

Conversely, for every generalized morphism
$(\pi_{A\rtimes_\alpha^uG}, \pi_{V\rtimes_{\alpha}^uG}, \pi_{B\rtimes_\alpha^uG})$ of 
$(A\rtimes_\alpha^uG, V\rtimes_\alpha^uG, B\rtimes_\alpha^uG)$ into $(M(C),M(W),M(D))$ there is a unique  
generalized covariant morphism
$(\rho_A,\rho_V\rho_B,u,v)$  of $\big((A,V,B),G,\alpha\big)$ into  \linebreak 
$\big(M(C), M(W), M(D)\big)$ such that
$$(\pi_{A\rtimes_\alpha^uG}, \pi_{V\rtimes_{\alpha}^uG}, \pi_{B\rtimes_\alpha^uG})=(\rho_A\rtimes u, u\ltimes\rho_V\rtimes v, \rho_B\rtimes v)$$
given by the composition of $(\pi_{A\rtimes_\alpha^uG}, \pi_{V\rtimes_{\alpha}^uG}, \pi_{B\rtimes_\alpha^uG})$ (extended to multipliers 
as in  Proposition \ref{prop-genmorextend}) with the universal representation $(i_A, i_V, i_G, i_B, i_G^A, i_G^B)$.
%
%
%
%
\end{proposition}
\begin{proof} Starting with $(\rho_A,\rho_V,\rho_B, u,v)$ we obtain a corresponding covariant homomorphism 
$(\rho_{C_u^*(A,V,B)}, u\oplus v)$ of $(C_u^*(A,V,B), G, \alpha^u)$ into $M(C_u^*(C,W,D))$ 
(use Propositions \ref{prop-multuniv} and \ref{prop-actions}). By the universal property of 
the maximal crossed product $C_u^*(A,V,B)\rtimes_{\alpha,u}G$ we obtain the integrated form
$$\rho_{C_u^*(A,V,B)}\rtimes  u\oplus v: C_u^*(A,V,B)\rtimes_{\alpha,u}G\to M(C_u^*(C,W,D))$$
whose restriction to $(C_c(G,A), C_c(G,V), C_c(G,B))$ coincides with the integrated form 
$(\rho_A\rtimes u, u\ltimes\rho_V\rtimes v, \rho_B\rtimes v)$ with values in 
$(M(C), M(W), M(D))\subseteq M(C_u^*(C,W,D))$. They therefore extend to 
$(A\rtimes_\alpha^uG, V\rtimes_\alpha^uG, B\rtimes_\alpha^uG)$ as desired. 

For the converse we need to show  that the integrated form of the covariant morphism 
$(\rho_A,\rho_V\rho_B,u,v)$ obtained by composing
 $(\pi_{A\rtimes_\alpha^uG}, \pi_{V\rtimes_{\alpha}^uG}, \pi_{B\rtimes_\alpha^uG})$ with 
 $(i_A, i_V, i_G, i_B, i_G^A, i_G^B)$ agrees with $(\pi_{A\rtimes_\alpha^uG}, \pi_{V\rtimes_{\alpha}^uG}, \pi_{B\rtimes_\alpha^uG})$
on  \linebreak
$\big(C_c(G,A), C_c(G,V), C_c(G,B)\big)$. But this follows from a straightforward computation which we omit.
\end{proof}

Recall from Proposition \ref{prop-C*-hull-bimodule} that for a $C^*$-operator bimodule $(A,V,B)$ we have the 
identities 
$$C_u^*(A,V,B)\cong C_u^*(\Op(A,V,B))\cong C_u^*(X(A,V,B))$$
where the first isomorphism is given by the universal property of $C_u^*(A,V,B)$ applied 
to the canonical (corner) inclusions of $(A,V,B)$  into $\Op(A,V,B)\subseteq C_u^*(\Op(A,V,B))$ 
and the second isomorphism is given by the universal property of $C_u^*(\Op(A,V,B))$
applied  to the canonical inclusion of $\Op(A,V,B)$ into $X(A,V,B)$. 
If $\alpha:G\to \Aut(A,V,B)$ is an action, then these isomorphism are $G$-equivariant,
where $ C_u^*(\Op(A,V,B))$ is equipped with the action extending $\alpha^{\Op}$ 
and $C_u^*(X(A,V,B))$ is equipped with the action extending $\alpha^X$,
where $\alpha^{\Op}$ and $\alpha^X$ are as in Proposition \ref{prop-actions}.

In \cite{KR} Katsoulis and Ramsay defined the universal crossed product 
of the operator algebra system $(\mathcal A,G,\alpha)$ as the closure of $C_c(G,\mathcal A)$ 
inside $C_u^*(\mathcal A)\rtimes_{\alpha,u}G$ and in Definition \ref{def-crossed-product} we defined the crossed product 
$X\rtimes_\alpha^u G$ for an action $\alpha$ on a $C^*$-operator system $X$ 
as the closure $X\rtimes_{\alpha}^uG$ of  $C_c(G,X)$ 
inside $C_u^*(X)\rtimes_{\alpha,u}G$ (surpressing the $C^*$-part of the $C^*$-operator system  in our notation). 
Thus, identifying $(C_c(G,A), C_c(G,V), C_c(G,B))$ with the three non-zero corners of 
$C_c(G, \Op(A,V,B))$ and the latter as a subspace of $C_c(G, X(A,V,B))$ we see that 
these inclusions extend to completely isometric inclusions
\begin{align*}
&\Op(A\rtimes_{\alpha}^uG, V\rtimes_\alpha^uG, B\rtimes_\alpha^uG)=\Op(A,V,B)\rtimes_\alpha^uG\\
&\subseteq  
X(A,V,B)\rtimes_{\alpha}^uG=X(A\rtimes_{\alpha}^uG, V\rtimes_\alpha^uG, B\rtimes_\alpha^uG)
\end{align*}
Together with Corollary \ref{cor-universal} and  Proposition \ref{prop-C*-hull-bimodule} we obtain   isomorphisms
\begin{align*}
&C_u^*(A,V, B)\rtimes_{\alpha,u}G \cong C_u^*(\Op(A,V,B))\rtimes_{\alpha,u}G \cong C_u^*(X(A,V,B))\rtimes_{\alpha,u}G\\
&\stackrel{\text{Corollary \ref{cor-universal}}}{\cong} C_u^*\big(X(A,V,B)\rtimes_{\alpha}^uG\big)\cong 
C_u^*\big(X(A\rtimes_{\alpha}^uG, V\rtimes_\alpha^uG, B\rtimes_\alpha^uG)\big)\\
&\quad\quad \cong C_u^*\big(A\rtimes_{\alpha}^uG, V\rtimes_\alpha^uG, B\rtimes_\alpha^uG\big)\\
&\quad\quad  \cong C_u^*\big(\Op(A\rtimes_{\alpha}^uG, V\rtimes_\alpha^uG, B\rtimes_\alpha^uG)\big)
\end{align*}
In particular we see that the theories for universal crossed product by corresponding actions of $G$ on 
$(A,V,B)$, $\Op(A,V,B)$ and  $X(A,V,B)$ are completely equivalent!

We close this section with a brief discussion of the reduced crossed product for an action 
$\alpha:G\to \Aut(A,V,B)$. The easiest way to do this at this point
is to form the reduced crossed product 
$$X(A,V,B)\rtimes_{\alpha}^rG \subseteq M\big(X(A,V,B)\otimes \mathcal K(L^2(G))\big)$$
as the image of the regular representation 
$\Lambda_{X(A,V,B)}: X(A,V,B)\rtimes_{\alpha}^u G\to M\big(X(A,V,B)\otimes \mathcal K(L^2(G))\big)$
as in Definition \ref{defn educed crossed product} and define 
$(A,V,B)\rtimes_\alpha^rG=\big(A\rtimes_{\alpha}^rG, V\rtimes_{\alpha}^rG, B\rtimes_{\alpha}^rG\big)$ 
via the images of the corners 
$(A\rtimes_{\alpha}^uG, V\rtimes_{\alpha}^uG, B\rtimes_{\alpha}^uG)$ inside $X(A,V,B)\rtimes_{\alpha}^rG$
(which is the same as taking closures of $(C_c(G,A), C_c(G,V), C_c(G,B))$ inside $X(A,V,B)\rtimes_{\alpha}^rG$).
We leave it as an exercise to the reader to formulate this in terms of a regular covariant representation 
of $\big((A,V,B),G,\alpha\big)$  into $\big(M(A\otimes \K), M(V\otimes \K), M(B\otimes \K)\big)$ for $\K=\K(L^2(G))$ and 
where, as usual,  ``$\otimes$'' denotes the spacial tensor product.

It follows from our construction and part (b) of Remark \ref{rem-reduced}  that 
$\big(A\rtimes_{\alpha}^rG, V\rtimes_{\alpha}^rG, B\rtimes_{\alpha}^rG\big)$ is completely isometrically 
isomorphic to the closures of $(C_c(G,  A), C_c(G,V), C_c(G,B))$ inside 
$C\rtimes_{\alpha^C,r}G$ for any $C^*$-hull  \linebreak $\big(C, (j_A, j_V, j_B)\big)$ of $(A,V,B)$ (where we identify
$(A,V,B)$ with the triple $(j_A(A), j_V(V), j_B(B))$ inside $C$) which carries an 
action $\alpha^C$ which is  compatible with the given action on $(A,V,B)$. In particular, we may 
take the closures inside $C_u^*(A,V,B)\rtimes_{\alpha^u,r}G$ or $C_e^*(A,V,B)\rtimes_{\alpha^e,r}G$.
From this we get

\begin{proposition}\label{prop-max=red}
Let $\alpha:G\to \Aut(A,V,B)$ be an action by an {\em amenable} group $G$. Then 
$$(A\rtimes_{\alpha}^uG, V\rtimes_{\alpha}^uG, B\rtimes_{\alpha}^uG)= \big(A\rtimes_{\alpha}^rG, V\rtimes_{\alpha}^rG, B\rtimes_{\alpha}^rG\big)$$
via the regular representation.
\end{proposition}
\begin{proof} This follows from the  above discussion and the fact that 
$$C_u^*(A,V,B)\rtimes_{\alpha^u,u}G\cong C_u^*(A,V,B)\rtimes_{\alpha^u,r}G$$ if $G$ is amenable.
\end{proof}

\section{Coactions and duality}
In this section we want to discuss  the duality theorems  for crossed products for $C^*$-operator bimodules.
The theory is  more or  less a direct consequence of the theory for $C^*$-operator systems vie the 
functor $(A,V,B)\mapsto X(A,V,B)$, so we'll try to be brief. Note that if $(A,V,B)$ is a $C^*$-operator system 
represented completely isometrically on the pair of Hilbert spaces $(H,K)$ and if $C$ is any $C^*$-algebra which 
is   represented faithfully on a Hilbert space $L$, then we can define the spatial tensor product 
$(A\otimes C, V\otimes C, B\otimes C)$ as the closures of the canonical inclusions of the algebraic tensor products 
$\big(A\odot C, V\odot C, B\odot C\big)$ 
inside
$\big(\mathcal B(H\otimes L), \mathcal B(K\otimes L, H\otimes L), \mathcal B(K\otimes L)\big)$.
One then checks that
$$X(A,V,B)\otimes C=X(A\otimes C, V\otimes C, B\otimes C)$$
(and, similarly $\Op(A,V,B)\otimes C=\Op(A\otimes C, V\otimes C, B\otimes C)$).
In what follows we often write $(A,V,B)\otimes C$ for the $C^*$-operator bimodule $(A\otimes C, V\otimes C, B\otimes C)$
and we write $M\big((A,V,B)\otimes C\big)$ for the multiplier bimodule 
$\big(M(A\otimes C), M(V\otimes  C), M(B\otimes C))\big)$.

\begin{definition}\label{def-coactionC*-bimodule}
Let $(A,V,B)$ be a $C^*$-operator bimodule.  A  {\em coaction} of the locally compact group $G$ 
on $(A,V,B)$ is a generalized morphism
$$\delta_{(A,V,B)}=(\delta_A,\delta_V, \delta_B): (A,V,B)\to M\big((A,V,B)\check\otimes C^*(G)\big)$$
such that the following hold:
\begin{enumerate}
\item the maps $\delta_A:A\to M(A\check\otimes C^*(G))$ and $\delta_B:B\to M(B\check\otimes C^*(G))$ are coactions of 
$G$ on the $C^*$-algebras $A$ and $B$, respectively.
\item The following diagram of generalized morphism commutes:
$$\begin{CD}
(A,V,B)@>\delta_{(A,V,B)} >> M\big((A,V,B)\check\otimes C^*(G)\big)\\
@V\delta_{(A,V,B)} VV   @VV \id_{(A,V,B)}\otimes \delta_GV\\
M\big((A,V,B)\check\otimes C^*(G)\big)  @>> \delta_{(A,V,B)}\otimes \id_{G}> M\big((A,V,B)\check\otimes C^*(G)\check\otimes C^*(G)\big)
\end{CD}
$$
\end{enumerate}
\end{definition}

Using the correspondence of generalized morphisms from $(A,V,B)$ to $M\big((A,V,B)\check\otimes C^*(G)\big)$ 
with generalized morphism from $X(A,V,B)$ to 
\linebreak
$X(M(A,V,B))$ of Corollary \ref{cor-morphisms}  
and the isomorphism 
$X(M(A,V,B))\cong M(X(A,V,B))$ of  Proposition \ref{prop-multiplier} we see that every coaction $\delta_{(A,V,B)}$ 
of $G$ on $(A,V,B)$ as in the definition above determines a coaction $\delta_{X(A,V,B)}$  of $G$  on $X(A,V,B)$
and vice versa. 
We then  call $\delta_{(A,V,B)}$ {\em non-degenerate} iff $\delta_{X(A,V,B)}$ is non-degenerate in the sense of 
Definition \ref{def-nondeg}. 

\begin{example}\label{ex-dual-coaction-bimodule}
Recall that for  each  action $\alpha:G\to \Aut(A,V,B)$ there corresponds a unique action (which here we also denote by $\alpha$)
of $G$ on $X(A,V,B)$ such that $X(A,V,B)\rtimes_\alpha^uG=X\big((A,V,B)\rtimes_\alpha^uG\big)$ (and similarly for the reduced crossed products).
Recall from Example \ref{ex dual} that there exist  canonical dual coactions 
$\widehat\alpha_u$ and $\widehat\alpha_r$ of $G$ on
$X(A,V,B)\rtimes_\alpha^uG$ and $X(A,V,B)\rtimes_\alpha^rG$, respectively. Identifying 
$X(A,V,B)\rtimes_\alpha^uG$ with $X\big( (A,V,B)\rtimes_\alpha^uG\big)$ 
and using the correspondence between coactions on  \linebreak $(A,V,B)\rtimes_\alpha^uG$ and  coactions on 
$X\big( (A,V,B)\rtimes_\alpha^uG\big)$ (and similarly for the reduced crossed products), we obtain 
dual coactions $\alpha_u$ and $\widehat\alpha_r$ on  the full and reduced crossed  products 
of $(A,V,B)$ by $G$, respectively. We leave it to the reader to spell out  direct formulas for these coactions. 
\end{example}

\begin{definition}\label{def-coact-covariant-bimodule}
Let $\delta_{(A,V,B)}$ be a coaction of $G$ on the $C^*$-operator bimodule $(A,V,B)$. Then a 
(generalized)
{\em covariant morphism} of the co-system $\big((A,V,B), G, \delta_{(A,V,B)}\big)$ into 
the muliplier bimodule $M(C,W,D)=\big(M(C), M(W), M(B)\big)$ of a $C^*$-operator  bimodule $(C,W,D)$
consists of a quintuple $\big(\rho_A, \rho_V, \rho_B, \mu, \nu\big)$ such that
\begin{enumerate}
\item $\rho=(\rho_A, \rho_V, \rho_B):(A,V,B)\to \big(M(C), M(W), M(B)\big)$ is a generalized morphism of $(A,V,B)$;
\item $\mu:C_0(G)\to M(C), \nu:C_0(G)\to M(D)$ are non-degenerate $*$-homomorphisms;
\item $(\rho_A, \mu)$  and  $(\rho_B,\nu)$ are covariant for $(A,G,\delta_A)$ and $(B, G, \delta_B)$, respectively; and
\item $(\rho_V\otimes \id_G)\circ \delta_V(v)=\big(\mu\otimes \id_G(w_G)\big)(\rho_V(v)\otimes 1)\big(\nu\otimes\id_G(w_G)\big)^*$
for all $v\in V$.
\end{enumerate}
where $w_G\in C^b_{st}(G,M(C^*(G))\cong M(C_0(G)\check\otimes  C^*(G))$ is  the strictly continuous 
function $w_G(g)= i_G(g)$.
A {\em covariant representation} of $\big((A,V,B), G, \delta_{(A,V,B)}\big)$ on the pair of Hilbert spaces $(H,K)$ 
is a morphism into $\big(\mathcal B(H), \mathcal B(K,H),\mathcal B(K)\big)$.
\end{definition}

It is now an easy exercise to see that there is a one-to-one correspondence between covariant morphism of 
$(A,V,B)$ into $M(C,V,D)$ and covariant morphisms 
of $\big(X(A,V,B), G, \delta_{X(A,V,B)}\big)$ into $M(X(C,V,D))$ given by assigning to 
$\big(\rho_A, \rho_V, \rho_B, \mu, \nu\big)$ the covariant pair $\big(\rho_{X(A,V,B)}, \mu\oplus \nu\big)$ with 
$$\rho_{X(A,V,B)}=\left(\begin{matrix} \rho_A &\rho_V\\ \rho_V^*& \rho_B\end{matrix}\right)$$
as in Corollary \ref{cor-morphisms}.

Moreover, using the identity $C_u^*(A,V,B)\cong C_u^*(X(A,V,B))$ and Proposition \ref{prop coaction}, we deduce 
easily that there  is a one-to-one correspondence between coactions $\delta_{(A,V,B)}$ of $G$ on $(A,V,B)$ and 
coactions $\delta_u$ of $G$ on $C_u^*(A,V,B)$ which satisfy the conditions
$$\delta_u(A)\subseteq M(A\check\otimes C^*(G)), \; \delta_u(V)\subseteq M(V\check\otimes C^*(G)), \;\text{and}\;
\delta_u(B)\subseteq M(B\check\otimes C^*(G)),$$
where we understand these inclusions with respect to the canonical inclusions of $(A,V,B)$ into $C_u^*(A,V,B)$ and of 
$M((A,V,B)\check\otimes C^*(G))$ into \linebreak
$M(C_u^*(A,V,B)\check\otimes C^*(G))$ which can  be deduced 
from Lemma \ref{lem mult univ}).

\begin{example}\label{ex-coact-bimodule-regular}
The {\em regular representation} of the co-system \linebreak 
$\big((A,V,B), G, \delta_{(A,V,B)}\big)$  is the 
covariant morphism from $\big((A,V,B), G, \delta_{(A,V,B)}\big)$  into $M\big( (A,V,B)\otimes \K(L^2(G))\big)$ 
defined as the quintuple
\begin{align*}
\big(\Lambda_A, \Lambda_V, &\Lambda_B, \Lambda_{\widehat{G}}^A, \Lambda_{\widehat{G}}^B\big)\\
&=\big((\id_A\otimes\lambda)\circ\delta_A,  (\id_A\otimes\lambda)\circ\delta_A, (\id_A\otimes\lambda)\circ\delta_A, 1_A\otimes M, 1_B\otimes M\big)
\end{align*}
where $\lambda=\lambda_G$ denotes the regular representation of $G$ on $L^2(G)$ and $M:C_0(G)\to \mathcal B(L^2(G))$ is the representation by multiplication operators.
One easily checks that this representation corresponds to the regular representation of the 
co-system $\big( X(A,V,B), G, \delta_{X(A,V,B)}\big)$ via  the above described correspondence. 
In particular, it is a  covariant morphism of $\big((A,V,B), G, \delta_{(A,V,B)}\big)$.
\end{example}

We are now ready to define the crossed products

\begin{proposition}\label{def-coact-crossed-bimodule}
Let $\delta_{(A,V,B)}$ be a coaction of $G$ on the $C^*$-operator bimodule $(A,V,B)$. 
We then define the crossed product 
$$(A,V,B)\rtimes_{\delta(A,V,B)}\widehat{G}=\big(A\rtimes_{\delta_A}\widehat{G}, V\rtimes_{\delta_V}\widehat{G}, B\rtimes_{\delta_B}\widehat{G}\big)$$
as 
$$A\rtimes_{\delta_A}\widehat{G}:=\cspn\{\Lambda_A(A)\Lambda_{\widehat{G}}^A(C_0(G))\},\quad B\rtimes_{\delta_B}\widehat{G}:=\cspn\{\Lambda_B(B)\Lambda_{\widehat{G}}^B(C_0(G))\},
$$
$$\text{and} \quad V\rtimes_{\delta_V}\widehat{G}=\cspn\{\Lambda_V(V)\Lambda_{\widehat{G}}^A(C_0(G))\}=\cspn\{\Lambda_{\widehat{G}}^B(C_0(G)\Lambda_V(V)\}$$
inside $M\big((A,V,B)\otimes \K(L^2(G))\big)$.
\end{proposition}

Note that it follows directly from the definitions that $(\Lambda_A, \Lambda_{\widehat{G}}^A)$ and $(\Lambda_B, \Lambda_{\widehat{G}}^B)$
are the regular representations of $(A,G, \delta_A)$ and $(B,G,\delta_B)$, respectively. 
We therefore  see that the $C^*$-algebras $A\rtimes_{\delta_A}\widehat{G}$ and $B\rtimes_{\delta_B}\widehat{G}$ coincide 
with the crossed  products of  these $C^*$-co-systems as described in  Section \ref{sec-coaction}. 
Using the above described correspondence between covariant morphisms of  $\big((A,V,B), G, \delta_{(A,V,B)}\big)$  and
covariant morphisms of $\big(X(A,V,B), G,\delta_{X(A,V,B)}\big)$ we now get from Proposition \ref{prop-dualuniv}:

\begin{theorem}\label{thm-crossed-coact-bimodule}
The crossed  product $(A,V,B)\rtimes_{\delta(A,V,B)}\widehat{G}$ is a well defined $C^*$-operator  bimodule 
such that 
$$X\big((A,V,B)\rtimes_{\delta(A,V,B)}\widehat{G}\big)=X(A,V,B)\rtimes_{\delta_{X(A,V,B)}}\widehat{G}.$$
and 
$$C_u^*\big((A,V,B)\rtimes_{\delta(A,V,B)}\widehat{G}\big)=C_u^*(A,V,B)\rtimes_{\delta_u}\widehat{G}.$$
Moreover, the pair
$$\Big((A,V,B)\rtimes_{\delta(A,V,B)}\widehat{G}, \big(\Lambda_A, \Lambda_V, \Lambda_B, \Lambda_{\widehat{G}}^A, \Lambda_{\widehat{G}}^B\big)
\Big)$$ satisfies the following universal property   for covariant morphisms:

If $\big(\rho_A, \rho_V, \rho_B, \mu, \nu\big)$ is any covariant morphism of $\big((A,V,B), G, \delta_{(A,V,B)}\big)$ into
$M(C,W,D)$ then there  exists a unique covariant morphism
$$(\rho_A\rtimes \mu, \mu\ltimes \rho_V\rtimes \nu, \rho_B\rtimes \nu): (A,V,B)\rtimes_{\delta(A,V,B)}\widehat{G}\to M(C,W,D)$$
such that 
$$\rho_A=(\rho_A\rtimes \mu)\circ \Lambda_A,\; \rho_V=(\mu\ltimes \rho_V\rtimes \nu)\circ \Lambda_V, \; \rho_B=(\rho_B\rtimes \nu)\circ \Lambda_B,
$$
$$ \mu=(\rho_A\rtimes \mu)\circ \Lambda_{\widehat{G}}^A\quad\text{and}\quad \nu=(\rho_B\rtimes \nu)\circ \Lambda_{\widehat{G}}^B.$$
Conversely, if $\Phi=\big(\Phi_{A\rtimes_{\delta_A}\widehat{G}}, \Phi_{V\rtimes_{\delta_V}\widehat{G}}, \Phi_{B\rtimes_{\delta_B}\widehat{G}}\big)$
is any generalized morphism from $(A,V,B)\rtimes_{\delta(A,V,B)}\widehat{G}$ into $M(C,W,D)$ then there exists a unique 
covariant morphism $\big(\rho_A, \rho_V, \rho_B, \mu, \nu\big)$ of $\big((A,V,B), G, \delta_{(A,V,B)}\big)$ such that 
$$\Phi=(\rho_A\rtimes \mu, \mu\ltimes \rho_V\rtimes \nu, \rho_B\rtimes \nu)$$
\end{theorem}

\begin{remark}\label{rem-dual action bimodule}
Let $\sigma:G\to \Aut(C_0(G))$ denote  action given by right translation. Then there is a dual action $\widehat\delta:G\to \Aut\big((A,V,B)\rtimes_{\delta}\widehat{G}\big)$ such that for each $s\in G$ the automorphism $\widehat\delta_s$ is given by the integrated form of the covariant morphism
$\big(\Lambda_A, \Lambda_V, \Lambda_B, \Lambda_{\widehat{G}}^A\circ \sigma(s), \Lambda_{\widehat{G}}^B\circ \sigma(s)\big)
\Big)$. Of course, it corresponds to the  dual action on $X(A,V,B)\rtimes_{\delta_X}\widehat{G}$.
\end{remark}

We now come to the duality theorems. We start with the $C^*$-operator bimodule version of the Imai-Takai duality theorem.
Using the version of the  Imai-Takai theorem for actions on $C^*$-operator systems, Theorem \ref{thm-ImaiTakai}, we
now get

\begin{theorem}\label{thm-Imai-Takai-bimodule}
Let $\alpha:G\to \Aut(A,V,B)$ be an action. Then there exist  canonical dual coactions 
$\widehat\alpha_u$ (resp. $\widehat\alpha_r$) of $G$ on the universal and reduced crossed products
$(A,V,B)\rtimes_{\alpha}^uG$ and  $(A,V,B)\rtimes_\alpha^rG$, respectively, such that
$$(A,V,B)\rtimes_{\alpha}^uG\rtimes_{\widehat\alpha_u}\widehat{G}\cong (A,V,B)\otimes \K(L^2(G))$$
and 
$$(A,V,B)\rtimes_{\alpha}^rG\rtimes_{\widehat\alpha_r}\widehat{G}\cong (A,V,B)\otimes \K(L^2(G)),$$
and  the isomorphism transforms the  double dual actions $\widehat{\widehat\alpha_u}$ and \
$\widehat{\widehat\alpha_r}$ to the action $\alpha\otimes \Ad\rho$ on $(A,V,B)\otimes \K(L^2(G))$,
where $\rho:G\to U(L^2(G))$ denotes the  right regular representation of $G$.
\end{theorem}

Dually, as an application of Theorem \ref{thm-Katayama} 
we get the following version of Katayama's duality for coactions on $C^*$-operator bimodules:

\begin{theorem}\label{thm-Katayama-bimodule}
Let $\delta=\delta_{(A,V,B)}$ be a coaction of $G$ on the $C^*$-operator bimodule $(A,V,B)$. 
 Then there exist  is a surjective morphism
 $$\Theta: (A,V,B)\rtimes_\delta\widehat{G}\rtimes_{\widehat{\delta}}^uG\onto (A,V,B)\otimes\K(L^2(G))$$
 which factors through an isomorphism 
  $$(A,V,B)\rtimes_\delta\widehat{G}\rtimes_{\widehat{\delta}}^\mu G\cong (A,V,B)\otimes\K(L^2(G)),$$
  where $(A,V,B)\rtimes_\delta\widehat{G}\rtimes_{\widehat{\delta}}^\mu G$ is a completion 
  of $C_c\big(G, (A,V,B)\rtimes_\delta\widehat{G}\big)$ with respect to a  norm which lies between the 
 universal and reduced crossed product norms. If $G$ is amenable, then $\Theta$ is an isomorphism.
\end{theorem}

\bibliographystyle{amsplain}


\end{document}